\documentclass[letterpaper,10pt]{article}

\usepackage[margin=0.5in]{geometry}

\usepackage{fullpage}
\usepackage{enumerate}
\usepackage{authblk}
\usepackage[colorinlistoftodos]{todonotes}
\usepackage{verbatim}%comments out anything between \begin{comment} and \end{comment}
\usepackage{section, amsthm, textcase, setspace, amssymb, lineno, amsmath, amssymb, amsfonts, latexsym, fancyhdr, longtable, ulem, mathtools}
\usepackage{epsfig, graphicx, pstricks,pst-grad,pst-text,tikz, tkz-berge,colortbl}
\usepackage{epsf}
\usepackage{graphicx, color}
\usepackage{float}
\usepackage[rflt]{floatflt}
\usepackage{amsfonts}
\usepackage{mathrsfs}
\usepackage{latexsym,enumitem}
\usetikzlibrary{fit,matrix,positioning}
\usetikzlibrary{decorations.markings}
\usepackage{pdflscape}
\usetikzlibrary{decorations.pathreplacing}
\usetikzlibrary{positioning}

\definecolor{darkgreen}{rgb}{0, 0.5, 0}

\newtheorem{theorem}{Theorem}
\newtheorem{lemma}[theorem]{Lemma}
\newtheorem{corollary}[theorem]{Corollary}
\newtheorem{conj}[theorem]{Conjecture}

\newtheorem{Ex}[theorem]{Example}

\newtheorem*{theorem*}{Theorem}
\newtheorem{remark}[theorem]{Remark}
\newtheorem{que}[theorem]{Question}

\newcommand{\ind}{{\rm ind \hspace{.1cm}}}

\title{Seaweed algebras and the unimodal spectrum property}

\author[*]{Nicholas Mayers}
\author[**]{Nicholas Russoniello}

\affil[*]{Department of Mathematics, North Carolina State University, Raleigh, NC 27695}
\affil[**]{Department of Mathematics, College of William \& Mary, Williamsburg, VA 23185}

\begin{document}
%\linenumbers

\maketitle

\begin{abstract}
\noindent
If $\mathfrak{g}$ is a Frobenius Lie algebra, then the spectrum of $\mathfrak{g}$ is an algebraic invariant equal to the multiset of eigenvalues corresponding to a particular operator acting on $\mathfrak{g}$. In the case of Frobenius seaweed subalgebras of $A_{n-1}=\mathfrak{sl}(n)$, or type-A seaweeds for short, it has been shown that the spectrum can be computed combinatorially using an attendant graph. With the aid of such graphs, it was further shown that the spectrum of a type-A seaweed consists of an unbroken sequence of integers centered at $\frac{1}{2}$. It has been conjectured that if the eigenvalues are arranged in increasing order, then the sequence of multiplicities forms a unimodal sequence about $\frac{1}{2}$. Here, we establish this conjecture for certain families of Frobenius type-A seaweeds by finding explicit formulas for their spectra; in fact, for some families we are able to show that the corresponding sequences of multiplicities form log-concave sequences. All arguments are combinatorial.
\end{abstract}


\section{Introduction}

A biparabolic (seaweed) subalgebra of a complex reductive Lie algebra $\mathfrak{r}$ is the intersection of two parabolic subalgebras whose sum is $\mathfrak{r}$. They -- along with certain associated planar graphs called  ``meanders" -- were first introduced in the case $\mathfrak{r}=\mathfrak{gl}(n)$ by Dergachev and Kirillov (\textbf{\cite{DK}}, 2000). One of the main results of \textbf{\cite{DK}} is that the algebra's ``index," a notoriously difficult-to-compute Lie-algebraic invariant of recent interest (see \textbf{\cite{seriesA,BCD,Panov,Pancent}}), can be found by counting the number of paths and cycles in its associated meander. Of particular significance are those seaweed subalgebras of $\mathfrak{gl}(n)$ whose meanders consist of a single path and no cycles. For such algebras, imposing a vanishing trace condition results in a seaweed subalgebra of $\mathfrak{sl}(n)$ with index zero. In general, algebras with index zero are called ``Frobenius" and have been studied extensively in the context of invariant theory (see \textbf{\cite{Ooms1}}) and are connected to the classical Yang-Baxter equation (see \textbf{\cite{GG1}} and \textbf{\cite{GG2}}). Here, we are concerned with the ``spectrum" of a Frobenius Lie algebra, which is an invariant multiset of eigenvalues arising from a certain operator's action on the algebra. In the case of a Frobenius seaweed subalgebra of $\mathfrak{sl}(n)$, it is conjectured that the multiplicities of eigenvalues in the spectrum form a unimodal sequence. We establish this conjecture for particular families of Frobenius seaweed subalgebras of $\mathfrak{sl}(n)$ by utilizing the seaweeds' meanders to find explicit formulas for their spectra.
\bigskip

To fix the notation, let $\mathfrak{g}$ be a Lie algebra over $\mathbb{C}$. From any linear form $F\in\mathfrak{g}^*$ arises a skew-symmetric, bilinear two-form $B_F(-,-)=F([-,-])$, called the \textit{Kirillov form}. The index of $\mathfrak{g}$ is then given by $$\ind\mathfrak{g}=\min_{F\in\mathfrak{g}^*}\dim\ker(B_F)$$
(see \textbf{\cite{Dix}}). The Lie algebra $\mathfrak{g}$ is called \textit{Frobenius} if its index is zero, or equivalently, if there exists a linear form $F\in\mathfrak{g}^*$ such that $B_F$ is non-degenerate. We call such an $F$ a \textit{Frobenius \textup(linear\textup) form}, and the natural map $\mathfrak{g}\to\mathfrak{g}^*$ defined by $x\mapsto F[x,-]$ is an isomorphism. The inverse image of $F$ under this map is called a \textit{principal element} of $\mathfrak{g}$ and will be denoted $\widehat{F}$; that is, $\widehat{F}$ is the unique element of $\mathfrak{g}$ such that $$F\circ\text{ad }\widehat{F}=F([\widehat{F},-])=F.$$ 

Ooms (\textbf{\cite{Ooms2}}, 1980) showed that the spectrum of the adjoint action of a principal element acting on its Frobenius Lie algebra is independent of the choice of principal element (see also \textbf{\cite{GG3}}, Theorem 3). Consequently, if $\mathfrak{g}$ is Frobenius and $\widehat{F}\in\mathfrak{g}$ is any principal element, then the multiset of eigenvalues of $ad_{\widehat{F}}~:~\mathfrak{g}\to\mathfrak{g}$ is an invariant of $\mathfrak{g}$, and so we call such a multiset the \textit{spectrum of} $\mathfrak{g}$. In the case that $\mathfrak{g}$ is a seaweed subalgebra of $\mathfrak{sl}(n),$ Gerstenhaber and Giaquinto (\textbf{\cite{GG3}}, 2009) asserted that the spectrum consists of an unbroken set of integers. Coll et al. (\textbf{\cite{unbroken}}, 2016) were able to prove  the unbrokenness property initially stated in \textbf{\cite{GG3}}, showing that the distinct eigenvalues in the spectrum of a Frobenius seaweed subalgebra of $\mathfrak{sl}(n)$ forms an interval of integers centered at $\frac{1}{2}$. Moreover, the authors of \textbf{\cite{unbroken}} conjectured that if the eigenvalues are arranged in increasing order, then the corresponding sequence of multiplicities form a unimodal sequence centered at $\frac{1}{2}$. The primary goal of this article is to establish this unimodality conjecture for certain families of seaweed subalgebras of $\mathfrak{sl}(n)$; we do so by finding explicit formulas for the spectra by combinatorial means.

The methods used in this paper make use of a family of graphs, called ``meanders," which can be associated to seaweed subalgebras of $\mathfrak{sl}(n)$ via the seaweed's defining compositions. See Section~\ref{sec:prelim} for the relationship between seaweed subalgebras of $\mathfrak{sl}(n)$ and compositions, as well as a detailed construction of a meander. The raison d'etre for the introduction of meanders into the theory of seaweed subalgebras was the ability to implement them in the computation of algebraic invariants, including index and spectrum. With this in mind, we are able to translate questions of spectrum to questions about graphs.

The outline of the paper is as follows. In Section~\ref{sec:prelim}, we outline the necessary preliminaries concerning seaweed subalgebras of $\mathfrak{sl}(n)$ and meanders. In Section~\ref{sec:mp}, we develop lemmas concerning the spectra of Frobenius seaweed subalgebras of $\mathfrak{sl}(n)$ associated with ordered pairs of compositions containing three parts in total; these lemmas are then used to find explicit formulas for the spectra of such algebras. In Section~\ref{sec:stable}, we consider the spectra of families of Frobenius seaweed subalgebras of $\mathfrak{sl}(n)$ parametrized by the number of occurrences of a fixed value in the defining compositions; for such families, we observe new behavior concerning the stability of the set of distinct eigenvalues.  Section~\ref{sec:epilogue} consists of a discussion of our results, a number of questions arising from our work, and directions for further research.

\section{Preliminaries}\label{sec:prelim}
Given a reductive Lie algebra $\mathfrak{r},$ a \textit{seaweed subalgebra of} $\mathfrak{r}$ is the intersection of two parabolic subalgebras $\mathfrak{p},\mathfrak{p}'\subset\mathfrak{r}$ satisfying $\mathfrak{p}+\mathfrak{p}'=\mathfrak{r}.$ While seaweed subalgebras are defined generally as above, when restricting to seaweed subalgebras of $\mathfrak{gl}(n,\mathbb{C})=\mathfrak{gl}(n)$ and $\mathfrak{sl}(n,\mathbb{C})=\mathfrak{sl}(n)$, an equivalent definition in terms of compositions is available. In particular, we define a \textit{seaweed subalgebra of} $\mathfrak{gl}(n)$, or simply a \textit{seaweed}, to be any matrix Lie algebra constructed from two compositions of $n$ as follows. Let $V$ be a complex $n$-dimensional vector space with basis $\{e_1,\dots,e_n\},$ let $\underline{a}=(a_1,\dots,a_m)$ and $\underline{b}=(b_1,\dots,b_t)$ be two compositions of $n,$ and consider the flags $$\mathscr{V}=\big\{\{0\}\subset V_1\subset \dots\subset V_{m-1}\subset V_m=V\big\}\text{  and  }\mathscr{W}=\big\{V=W_0\supset W_1\supset\dots\supset W_t=\{0\}\big\},$$ where $V_i=span\{e_1,\dots,e_{a_1+\dots +a_i}\}$ and $W_j=span\{e_{b_1+\dots+b_j+1},\dots,e_n\}.$ Then the seaweed $\mathfrak{p}(\underline{a}|\underline{b})=\mathfrak{p}\frac{a_1|\dots|a_m}{b_1|\dots|b_t}$ is the subalgebra of $\mathfrak{gl}(n)$ that preserves the flags $\mathscr{V}$ and $\mathscr{W}.$ Similarly, the seaweed subalgebra of $A_{n-1}=\mathfrak{sl}(n)$, or \textit{type-A seaweed}, $\mathfrak{p}^A(\underline{a}|\underline{b})=\mathfrak{p}^A\frac{a_1|\dots|a_m}{b_1|\dots|b_t}$ is the subalgebra of $\mathfrak{sl}(n)$ that preserves the flags $\mathscr{V}$ and $\mathscr{W}$. One special case is of note: if $\underline{a}=(n)$ and $\underline{b}=(b_1,b_2)$, then the type-A seaweeds $\mathfrak{p}^A(\underline{a}|\underline{b})$ and $\mathfrak{p}^A(\underline{b}|\underline{a})$ are called \textit{maximal parabolic}.

\begin{remark}
Typically, one includes $n$ in the notation for seaweed subalgebras of $\mathfrak{gl}(n)$ and $\mathfrak{sl}(n)$, writing $\mathfrak{p}_n(\underline{a}|\underline{b})=\mathfrak{p}_n\frac{a_1|\dots|a_m}{b_1|\dots|b_t}$ and $\mathfrak{p}^A_n(\underline{a}|\underline{b})=\mathfrak{p}^A_n\frac{a_1|\dots|a_m}{b_1|\dots|b_t}$ for the corresponding \textup(type-A\textup) seaweeds; however, since $n$ is encoded in the compositions $\underline{a}$ and $\underline{b},$ we omit it for ease of notation.
\end{remark}

%In this section, we formally define seaweed algebras in $\mathfrak{sl}(n)$ -- the set of all $n\times n$ matrices of trace zero. As we will see, such seaweed algebras are naturally defined in terms of two compositions of the positive integer $n$.  

%\begin{definition}
%If $V$ is an $n$-dimensional vector space with a basis 
%$\{e_1,\dots, e_n \}$, let $a_1+\cdots+a_m$ and $b_1+\cdots+b_l$ be two compositions of $n$ and consider the flags 

%$$
%\{0\} \subset V_1 \subset \cdots \subset V_{m-1} \subset V_m =V~~~\text{and}~~~ V=W_0\supset W_1\supset \cdots \supset W_t=\{0\}, 
%$$
%where $V_i=\text{span}\{e_1,\dots, e_{a_1+\cdots +a_i}\}$ and $W_j=\text{span}\{e_{b_1+\cdots +b_j+1},\dots, e_n\}$.  
%\end{definition}

%The subalgebra of $\mathfrak{sl}(n)$ preserving these flags is called a \textit{seaweed Lie algebra}, or simply \textit{seaweed}, and is denoted by the symbol 
%$\displaystyle \frac{a_1|\cdots|a_m}{b_1|\cdots|b_t}$, which we  refer to as the \textit{type} of the seaweed.  If $b_1=n$, the seaweed is called \textit{maximal parabolic}.

%\begin{remark}
%The preservation of flags in the definition above ensures that seaweeds are closed under matrix multiplication, \rn{thus defining} Lie algebras under the commutator bracket.
%\end{remark}

The evocative ``seaweed'' is descriptive of the shape of the algebra when exhibited in matrix form. For example, the seaweed algebra $\mathfrak{p}^A\frac{2|4}{1|2|3}$ consists of traceless matrices of the form depicted in Figure~\ref{fig:seaweed} (left), where $\ast$'s indicate potential non-zero entries.

\begin{figure}[H]
$$\begin{tikzpicture}[scale=0.75]
\draw (0,0) -- (0,6);
\draw (0,6) -- (6,6);
\draw (6,6) -- (6,0);
\draw (6,0) -- (0,0);
\draw [line width=3](0,6) -- (0,4);
\draw [line width=3](0,4) -- (2,4);
\draw [line width=3](2,4) -- (2,0);
\draw [line width=3](2,0) -- (6,0);

\draw [line width=3](0,6) -- (1,6);
\draw [line width=3](1,6) -- (1,5);
\draw [line width=3](1,5) -- (3,5);
\draw [line width=3](3,5) -- (3,3);
\draw [line width=3](3,3) -- (6,3);
\draw [line width=3](6,3) -- (6,0);

\draw [dotted] (0,6) -- (6,0);

\node at (.5,5.4) {{\LARGE *}};
\node at (.5,4.4) {{\LARGE *}};
\node at (1.5,4.4) {{\LARGE *}};
\node at (2.5,4.4) {{\LARGE *}};
\node at (2.5,3.4) {{\LARGE *}};
\node at (2.5,2.4) {{\LARGE *}};
\node at (2.5,1.4) {{\LARGE *}};
\node at (2.5,0.4) {{\LARGE *}};
\node at (3.5,2.4) {{\LARGE *}};
\node at (3.5,1.4) {{\LARGE *}};
\node at (3.5,0.4) {{\LARGE *}};
\node at (4.5,2.4) {{\LARGE *}};
\node at (4.5,1.4) {{\LARGE *}};
\node at (5.5,2.4) {{\LARGE *}};
\node at (5.5,1.4) {{\LARGE *}};
\node at (4.5,0.4) {{\LARGE *}};
\node at (5.5,0.4) {{\LARGE *}};

\node at (.5,6.4) {1};
\node at (2,5.4) {2};
\node at (4.5,3.4) {3};
\node at (-0.5,4.9) {2};
\node at (1.5,1.9) {4};

\end{tikzpicture}\hspace{1.5cm}\begin{tikzpicture}[scale=1.3]
	\def\Node{\node [circle,  fill, inner sep=2pt]}
	\node at (0,0) {};
     \Node[label=left:$v_1$] (1) at (0,1.8) {};
	\Node[label=left:$v_2$] (2) at (1,1.8) {};
	\Node[label=left:$v_3$] (3) at (2,1.8) {};
	\Node[label=left:$v_4$] (4) at (3,1.8) {};
	\Node[label=left:$v_5$] (5) at (4,1.8) {};
	\Node[label=left:$v_6$] (6) at (5,1.8) {};
	\draw (1) to[bend left=50] (2);
	\draw (3) to[bend left=50] (6);
	\draw (4) to[bend left=50] (5);
	\draw (2) to[bend right=50] (3);
	\draw (4) to[bend right=50] (6);
\end{tikzpicture}$$
\caption{The seaweed $\mathfrak{p}\frac{2|4}{1|2|3}$ (left) and its associated meander (right)}\label{fig:seaweed}
\end{figure}
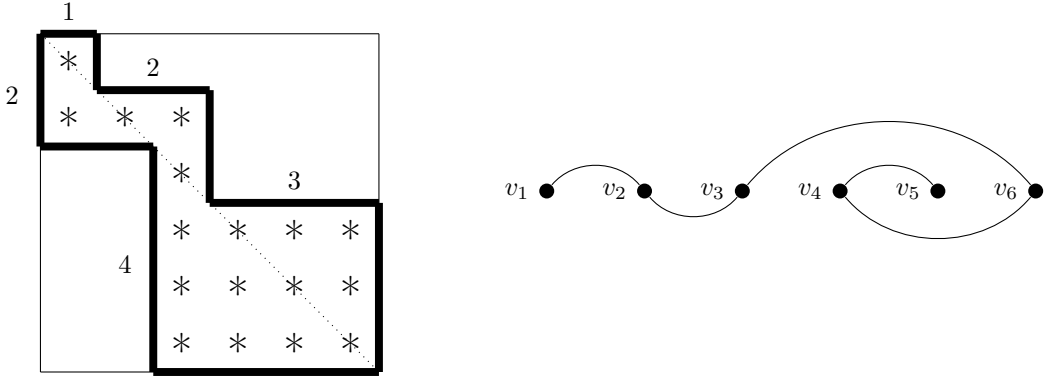

In \textbf{\cite{DK}}, Dergachev and Kirillov assign to each seaweed $\mathfrak{g}=\mathfrak{p}\frac{ a_1|\cdots|a_m}{b_1|\cdots|b_t}$ with $\sum_{i=1}^ma_i=\sum_{j=1}^tb_j=n$ a planar graph $M(\mathfrak{g})$ on $n$ vertices, called a \textit{meander}, as follows. Begin by placing $n$ labeled vertices $v_1, v_2, \dots , v_n$ from left to right along a horizontal line. Next, partition the vertices by grouping together the first $a_1$ vertices, then the next $a_2$ vertices, and so on, lastly grouping together the final $a_m$ vertices. We call each set of vertices formed a \textit{block}. For each block in the prescribed set partition, add a concave-down edge, called a \textit{top edge}, from the first vertex of the block to the last vertex of the block, then add a top edge between the second vertex of the block and the second-to-last vertex of the block, and so on within each block, assuming that the vertices being connected are distinct. In a similar way, partition the vertices according to the composition $\underline{b},$ and then place \textit{bottom edges}, i.e., concave-up edges, between vertices in each block. See Figure \ref{fig:seaweed} (right).

The main theorem of \textup{\textbf{\cite{DK}}} establishes the combinatorial formula for the index of a seaweed given in Theorem~\ref{rem:indexform} below.

\begin{theorem}\label{rem:indexform} If $\mathfrak{g}$ is a seaweed, then $$\ind \mathfrak{g} =2C + P,$$ where $C$ is the number of cycles and $P$ is the number of paths in $M(\mathfrak{g})$.
\end{theorem}

\noindent Here, a path is any acyclic connected component of a meander; that is, in addition to traditional path graphs, connected components consisting of a single vertex are considered paths. Utilizing Theorem~\ref{rem:indexform}, the corollary below follows immediately.

\begin{corollary}\label{cor:indA}
If $\mathfrak{g}$ is a type-A seaweed, then $$\ind\mathfrak{g}=2C+P-1,$$ where $C$ is the number of cycles and $P$ is the number of paths in $M(\mathfrak{g}).$
\end{corollary}

Interestingly, in (\textbf{\cite{GG2}}, 2008) and \textbf{\cite{GG3}}, Gerstenhaber and Giaquinto show, in particular, that one can also determine the spectrum of a Frobenius, type-A seaweed $\mathfrak{g}$ using $M(\mathfrak{g})$. The result is based on the construction of a principal element $\widehat{F}\in\mathfrak{g}$ for which each matrix unit $e_{i,j}\in\mathfrak{g}$ and each $e_{i,i}-e_{i+1,i+1}\in\mathfrak{g}$ is an eigenvector of $\text{ad }\widehat{F}$. To describe their choice of $\widehat{F}$, we first need to decorate $M(\mathfrak{g})$ by adding an orientation to its edges: top edges are oriented from right to left, while bottom edges are oriented from left to right. We refer to a meander with such an orientation as the \textit{oriented meander}, denoted $\overrightarrow{M}(\mathfrak{g}).$ Considering Corollary~\ref{cor:indA}, $M(\mathfrak{g})$ must consist of a single path. Thus, for all $1\leq i,j\leq n,$ there is a unique path from vertex $v_i$ to vertex $v_j$ in $M(\mathfrak{g})$, denoted $P_{i,j}(\mathfrak{g})$. Define the \textit{weight} $w(P_{i,j}(\mathfrak{g}))$ of the path $P_{i,j}(\mathfrak{g})$ in $M(\mathfrak{g})$ from $v_i$ to $v_j$ to be the number of forward edges minus the number of backward edges encountered when moving along $P_{i,j}(\mathfrak{g})$ from $v_i$ to $v_j$ in $\overrightarrow{M}(\mathfrak{g})$.

The following results appear in \textbf{\cite{unbroken}} as a consequence of a result from Gerstenhaber and Giaquinto \textbf{\cite{GG2}}; however, we include proofs for completeness.

\begin{lemma}\label{lem:specmeander}
Let $\mathfrak{g}\subset\mathfrak{gl}(n)$ be a seaweed with oriented meander $\overrightarrow{M}(\mathfrak{g})=(V,E)$ consisting of one path and no cycles, where $V=\{v_1,\dots,v_n\}$ is the vertex set and $E$ is the edge set of $\overrightarrow{M}(\mathfrak{g}).$ If $$F=\sum_{(v_i,v_j)\in E}e_{i,j}^*\in\mathfrak{g}^*$$ and $$\widehat{F}=\sum_{i=1}^n w(P_{i,n}(\mathfrak{g}))e_{i,i}\in\mathfrak{g},$$ then $F\left(\left[\widehat{F},x\right]\right)=F(x),$ for all $x\in\mathfrak{g}.$
\end{lemma}
\begin{proof}
If $P_{i,j}(\mathfrak{g})$ is the unique path from $v_i$ to $v_j$ in $\overrightarrow{M}(\mathfrak{g}),$ then it immediately follows that $w(P_{i,j}(\mathfrak{g}))=-w(P_{j,i}(\mathfrak{g}))$ and $w(P_{i,j}(\mathfrak{g}))+w(P_{j,k}(\mathfrak{g}))=w(P_{i,k}(\mathfrak{g})),$ for all $1\leq i,j,k\leq n.$ Now, if $\widehat{F}=\sum_{i=1}^nw(P_{i,n}(\mathfrak{g}))e_{i,i},$ then we have that $$\left[\widehat{F},e_{i,j}\right]=\big(w(P_{i,n}(\mathfrak{g}))-w(P_{j,n}(\mathfrak{g}))\big)e_{i,j}=\big(w(P_{i,n}(\mathfrak{g}))+w(P_{n,j}(\mathfrak{g}))\big)e_{i,j},$$ for all $1\leq i,j\leq n$. Note that if $i=j,$ then $F\left(\left[\widehat{F},e_{i,j}\right]\right)=\left[\widehat{F},e_{i,j}\right]=0.$ On the other hand, if $i\neq j,$ then $$\left[\widehat{F},e_{i,j}\right]=\big(w(P_{i,n}(\mathfrak{g}))+w(P_{n,j}(\mathfrak{g}))\big)e_{i,j}=\big(w(P_{i,j}(\mathfrak{g}))\big)e_{i,j}.$$ Therefore, since $w(P_{i,j}(\mathfrak{g}))=1$ for all $(v_i,v_j)\in E,$ we have that $$F\left(\left[\widehat{F},e_{i,j}\right]\right)=\begin{cases}
1, & \text{ if } (v_i,v_j)\in E;\\
0, & \text{ otherwise}.
\end{cases}$$ The result follows from the linearity of $F$ and the fact that the set of all $e_{i,j}$ such that $(i,j)$ is a potentially nonzero entry in the matrix form of $\mathfrak{g}$ is a basis for $\mathfrak{g}.$ 
\end{proof}

\begin{theorem}\label{thm:specmeander}
Let $\mathfrak{g}\subset\mathfrak{sl}(n)$ be a Frobenius, type-A seaweed with oriented meander $\overrightarrow{M}(\mathfrak{g})=(V,E),$ where $V=\{v_1,\dots,v_n\}$ is the vertex set and $E$ is the edge set of $\overrightarrow{M}(\mathfrak{g}).$ If $$F=\sum_{(v_i,v_j)\in E}e_{i,j}^*\in\mathfrak{g}^*,$$ is a Frobenius form on $\mathfrak{g},$ then the principal element corresponding to $F$ is given by $$\widehat{F}=\sum_{i=1}^n\left(w(P_{i,n}(\mathfrak{g}))-\frac{\sum_{j=1}^n w(P_{j,n}(\mathfrak{g}))}{n}\right)e_{i,i}\in\mathfrak{g}.$$
\end{theorem}
\begin{proof}
First, note that since $\mathfrak{g}$ is a Frobenius, type-A seaweed, the oriented meander $\overrightarrow{M}(\mathfrak{g})$ consists of one path and no cycles.
 Now, if we view $F=\displaystyle\sum_{(v_i,v_j)\in E}e_{i,j}^*\in\mathfrak{g}^*$ as an element of $(\mathfrak{gl}(n))^*,$ then we may apply Lemma~\ref{lem:specmeander} to conclude that 
 \begin{align*}
        F\left(\left[\widehat{F},e_{i,j}\right]\right)&=F\left(\left[\sum_{k=1}^n\left(w(P_{k,n}(\mathfrak{p}))-\frac{\sum_{\ell=1}^n w(P_{\ell,n}(\mathfrak{g}))}{n}\right)e_{k,k},e_{i,j}\right]\right)\\
        &=F\left(\left[\sum_{k=1}^n w(P_{k,n}(\mathfrak{g}))e_{k,k},e_{i,j}\right]\right)\\
        &=F(e_{i,j}),
 \end{align*}
for all $e_{i,j}$ such that $(i,j)$ is a potentially nonzero entry in the matrix form of $\mathfrak{g}.$
Therefore, $\widehat{F}$ is an element of $\mathfrak{g}$ -- in particular, $\widehat{F}$ has trace zero -- with the property that $F\left(\left[\widehat{F},x\right]\right)=F(x),$ for all $x\in\mathfrak{g}.$ Since $F$ is Frobenius, we have that $\widehat{F}$ is a principal element of $F.$ The result follows from uniqueness of the principal element.
\end{proof}

In (Dougherty \textbf{\cite{aria}}, 2019), it is shown that the $F$ in Theorem~\ref{thm:specmeander} is Frobenius. Consequently, to determine the spectrum of a Frobenius, type-A seaweed $\mathfrak{g}$, it suffices to compute the eigenvalues of $ad~\widehat{F}$, for $\widehat{F}$ as described in Theorem~\ref{thm:specmeander}. Since $\widehat{F}$ is diagonal, the eigenvalues corresponding to $e_{i,i}-e_{i+1,i+1}\in\mathfrak{g}$ are 0. As for basis elements of the form $e_{i,j}\in\mathfrak{g}$, the proof of Theorem~\ref{thm:specmeander} shows that the eigenvalue is given by $w(P_{i,j}(\mathfrak{g}))$. The collection of such values can be nicely organized to form the \textit{spectrum matrix} of a Frobenius, type-A seaweed $\mathfrak{g}$, denoted $\Sigma(\mathfrak{g})$. The matrix $\Sigma(\mathfrak{g})$ is the element of $\mathfrak{g}$ whose $i,j$ entry is equal to $w(P_{i,j}(\mathfrak{g})).$ See Example~\ref{ex:meandereigen}.

\begin{Ex}\label{ex:meandereigen}
In Figure~\ref{fig:specmat}, we illustrate the spectrum matrix \textup(left\textup) and oriented meander \textup(right\textup) corresponding to the Frobenius, type-A seaweed $\mathfrak{g}=\mathfrak{p}^A\frac{2|4}{1|2|3}$. Using $\Sigma(\mathfrak{g})$, we find that the spectrum of $\mathfrak{g}$ is equal to $\{-2,-1^2,0^5,1^5,2^2,3\}$.
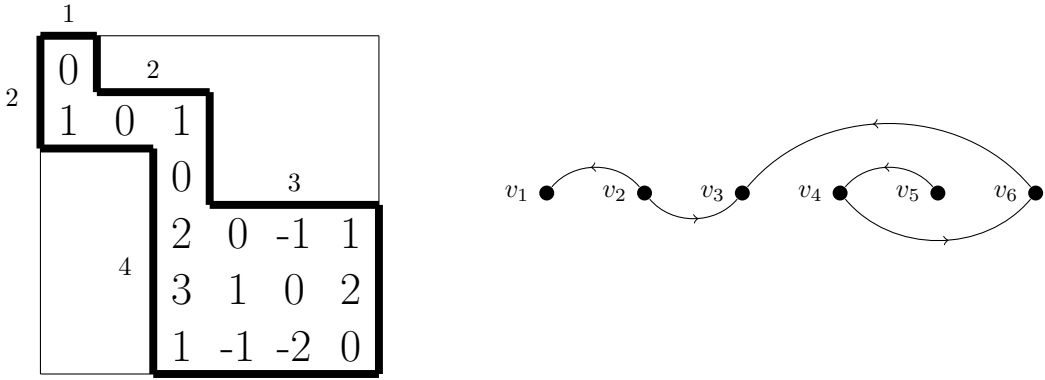
\begin{figure}[H]
$$\begin{tikzpicture}[scale=0.75]
\draw (0,0) -- (0,6);
\draw (0,6) -- (6,6);
\draw (6,6) -- (6,0);
\draw (6,0) -- (0,0);
\draw [line width=3](0,6) -- (0,4);
\draw [line width=3](0,4) -- (2,4);
\draw [line width=3](2,4) -- (2,0);
\draw [line width=3](2,0) -- (6,0);

\draw [line width=3](0,6) -- (1,6);
\draw [line width=3](1,6) -- (1,5);
\draw [line width=3](1,5) -- (3,5);
\draw [line width=3](3,5) -- (3,3);
\draw [line width=3](3,3) -- (6,3);
\draw [line width=3](6,3) -- (6,0);

%\draw [dotted] (0,6) -- (6,0);

\node at (.5,5.4) {{\LARGE 0}};
\node at (.5,4.5) {{\LARGE 1}};
\node at (1.5,4.5) {{\LARGE 0}};
\node at (2.5,4.5) {{\LARGE 1}};
\node at (2.5,3.5) {{\LARGE 0}};
\node at (2.5,2.5) {{\LARGE 2}};
\node at (2.5,1.5) {{\LARGE 3}};
\node at (2.5,0.5) {{\LARGE 1}};
\node at (3.5,2.5) {{\LARGE 0}};
\node at (3.5,1.5) {{\LARGE 1}};
\node at (3.5,0.5) {{\LARGE -1}};
\node at (4.5,2.5) {{\LARGE -1}};
\node at (4.5,1.5) {{\LARGE 0}};
\node at (5.5,2.5) {{\LARGE 1}};
\node at (5.5,1.5) {{\LARGE 2}};
\node at (4.5,0.5) {{\LARGE -2}};
\node at (5.5,0.5) {{\LARGE 0}};

\node at (.5,6.4) {1};
\node at (2,5.4) {2};
\node at (4.5,3.4) {3};
\node at (-0.5,4.9) {2};
\node at (1.5,1.9) {4};

\end{tikzpicture}\hspace{1.5cm}\begin{tikzpicture}[scale=1.3]
	\def\Node{\node [circle,  fill, inner sep=2pt]}
 \tikzset{->-/.style={decoration={
  markings,
  mark=at position .55 with {\arrow{>}}},postaction={decorate}}}
	\node at (0,0) {};
    \Node[label=left:$v_1$] (1) at (0,1.8) {};
	\Node[label=left:$v_2$] (2) at (1,1.8) {};
	\Node[label=left:$v_3$] (3) at (2,1.8) {};
	\Node[label=left:$v_4$] (4) at (3,1.8) {};
	\Node[label=left:$v_5$] (5) at (4,1.8) {};
	\Node[label=left:$v_6$] (6) at (5,1.8) {};
	\draw[->-] (2) to[bend right=50] (1);
	\draw[->-] (6) to[bend right=50] (3);
	\draw[->-] (5) to[bend right=50] (4);
	\draw[->-] (2) to[bend right=50] (3);
	\draw[->-] (4) to[bend right=50] (6);
\end{tikzpicture}$$
\caption{The spectrum matrix $\Sigma(\mathfrak{p})$ \textup(left\textup) and oriented meander $\protect\overrightarrow{M}(\mathfrak{p})$ \textup(right\textup)}\label{fig:seaweedeigen}
\label{fig:specmat}
\end{figure}
\end{Ex}

Given a Frobenius, type-A seaweed $\mathfrak{g}$, care must be taken when using $\Sigma(\mathfrak{g})$ to compute its spectrum. Since we are working with seaweed subalgebras of $\mathfrak{sl}(n)$, we do not have basis elements of the form $e_{i,i}$. Instead, basis elements corresponding to diagonal matrices are of the form $e_{i,i}-e_{i+1,i+1}$. Consequently, the multiset of values contained in $\Sigma(\mathfrak{g})$ is equal to the spectrum of $\mathfrak{g}$ with exactly one additional 0. It would be more precise to exclude the 0 in the $n,n$-location of $\Sigma(\mathfrak{g})$ and interpret a 0 in the $i,i$-location of $\Sigma(\mathfrak{g})$ as the eigenvalue corresponding to the basis element $e_{i,i}-e_{i+1,i+1}$, for $1\le i<n$. However, for the sake of arguments in the following sections, we leave the 0 in the bottom right corner. So that we may recall this subtlety later, we include the remark below.

\begin{remark}\label{rem:extra0}
Let $\mathfrak{g}$ be a Frobenius, type-A seaweed. Then the multiset of values contained in $\Sigma(\mathfrak{g})$ is equal to the spectrum of $\mathfrak{g}$ with exactly one additional value of 0.
\end{remark}

Utilizing the relationship between the spectrum of a Frobenius, type-A seaweed $\mathfrak{g}$ and its oriented meander, the authors of \textbf{\cite{unbroken}} were able to show that the set of distinct eigenvalues in the spectrum of $\mathfrak{g}$ consists of an unbroken sequence of integers centered at one-half. Moreover, the authors supplied the following conjecture.

\begin{conj}[Coll et al. \textbf{\cite{unbroken}}, 2016]\label{conj:uni}
Let $\mathfrak{g}$ be a Frobenius, type-A seaweed. If the eigenvalues in the spectrum of $\mathfrak{g}$ are written in increasing order, then the corresponding sequence of multiplicities forms a unimodal sequence about one-half.
\end{conj}

\noindent 
Recall that a sequence $a_0,\hdots,a_n$ of positive integers is called \textit{unimodal} if there exists $0\le i\le n$ such that $a_0\le a_1\le \hdots\le a_i\ge \hdots\ge a_{n-1}\ge a_n$ and is called \textit{log-concave} if $a_i^2\ge a_{i+1}a_{i-1}$, for $0<i<n$. Ongoing, we will say that a Frobenius, type-A seaweed $\mathfrak{g}$ satisfying Conjecture~\ref{conj:uni} has the \textit{unimodal spectrum property}. Similarly, if the sequence of mulitplicities, ordered as in Conjecture~\ref{conj:uni}, forms a log-concave sequence, then we say that $\mathfrak{g}$ has the \textit{log-concave spectrum property}. Note that the log-concave spectrum property implies the unimodal spectrum property.

\begin{Ex}
    Taking $\mathfrak{g}$ to be the type-A seaweed of Example~\ref{ex:meandereigen}, the sequence of multiplicities associated with the spectrum of $\mathfrak{g}$ is equal to $1,2,5,5,2,1$. Clearly, $\mathfrak{g}$ has the log-concave \textup(and, consequently, the unimodal\textup) spectrum property.
\end{Ex}

One of the main goals of this paper is to show that certain families of Frobenius, type-A seaweeds have the unimodal spectrum property. In this pursuit, an extended notion of spectrum will prove helpful. Extending the spectrum matrix to contain all weights of paths between vertices $v_i$ and $v_j$ of $M(\mathfrak{g})$ results in what we will call the \textit{extended spectrum matrix} of $\mathfrak{g}$, denoted $\widehat{\Sigma}(\mathfrak{g})$. Analogous to spectrum, excluding a single value of 0, we refer to the remaining multiset of values occurring in the extended spectrum matrix of a Frobenius, type-A seaweed $\mathfrak{g}$ as the \textit{extended spectrum} of $\mathfrak{g}$. See Example~\ref{ex:meanderexteigen}.

\begin{Ex}\label{ex:meanderexteigen}
In Figure~\ref{fig:extspec}, we illustrate the extended spectrum matrix corresponding to the Frobenius, type-A seaweed $\mathfrak{g}=\mathfrak{p}^A\frac{2|4}{1|2|3}$. Using $\widehat{\Sigma}(\mathfrak{g})$, we find that the extended spectrum of $\mathfrak{g}$ is equal to\\ $\{-3^2,-2^4,-1^7,0^9,1^7,2^4,3^2\}$.
\begin{figure}[H]
$$\begin{tikzpicture}[scale=0.75]
\draw (0,0) -- (0,6);
\draw (0,6) -- (6,6);
\draw (6,6) -- (6,0);
\draw (6,0) -- (0,0);
\draw [line width=3](0,6) -- (0,4);
\draw [line width=3](0,4) -- (2,4);
\draw [line width=3](2,4) -- (2,0);
\draw [line width=3](2,0) -- (6,0);

\draw [line width=3](0,6) -- (1,6);
\draw [line width=3](1,6) -- (1,5);
\draw [line width=3](1,5) -- (3,5);
\draw [line width=3](3,5) -- (3,3);
\draw [line width=3](3,3) -- (6,3);
\draw [line width=3](6,3) -- (6,0);

%\draw [dotted] (0,6) -- (6,0);

\node at (.5,5.5) {{\LARGE 0}};
\node at (1.5,5.5) {{\LARGE -1}};
\node at (2.5,5.5) {{\LARGE 0}};
\node at (3.5,5.5) {{\LARGE -2}};
\node at (4.5,5.5) {{\LARGE -3}};
\node at (5.5,5.5) {{\LARGE -1}};
\node at (.5,4.5) {{\LARGE 1}};
\node at (1.5,4.5) {{\LARGE 0}};
\node at (2.5,4.5) {{\LARGE 1}};
\node at (3.5,4.5) {{\LARGE -1}};
\node at (4.5,4.5) {{\LARGE -2}};
\node at (5.5,4.5) {{\LARGE 0}};
\node at (0.5,3.5) {{\LARGE 0}};
\node at (1.5,3.5) {{\LARGE -1}};
\node at (2.5,3.5) {{\LARGE 0}};
\node at (3.5,3.5) {{\LARGE -2}};
\node at (4.5,3.5) {{\LARGE -3}};
\node at (5.5,3.5) {{\LARGE -1}};
\node at (2.5,2.5) {{\LARGE 2}};
\node at (0.5,2.5) {{\LARGE 2}};
\node at (1.5,2.5) {{\LARGE 1}};
\node at (0.5,1.5) {{\LARGE 3}};
\node at (1.5,1.5) {{\LARGE 2}};
\node at (2.5,1.5) {{\LARGE 3}};
\node at (0.5,0.5) {{\LARGE 1}};
\node at (1.5,0.5) {{\LARGE 0}};
\node at (2.5,0.5) {{\LARGE 1}};
\node at (3.5,2.5) {{\LARGE 0}};
\node at (3.5,1.5) {{\LARGE 1}};
\node at (3.5,0.5) {{\LARGE -1}};
\node at (4.5,2.5) {{\LARGE -1}};
\node at (4.5,1.5) {{\LARGE 0}};
\node at (5.5,2.5) {{\LARGE 1}};
\node at (5.5,1.5) {{\LARGE 2}};
\node at (4.5,0.5) {{\LARGE -2}};
\node at (5.5,0.5) {{\LARGE 0}};
\end{tikzpicture}$$
\caption{The extended spectrum matrix $\widehat{\Sigma}(\mathfrak{p})$}\label{fig:seaweedexteigen}
\label{fig:extspec}
\end{figure}
\end{Ex}

\begin{remark}\label{rem:skew}
Let $\mathfrak{g}$ be a Frobenius, type-A seaweed. Since $w(P_{i,j}(\mathfrak{g}))=-w(P_{j,i}(\mathfrak{g}))$, it follows that $\widehat{\Sigma}(\mathfrak{g})$ is skew-symmetric.
\end{remark}

\section{Maximal parabolic}\label{sec:mp}

In this section, we consider the spectra of Frobenius, maximal parabolic, type-A seaweeds, i.e., Frobenius, type-A seaweeds of the form $\mathfrak{p}^A\frac{a|b}{a+b}$ or $\mathfrak{p}^A\frac{a+b}{a|b}$. It is well known that such algebras are Frobenius if and only if $\gcd(a,b)=1$.

Before calculating any spectra, we prove some structural lemmas. The first group, beginning with Lemma~\ref{lem:vflip} and ending with Corollary~\ref{cor:vhfes}, concern general Frobenius, type-A seaweeds. In particular, given a Frobenius, type-A seaweed, we show that switching and reversing the defining compositions does not affect the algebra's (extended) spectrum. The final group of structural lemmas -- consisting of Lemmas~\ref{lem:trbbase1} through~\ref{lem:trbgen}-- concern Frobenius, maximal parabolic, type-A seaweeds.

\begin{lemma}\label{lem:vflip}
Let $\mathfrak{g}_1=\mathfrak{p}^A\frac{a_1|\hdots|a_m}{b_1|\hdots|b_t}$ and $\mathfrak{g}_2=\mathfrak{p}^A\frac{b_1|\hdots|b_t}{a_1|\hdots|a_m}$ be Frobenius, type-A seaweeds, where $n=\sum_{i=1}^ma_i=\sum_{j=1}^tb_j$. Then $w(P_{i,j}(\mathfrak{g}_1))=w(P_{j,i}(\mathfrak{g}_2))$, for all $1\leq i,j\leq n.$
\end{lemma}
\begin{proof}
Take an arbitrary path $P_{i,j}(\mathfrak{g}_1)$ in $M(\mathfrak{g}_1)$. Recall that $P_{i,j}(\mathfrak{g}_1)$ is the unique path from $v_i$ to $v_j$ in $M(\mathfrak{g}_1)$. Let $l$ denote a line parallel to the line through the vertices of $M(\mathfrak{g}_1).$ Note that reflecting $M(\mathfrak{g}_1)$ about $l$ results in $M(\mathfrak{g}_2)$ with each vertex remaining fixed; that is, vertex $v_i$ of $M(\mathfrak{g}_1)$ becomes vertex $v_i$ of $M(\mathfrak{g}_2)$, for $1\le i\le n$. Moreover, reflecting $\overrightarrow{M}(\mathfrak{g}_1)$ about $l$ results in the meander $M(\mathfrak{g}_2)$ with each edge oriented in the opposite direction to that in $\overrightarrow{M}(\mathfrak{g}_2)$. Thus, $w(P_{i,j}(\mathfrak{g}_1))=-w(P_{i,j}(\mathfrak{g}_2))=w(P_{j,i}(\mathfrak{g}_2))$, for all $1\leq i,j\leq n.$
\end{proof}

\begin{corollary}
Let $\mathfrak{g}_1=\mathfrak{p}^A\frac{a_1|\hdots|a_m}{b_1|\hdots|b_t}$ and $\mathfrak{g}_2=\mathfrak{p}^A\frac{b_1|\hdots|b_t}{a_1|\hdots|a_m}$ be Frobenius, type-A seaweeds. Then $\Sigma(\mathfrak{g}_1)=\Sigma(\mathfrak{g}_2)^t$ and $\widehat{\Sigma}(\mathfrak{g}_1)=\widehat{\Sigma}(\mathfrak{g}_2)^t$.
\end{corollary}
\begin{proof}
Applying Lemma~\ref{lem:vflip}, it follows that the $i,j$-entry of $\widehat{\Sigma}(\mathfrak{g}_1)$ is equal to the $j,i$-entry of $\widehat{\Sigma}(\mathfrak{g}_2)$, i.e., $\widehat{\Sigma}(\mathfrak{g}_1)=\widehat{\Sigma}(\mathfrak{g}_2)^t$. Consequently, since $e_{i,j}\in\mathfrak{g}_1$ if and only if $e_{j,i}\in\mathfrak{g}_2$, $\Sigma(\mathfrak{g}_1)=\Sigma(\mathfrak{g}_2)^t$.
\end{proof}

\begin{corollary}\label{cor:vfes}
Let $\mathfrak{g}_1=\mathfrak{p}^A\frac{a_1|\hdots|a_m}{b_1|\hdots|b_t}$ and $\mathfrak{g}_2=\mathfrak{p}^A\frac{b_1|\hdots|b_t}{a_1|\hdots|a_m}$ be Frobenius, type-A seaweeds. Then the \textup(extended\textup) spectra of $\mathfrak{g}_1$ and $\mathfrak{g}_2$ are equal.
\end{corollary}

\begin{lemma}\label{lem:hflip}
Let $\mathfrak{g}_1=\mathfrak{p}^A\frac{a_1|\hdots|a_m}{b_1|\hdots|b_t}$ and $\mathfrak{g}_2=\mathfrak{p}^A\frac{a_m|\hdots|a_1}{b_t|\hdots|b_1}$ be Frobenius, type-A seaweeds, where $n=\sum_{i=1}^ma_i=\sum_{j=1}^tb_j$. Then $w(P_{i,j}(\mathfrak{g}_1))=w(P_{n-j+1,n-i+1}(\mathfrak{g}_2))$, for all $1\leq i,j\leq n.$
\end{lemma}
\begin{proof}
Take an arbitrary path $P_{i,j}(\mathfrak{g}_1)$ in $M(\mathfrak{g}_1)$. Let $l$ denote a line perpendicular to the line through the vertices of $M(\mathfrak{g}_1)$. Note that reflecting $M(\mathfrak{g}_1)$ about $l$ results in $M(\mathfrak{g}_2)$ with vertex $v_i$ becoming vertex $v_{n-i+1}$, for $1\le i\le n$. Moreover, reflecting $\overrightarrow{M}(\mathfrak{g}_1)$ about $l$ results in the meander $M(\mathfrak{g}_2)$ with each edge oriented in the opposite direction to that in $\overrightarrow{M}(\mathfrak{g}_2)$. Thus, $$w(P_{i,j}(\mathfrak{g}_1))=-w(P_{n-i+1,n-j+1}(\mathfrak{g}_2))=w(P_{n-j+1,n-i+1}(\mathfrak{g}_2)).$$
\end{proof}

\begin{corollary}
Let $\mathfrak{g}_1=\mathfrak{p}^A\frac{a_1|\dots|a_m}{b_1|\dots|b_t}$ and $\mathfrak{g}_2=\frac{a_m|\dots|a_1}{b_t|\dots|b_1}$ be Frobenius, type-A seaweeds. Then $\Sigma(\mathfrak{g}_1)=\Sigma(\mathfrak{g}_2)^{\tau}$ and $\widehat{\Sigma}(\mathfrak{g}_1)=\widehat{\Sigma}(\mathfrak{g}_2)^{\tau},$ where $\tau$ denotes transpose with respect to the antidiagonal.
\end{corollary}
\begin{proof}
Applying Lemma~\ref{lem:hflip}, it follows that the $i,j$-entry of $\widehat{\Sigma}(\mathfrak{g}_1)$ is equal to the $n-j+1,n-i+1$-entry of $\widehat{\Sigma}(\mathfrak{g}_2),$ i.e., $\widehat{\Sigma}(\mathfrak{g}_1)=\widehat{\Sigma}(\mathfrak{g}_2).$ Consequently, since $e_{i,j}\in\mathfrak{g}$ if and only if $e_{n-j+1,n-i+1},$ $\Sigma(\mathfrak{g}_1)=\Sigma(\mathfrak{g}_2)^{\tau}.$
\end{proof}

\begin{corollary}\label{cor:hfes}
Let $\mathfrak{g}_1=\mathfrak{p}^A\frac{a_1|\hdots|a_m}{b_1|\hdots|b_t}$ and $\mathfrak{g}_2=\mathfrak{p}^A\frac{a_m|\hdots|a_1}{b_t|\hdots|b_1}$ be Frobenius, type-A seaweeds. Then the \textup(extended\textup) spectra of $\mathfrak{g}_1$ and $\mathfrak{g}_2$ are equal.
\end{corollary}

\begin{corollary}\label{cor:vhfes}
Let $\mathfrak{g}_1=\mathfrak{p}^A\frac{a_1|\hdots|a_m}{b_1|\hdots|b_t}$ and $\mathfrak{g}_2=\mathfrak{p}^A\frac{b_t|\hdots|b_1}{a_m|\hdots|a_1}$ be Frobenius, type-A seaweeds. Then the \textup(extended\textup) spectra of $\mathfrak{g}_1$ and $\mathfrak{g}_2$ are equal.
\end{corollary}
\begin{proof}
The result follows upon combining Corollaries~\ref{cor:vfes} and~\ref{cor:hfes}.
\end{proof}

\begin{remark}
One can show that, for fixed compositions $(a_1,\dots,a_m)$ and $(b_1,\dots,b_t)$ of $n,$ the type-A seaweeds $\mathfrak{p}^A\frac{a_1|\dots|a_m}{b_1|\dots|b_t},$ $\mathfrak{p}^A\frac{b_1|\dots|b_t}{a_1|\dots|a_m},$ $\mathfrak{p}^A\frac{a_m|\dots|a_1}{b_t|\dots|b_1}$, and $\mathfrak{p}^A\frac{b_t|\dots|b_1}{a_m|\dots|a_1}$ are isomorphic, regardless of their \textup(shared\textup) respective index values.
\end{remark}

Now, turning to Frobenius, maximal parabolic, type-A seaweeds, the following lemmas relate the spectra of ``bigger" algebras to naturally associated ``smaller" algebras. First, given such an algebra $\mathfrak{g}=\mathfrak{p}^A\frac{a|b}{n}$ with $a>b$, Lemma~\ref{lem:trbbase1} below relates the values in the top left $a\times a$ block of $\Sigma(\mathfrak{g})$ to the values contained in $\widehat{\Sigma}(\mathfrak{g}')$, where $\mathfrak{g}'$ is a seaweed subalgebra of $\mathfrak{sl}(a)$. To aid the proof of Lemma~\ref{lem:trbbase1}, we first illustrate the result in Example~\ref{ex:trbbase1}.

\begin{Ex}\label{ex:trbbase1}
Let $\mathfrak{g}=\mathfrak{p}^A\frac{5|2}{7}$ and $\mathfrak{g}'=\mathfrak{p}^A\frac{3|2}{5}$. See Figure~\ref{fig:trb1} for \textup{(a)} $\Sigma(\mathfrak{g})$ and \textup{(b)} $M(\mathfrak{g})$, and see Figure~\ref{fig:trb2} for \textup{(a)} $\widehat{\Sigma}(\mathfrak{g}')$ and \textup{(b)} $M(\mathfrak{g}')$. Note that the top left $5\times 5$ block of $\Sigma(\mathfrak{g})$ \textup(to the left of the vertical dashed line\textup) consists of the same multiset of values as $\widehat{\Sigma}(\mathfrak{g}')$. Moreover, note that there is a copy of $M(\mathfrak{g}')$ inside of $M(\mathfrak{g})$. In particular, identifying the vertices connected by dashed edges in Figure~\ref{fig:trb1} \textup{(b)} and then rotating the resulting graph by 180 degrees yields $M(\mathfrak{g}')$. 
\begin{figure}[H]
$$\begin{tikzpicture}[scale=0.7]
\draw (0,0) -- (0,7);
\draw (0,7) -- (7,7);
\draw (7,7) -- (7,0);
\draw (7,0) -- (0,0);
\draw [line width=3](0,7) -- (0,2);
\draw [line width=3](0,2) -- (5,2);
\draw [line width=3](5,2) -- (5,0);
\draw [line width=3](5,0) -- (7,0);
\draw [line width=3](7,0) -- (7,7);
\draw [line width=3](0,7) -- (7,7);
%\draw[dashed] (5,2)--(5,7);
%\draw[dashed] (5,2)--(7,2);
%\draw [dotted] (0,6) -- (6,0);

\node at (.5,6.5) {{\LARGE 0}};
\node at (1.5,6.5) {{\LARGE 1}};
\node at (2.5,6.5) {{\LARGE -2}};
\node at (3.5,6.5) {{\LARGE 0}};
\node at (4.5,6.5) {{\LARGE -1}};
\node at (5.5,6.5) {{\LARGE 2}};
\node at (6.5,6.5) {{\LARGE 1}};

\node at (.5,5.5) {{\LARGE -1}};
\node at (1.5,5.5) {{\LARGE 0}};
\node at (2.5,5.5) {{\LARGE -3}};
\node at (3.5,5.5) {{\LARGE -1}};
\node at (4.5,5.5) {{\LARGE -2}};
\node at (5.5,5.5) {{\LARGE 1}};
\node at (6.5,5.5) {{\LARGE 0}};

\node at (.5,4.5) {{\LARGE 2}};
\node at (1.5,4.5) {{\LARGE 3}};
\node at (2.5,4.5) {{\LARGE 0}};
\node at (3.5,4.5) {{\LARGE 2}};
\node at (4.5,4.5) {{\LARGE 1}};
\node at (5.5,4.5) {{\LARGE 4}};
\node at (6.5,4.5) {{\LARGE 3}};

\node at (.5,3.5) {{\LARGE 0}};
\node at (1.5,3.5) {{\LARGE 1}};
\node at (2.5,3.5) {{\LARGE -2}};
\node at (3.5,3.5) {{\LARGE 0}};
\node at (4.5,3.5) {{\LARGE -1}};
\node at (5.5,3.5) {{\LARGE 2}};
\node at (6.5,3.5) {{\LARGE 1}};

\node at (.5,2.5) {{\LARGE 1}};
\node at (1.5,2.5) {{\LARGE 2}};
\node at (2.5,2.5) {{\LARGE -1}};
\node at (3.5,2.5) {{\LARGE 1}};
\node at (4.5,2.5) {{\LARGE 0}};
\node at (5.5,2.5) {{\LARGE 3}};
\node at (6.5,2.5) {{\LARGE 2}};

\node at (5.5,1.5) {{\LARGE 0}};
\node at (6.5,1.5) {{\LARGE -1}};

\node at (5.5,0.5) {{\LARGE 1}};
\node at (6.5,0.5) {{\LARGE 0}};

\draw[dashed] (5,2)--(5,7);

\node at (3.5, -1.5) {(a)};
\end{tikzpicture}\hspace{1.5cm}\begin{tikzpicture}[scale=1.3]
	\def\Node{\node [circle,  fill, inner sep=2pt]}
	\node at (0,0) {};
    \Node[label=left:$v_1$] (1) at (0,2) {};
	\Node[label=left:$v_2$] (2) at (1,2)  {};
	\Node[label=left:$v_3$] (3) at (2,2)  {};
	\Node[label=left:$v_4$] (4) at (3,2)  {};
	\Node[label=left:$v_5$] (5) at (4,2)  {};
	\Node[label=left:$v_6$] (6) at (5,2)  {};
	\Node[label=left:$v_7$] (7) at (6,2)  {};
	\draw (1) to[bend left=50] (5);
	\draw (2) to[bend left=50] (4);
	\draw (6) to[bend left=50] (7);
	\draw[dashed] (1) to[bend right=50] (7);
	\draw[dashed] (2) to[bend right=50] (6);
	\draw (3) to[bend right=50] (5);
    \node at (3, -1) {(b)};
\end{tikzpicture}$$
\caption{(a) $\Sigma(\mathfrak{g})$ and (b) $M(\mathfrak{g})$}\label{fig:trb1}
\end{figure}

\begin{figure}[H]
$$\begin{tikzpicture}[scale=0.7]
\draw (0,2) -- (0,7);
\draw (0,7) -- (5,7);
\draw (5,7) -- (5,2);
\draw (5,2) -- (0,2);
\draw [line width=3](0,7) -- (5,7);
\draw [line width=3](5,7) -- (5,2);
\draw [line width=3](5,2) -- (3,2);
\draw [line width=3](3,2) -- (3,4);
\draw [line width=3](3,4) -- (0,4);
\draw [line width=3](0,4) -- (0,7);

\node at (.5,6.5) {\LARGE{0}};
\node at (1.5,6.5) {\LARGE{1}};
\node at (2.5,6.5) {\LARGE{-1}};
\node at (3.5,6.5) {\LARGE{2}};
\node at (4.5,6.5) {\LARGE{1}};
\node at (.5,5.5) {\LARGE{-1}};
\node at (1.5,5.5) {\LARGE{0}};
\node at (2.5,5.5) {\LARGE{-2}};
\node at (3.5,5.5) {\LARGE{1}};
\node at (4.5,5.5) {\LARGE{0}};
\node at (.5,4.5) {\LARGE{1}};
\node at (1.5,4.5) {\LARGE{2}};
\node at (2.5,4.5) {\LARGE{0}};
\node at (3.5,4.5) {\LARGE{3}};
\node at (4.5,4.5) {\LARGE{2}};
\node at (.5,3.5) {\LARGE{-2}};
\node at (1.5,3.5) {\LARGE{-1}};
\node at (2.5,3.5) {\LARGE{3}};
\node at (3.5,3.5) {\LARGE{0}};
\node at (4.5,3.5) {\LARGE{-1}};
\node at (.5,2.5) {\LARGE{-1}};
\node at (1.5,2.5) {\LARGE{0}};
\node at (2.5,2.5) {\LARGE{-2}};
\node at (3.5,2.5) {\LARGE{1}};
\node at (4.5,2.5) {\LARGE{0}};
\node at (2.5, 0) {(a)};

\end{tikzpicture}\hspace{1.5cm}\begin{tikzpicture}
	\def\Node{\node [circle,  fill, inner sep=2pt]}
	\tikzset{->-/.style={decoration={
  markings,
  mark=at position .55 with {\arrow{>}}},postaction={decorate}}}
	\node at (0,-0.3) {};
    \Node[label=left:\footnotesize {$v_1$}] (1) at (0,1.8) {};
	\Node[label=left:\footnotesize {$v_2$}] (2) at (1,1.8) {};
	\Node[label=left:\footnotesize {$v_3$}] (3) at (2,1.8) {};
	\Node[label=left:\footnotesize {$v_4$}] (4) at (3,1.8) {};
	\Node[label=left:\footnotesize {$v_5$}] (5) at (4,1.8) {};
	\draw (3) to[bend right=50] (1);
	\draw (5) to[bend right=50] (4);
	\draw (1) to[bend right=50] (5);
	\draw (2) to[bend right=50] (4);
    \node at (2, -1) {(b)};
\end{tikzpicture}$$
\caption{(a) $\widehat{\Sigma}(\mathfrak{g}')$ and (b) $M(\mathfrak{g}')$}\label{fig:trb2}
\end{figure}
\end{Ex}

\begin{lemma}\label{lem:trbbase1}
Let $k_1,k_2\in\mathbb{Z}_{>0}$ satisfy $\gcd(k_1,k_2)=1$. If $\mathfrak{g}_1=\mathfrak{p}^A\frac{mk_1+k_2|k_1}{(m+1)k_1+k_2}$ and $\mathfrak{g}_2=\mathfrak{p}^A\frac{(m-1)k_1+k_2|k_1}{mk_1+k_2}$, then the block corresponding to rows $\{1,\hdots,mk_1+k_2\}$ and columns $\{1,\hdots,mk_1+k_2\}$ of $\Sigma(\mathfrak{g}_1)$ consists of the same multiset of values as $\widehat{\Sigma}(\mathfrak{g}_2)$.
\end{lemma}
\begin{proof}
Define the sets $$H=\{1,\hdots,mk_1+k_2\}\quad\text{and}\quad T=\{mk_1+k_2+1,\hdots,(m+1)k_1+k_2\}.$$ Note that $\{v_i~|~i\in H\}$ and $\{v_i~|~i\in T\}$ form a partition of the vertices of $M(\mathfrak{g}_1)$. The block $\Delta$ of $\Sigma(\mathfrak{g}_1)$ corresponding to rows $r\in H$ and columns $c\in H$ consists of the weights of paths between the vertices of $\{v_i~|~i\in H\}$ in $M(\mathfrak{g}_1)$. Clearly, both $\Delta$ and $\widehat{\Sigma}(\mathfrak{g}_2)$ have the same number, i.e., $mk_1+k_2$, of 0's on their respective diagonals. Thus, we need to show that $\Delta$ and $\widehat{\Sigma}(\mathfrak{g}_2)$ have the same multisets of values corresponding to off-diagonal entries. To do so, we define a weight-preserving bijection 
$$\phi:~\{P_{i,j}(\mathfrak{g}_1)~|~i,j\in H,~i\neq j\}\to\{P_{i,j}(\mathfrak{g}_2)~|~i,j\in H,~i\neq j\},$$ i.e., a weight-preserving bijection from paths in $\overrightarrow{M}(\mathfrak{g}_1)$ between distinct vertices $v_i$ and $v_j$ with $i,j\in H$ to all paths between distinct vertices in $\overrightarrow{M}(\mathfrak{g}_2)$. To aid in defining $\phi$, we make use of a transformation $\mathcal{S}$ of $\overrightarrow{M}(\mathfrak{g}_1)$ which is defined as follows. If $v_h$ is adjacent to $v_t$ in $\overrightarrow{M}(\mathfrak{g}_1),$ for $h\in H$ and $t\in T,$ then identify $v_h$ and $v_t$ and remove any edges between them. Let $\mathfrak{g}_3=\mathfrak{p}^A\frac{mk_1+k_2}{k_1|(m-1)k_1+k_2}$. Note that $\mathcal{S}(\overrightarrow{M}(\mathfrak{g}_1))=\overrightarrow{M}(\mathfrak{g}_3)$ with $\mathcal{S}$ mapping vertices $v_{i}$ and $v_{(m+1)k_1+k_2-i+1}$ in $\overrightarrow{M}(\mathfrak{g}_1)$ to $v_{i}$ in $\overrightarrow{M}(\mathfrak{g}_3)$, for $1\le i\le k_1$, and vertices $v_{j}$ in $\overrightarrow{M}(\mathfrak{g}_1)$ to $v_{j}$ in $\overrightarrow{M}(\mathfrak{g}_3)$, for $k_1+1\le j\le mk_1+k_2$. We claim that $w(P_{i,j}(\mathfrak{g}_1))=w(P_{i,j}(\mathfrak{g}_3))$, for all distinct $i,j\in H$.

To establish the claim, take distinct $i,j\in H$. 
If $P_{i,j}(\mathfrak{g}_1)$ contains no $v_{t}$, for $t\in T$, then $P_{i,j}(\mathfrak{g}_1)$ along with its orientation in $\overrightarrow{M}(\mathfrak{g}_1)$ remain invariant under $\mathcal{S}$; that is, $w(P_{i,j}(\mathfrak{g}_1))=w(P_{i,j}(\mathfrak{g}_3))$. Otherwise, $P_{i,j}(\mathfrak{g}_1)$ contains a vertex $v_{t}$, for $t\in T$. In this case, $P_{i,j}(\mathfrak{g}_1)$ contains subpaths $P_{h_1,h_2}(\mathfrak{g}_1)$ defined by sequences of vertices $v_{h_1},v_{t_1},v_{t_2},v_{h_2}$, where
\begin{itemize}
    \item $h_1,h_2\in H$ and $t_1,t_2\in T$ are distinct,
    \item $v_{h_1}$ is adjacent to $v_{t_1}$ via a bottom edge directed from $v_{h_1}$ to $v_{t_1}$ in $\overrightarrow{M}(\mathfrak{g}_1)$,
    \item $v_{t_1}$ is adjacent to $v_{t_2}$ via a top edge in $\overrightarrow{M}(\mathfrak{g}_1)$, and
    \item $v_{t_2}$ is adjacent to $v_{h_2}$ via a bottom edge directed from $v_{h_2}$ to $v_{t_2}$ in $\overrightarrow{M}(\mathfrak{g}_1)$.
\end{itemize}
To establish the claim in this case, it suffices to show that the weights of such subpaths are invariant under $\mathcal{S}$. Note that
\begin{align*}
    w(P_{h_1,h_2}(\mathfrak{g}_1))&=w(P_{h_1,t_1}(\mathfrak{g}_1))+w(P_{t_1,t_2}(\mathfrak{g}_1))+w(P_{t_2,h_2}(\mathfrak{g}_1))\\
    &=1+w(P_{t_1,t_2}(\mathfrak{g}_1))-1\\
    &=w(P_{t_1,t_2}(\mathfrak{g}_1))\\
    &=w(P_{h_1,h_2}(\mathfrak{g}_3))\\
    &=w(\mathcal{S}(P_{h_1,h_2}(\mathfrak{g}_1))),
\end{align*}
where for the penultimate equality we have used the fact that $\mathcal{S}$ maps the vertices $v_{t_i}$ to $v_{h_i}$, for $i=1,2$, and the edge connecting $v_{t_1}$ and $v_{t_2}$ in $\overrightarrow{M}(\mathfrak{g}_1)$ to the edge connecting $v_{h_1}$ and $v_{h_2}$ in $\overrightarrow{M}(\mathfrak{g}_3)$, preserving their orientations. Thus, $w(P_{h_1,h_2}(\mathfrak{g}_1))=w(\mathcal{S}(P_{h_1,h_2}(\mathfrak{g}_1)))$, and we may conclude that $w(P_{i,j}(\mathfrak{g}_1))=w(P_{i,j}(\mathfrak{g}_3))$. Consequently, $w(P_{i,j}(\mathfrak{g}_1))=w(P_{i,j}(\mathfrak{g}_3))$, for all distinct $i,j\in H$, as claimed.

Now, applying Lemmas~\ref{lem:vflip} and~\ref{lem:hflip}, it follows that $w(P_{i,j}(\mathfrak{g}_1))=w(P_{mk_1+k_2-i+1,mk_1+k_2-j+1}(\mathfrak{g}_2))$, for all distinct $i,j\in H$. Considering our work above, it follows that we can form our desired weight-preserving bijection $\phi$ by mapping the path $P_{i,j}(\mathfrak{g}_1)$ to $P_{mk_1+k_2-i+1,mk_1+k_2-j+1}(\mathfrak{g}_2)$, for all distinct $i,j\in H$. The result follows.
\end{proof}

\begin{remark}
In fact, for $\mathfrak{g}_1$ and $\mathfrak{g}_2$ as in Lemma~\ref{lem:trbbase1}, the block corresponding to rows $\{1,\hdots,k_1+k_2\}$ and columns $\{1,\hdots,k_1+k_2\}$ of $\Sigma(\mathfrak{g}_1)$ is equal to $-\widehat{\Sigma}(\mathfrak{g}_2)^t$.
\end{remark}

Next, given a Frobenius, maximal parabolic, type-A seaweed $\mathfrak{g}=\mathfrak{p}^A\frac{a|b}{n}$ with $a>b$, we show how the values in the bottom right $b\times b$ block of $\Sigma(\mathfrak{g})$ can be related to the values contained in $\widehat{\Sigma}(\mathfrak{g}')$, where $\mathfrak{g}'$ is a seaweed subalgebra of $\mathfrak{sl}(b)$. Unlike in Lemma~\ref{lem:trbbase1}, we have to consider two separate cases corresponding to $\lfloor\frac{a}{b}\rfloor=1$ and $\lfloor\frac{a}{b}\rfloor>1$. Lemma~\ref{lem:blbbase1} addresses the case $\lfloor\frac{a}{b}\rfloor=1$. An illustration of Lemma~\ref{lem:blbbase1} is given in Example~\ref{ex:brb1}.

\begin{Ex}\label{ex:brb1}
Let $\mathfrak{g}=\mathfrak{p}^A\frac{5|3}{8}$ and $\mathfrak{g}'=\mathfrak{p}^A\frac{1|2}{3}.$ See Figure~\ref{fig:brb1big} for \textup{(a)} $\Sigma(\mathfrak{g})$ and \textup{(b)} $M(\mathfrak{g}),$ and see Figure~\ref{fig:brb1small} for \textup{(a)} $\widehat{\Sigma}(\mathfrak{g}')$ and \textup{(b)} $M(\mathfrak{g}').$ Note that the bottom right $3\times 3$ block \textup(below the horizontal dashed line\textup) of $\Sigma(\mathfrak{g})$ consists of the same multiset of values as $\widehat{\Sigma}(\mathfrak{g}').$ Moreover, note that there is a copy of $M(\mathfrak{g}')$ inside of $M(\mathfrak{g}).$ In particular, identifying the vertices connected by dashed edges in Figure~\ref{fig:brb1big} \textup{(b)} and then rotating the resulting graph by 180 degrees yields $M(\mathfrak{g}').$
\end{Ex}

\begin{figure}[H]
$$\begin{tikzpicture}[scale=0.7]
\draw (0,2) -- (0,10);
\draw (0,10) -- (8,10);
\draw (8,10) -- (8,2);
\draw (8,2) -- (0,2);
\draw [line width=3](0,10) -- (8,10);
\draw [line width=3](8,10) -- (8,2);
\draw [line width=3](8,2) -- (5,2);
\draw [line width=3](5,2) -- (5,5);
\draw [line width=3](5,5) -- (0,5);
\draw [line width=3](0,5) -- (0,10);

\draw [dashed](5,5)--(8,5);

\node at (.5,9.5) {\LARGE{0}};
\node at (1.5,9.5) {\LARGE{-1}};
\node at (2.5,9.5) {\LARGE{1}};
\node at (3.5,9.5) {\LARGE{-2}};
\node at (4.5,9.5) {\LARGE{-1}};
\node at (5.5,9.5) {\LARGE{2}};
\node at (6.5,9.5) {\LARGE{0}};
\node at (7.5,9.5) {\LARGE{1}};
\node at (.5,8.5) {\LARGE{1}};
\node at (1.5,8.5) {\LARGE{0}};
\node at (2.5,8.5) {\LARGE{2}};
\node at (3.5,8.5) {\LARGE{-1}};
\node at (4.5,8.5) {\LARGE{0}};
\node at (5.5,8.5) {\LARGE{3}};
\node at (6.5,8.5) {\LARGE{1}};
\node at (7.5,8.5) {\LARGE{2}};
\node at (.5,7.5) {\LARGE{-1}};
\node at (1.5,7.5) {\LARGE{-2}};
\node at (2.5,7.5) {\LARGE{0}};
\node at (3.5,7.5) {\LARGE{-3}};
\node at (4.5,7.5) {\LARGE{-2}};
\node at (5.5,7.5) {\LARGE{1}};
\node at (6.5,7.5) {\LARGE{-1}};
\node at (7.5,7.5) {\LARGE{0}};
\node at (.5,6.5) {\LARGE{2}};
\node at (1.5,6.5) {\LARGE{1}};
\node at (2.5,6.5) {\LARGE{3}};
\node at (3.5,6.5) {\LARGE{0}};
\node at (4.5,6.5) {\LARGE{1}};
\node at (5.5,6.5) {\LARGE{4}};
\node at (6.5,6.5) {\LARGE{2}};
\node at (7.5,6.5) {\LARGE{3}};
\node at (.5,5.5) {\LARGE{1}};
\node at (1.5,5.5) {\LARGE{0}};
\node at (2.5,5.5) {\LARGE{2}};
\node at (3.5,5.5) {\LARGE{-1}};
\node at (4.5,5.5) {\LARGE{0}};
\node at (5.5,5.5) {\LARGE{3}};
\node at (6.5,5.5) {\LARGE{1}};
\node at (7.5,5.5) {\LARGE{2}};
\node at (5.5,4.5) {\LARGE{0}};
\node at (6.5,4.5) {\LARGE{-2}};
\node at (7.5,4.5) {\LARGE{-1}};
\node at (5.5,3.5) {\LARGE{2}};
\node at (6.5,3.5) {\LARGE{0}};
\node at (7.5,3.5) {\LARGE{1}};
\node at (5.5,2.5) {\LARGE{1}};
\node at (6.5,2.5) {\LARGE{-1}};
\node at (7.5,2.5) {\LARGE{0}};
\node at (4,0) {(a)};

\end{tikzpicture}\hspace{1.5cm}\begin{tikzpicture}
	\def\Node{\node [circle,  fill, inner sep=2pt]}
	\tikzset{->-/.style={decoration={
  markings,
  mark=at position .55 with {\arrow{>}}},postaction={decorate}}}
	\node at (0,-0.3) {};
    \Node[label=left:\footnotesize {$v_1$}] (1) at (0,3.2) {};
	\Node[label=left:\footnotesize {$v_2$}] (2) at (1,3.2) {};
	\Node[label=left:\footnotesize {$v_3$}] (3) at (2,3.2) {};
	\Node[label=left:\footnotesize {$v_4$}] (4) at (3,3.2) {};
	\Node[label=left:\footnotesize {$v_5$}] (5) at (4,3.2) {};
	\Node[label=left:\footnotesize {$v_6$}] (6) at (5,3.2) {};
	\Node[label=left:\footnotesize {$v_7$}] (7) at (6,3.2) {};
	\Node[label=left:\footnotesize {$v_8$}] (8) at (7,3.2) {};
	\draw[dashed] (5) to[bend right=50] (1);
	\draw[dashed] (4) to[bend right=50] (2);
	\draw[dashed] (1) to[bend right=50] (8);
	\draw[dashed] (2) to[bend right=50] (7);
	\draw[dashed] (3) to[bend right=50] (6);
	\draw (4) to[bend right=50] (5);
	\draw (8) to[bend right=50] (6);
    \node at (3.5, -1) {(b)};
\end{tikzpicture}$$
    \caption{(a) $\Sigma(\mathfrak{g})$ and (b) $M(\mathfrak{g})$}
    \label{fig:brb1big}
\end{figure}
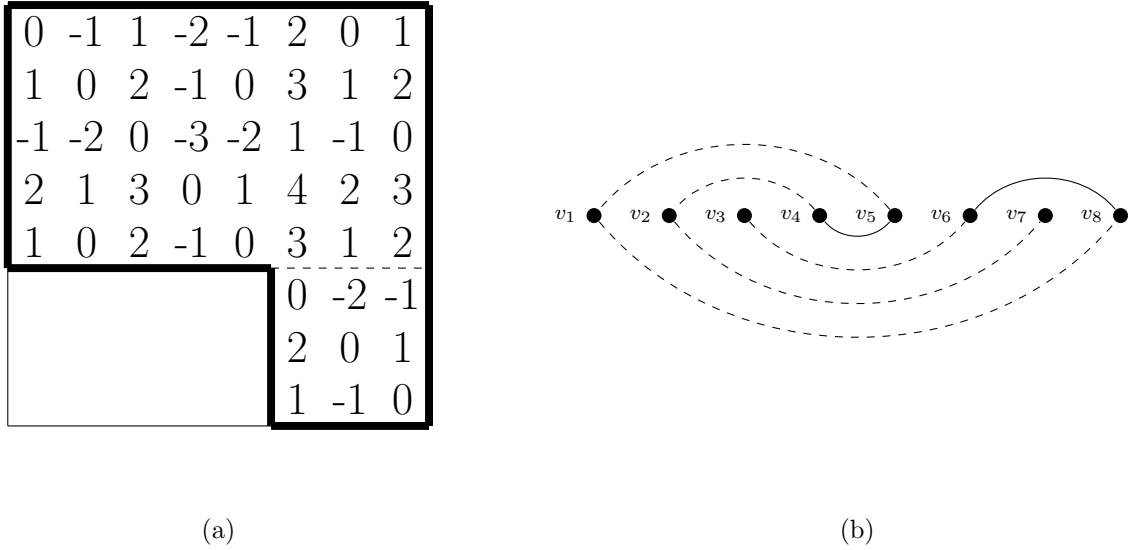

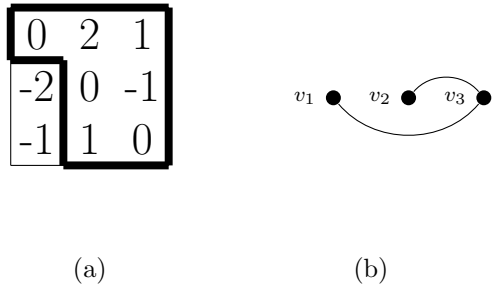
\begin{figure}[H]
$$\begin{tikzpicture}[scale=0.7]
\draw (2,2) -- (2,5);
\draw (2,5) -- (5,5);
\draw (5,5) -- (5,2);
\draw (5,2) -- (2,2);
\draw [line width=3](2,5) -- (5,5);
\draw [line width=3](5,5) -- (5,2);
\draw [line width=3](5,2) -- (3,2);
\draw [line width=3](3,2) -- (3,4);
\draw [line width=3](3,4) -- (2,4);
\draw [line width=3](2,4) -- (2,5);
\node at (2.5,4.5) {\LARGE{0}};
\node at (3.5,4.5) {\LARGE{2}};
\node at (4.5,4.5) {\LARGE{1}};
\node at (2.5,3.5) {\LARGE{-2}};
\node at (3.5,3.5) {\LARGE{0}};
\node at (4.5,3.5) {\LARGE{-1}};
\node at (2.5,2.5) {\LARGE{-1}};
\node at (3.5,2.5) {\LARGE{1}};
\node at (4.5,2.5) {\LARGE{0}};
\node at (3.5, 0) {(a)};

\end{tikzpicture}\hspace{1.5cm}\begin{tikzpicture}
	\def\Node{\node [circle,  fill, inner sep=2pt]}
	\tikzset{->-/.style={decoration={
  markings,
  mark=at position .55 with {\arrow{>}}},postaction={decorate}}}
	\node at (0,-0.3) {};
    \Node[label=left:\footnotesize {$v_1$}] (1) at (0,1.8) {};
	\Node[label=left:\footnotesize {$v_2$}] (2) at (1,1.8) {};
	\Node[label=left:\footnotesize {$v_3$}] (3) at (2,1.8) {};
	\draw (1) to[bend right=50] (3);
	\draw (3) to[bend right=50] (2);
    \node at (0.5, -0.5) {(b)};
\end{tikzpicture}$$
    \caption{(a) $\widehat{\Sigma}(\mathfrak{g}')$ and (b) $M(\mathfrak{g}')$}
    \label{fig:brb1small}
\end{figure}

\begin{lemma}\label{lem:blbbase1}
Let $k_1,k_2\in\mathbb{Z}_{>0}$ satisfy $k_1>k_2$ and $\gcd(k_1,k_2)=1$. If $\mathfrak{g}_1=\mathfrak{p}^A\frac{k_1+k_2|k_1}{2k_1+k_2}$ and $\mathfrak{g}_2=\mathfrak{p}^A\frac{k_1-k_2|k_2}{k_1}$, then the block corresponding to rows $\{k_1+k_2+1,\hdots,2k_1+k_2\}$ and columns $\{k_1+k_2+1,\hdots,2k_1+k_2\}$ in $\Sigma(\mathfrak{g}_1)$ consists of the same multiset of values as $\widehat{\Sigma}(\mathfrak{g}_2)$.
\end{lemma}
\begin{proof}
Define the sets $$H_1=\{1,\hdots,k_2\}, \quad H_2=\{k_2+1,\hdots,k_1\}, \quad H_3=\{k_1+1,\hdots,k_1+k_2\},\quad\text{and}\quad T=\{k_1+k_2+1,\hdots,2k_1+k_2\}.$$ Note that $\{v_i~|~i\in H_1\},\{v_i~|~i\in H_2\},\{v_i~|~i\in H_3\},$ and $\{v_i~|~i\in T\}$ form a partition of the vertices of $M(\mathfrak{g}_1)$ and are defined in such a way that top edges in $M(\mathfrak{g}_1)$ only connect pairs of vertices $u$ and $v$ with
\begin{itemize}
    \item  $u\in \{v_i~|~i\in H_1\}$ and $v\in \{v_i~|~i\in H_3\}$,
    \item $u,v\in \{v_i~|~i\in H_2\}$, or
    \item $u,v\in \{v_i~|~i\in T\}$
\end{itemize}
and bottom edges in $M(\mathfrak{g}_1)$ only connect pairs of vertices $u$ and $v$ with
\begin{itemize}
    \item $u\in\{v_i~|~i\in H_1\cup H_2\}$ and $v\in\{v_i~|~i\in T\}$ or
    \item $u,v\in \{v_i~|~i\in H_3\}$.
\end{itemize}
Now, the block $\Delta$ of $\Sigma(\mathfrak{g}_1)$ corresponding to rows $r\in T$ and columns $c\in T$ consists of the values $w(P_{i,j}(\mathfrak{g}_1))$, for $i,j\in T$. Clearly, both $\Delta$ and $\widehat{\Sigma}(\mathfrak{g}_2)$ have the same number, i.e., $k_1$, of 0's on their respective diagonals. Thus, we need to show that $\Delta$ and $\widehat{\Sigma}(\mathfrak{g}_2)$ have the same multisets of values corresponding to off-diagonal entries. To do so, we define a weight-preserving bijection 
$$\phi:~\{P_{i,j}(\mathfrak{g}_1)~|~i,j\in T,~i\neq j\}\to\{P_{i,j}(\mathfrak{g}_2)~|~i,j\in H_1\cup H_2,~i\neq j\},$$ i.e., a weight-preserving bijection from paths in $\overrightarrow{M}(\mathfrak{g}_1)$ between distinct vertices $u,v\in\{v_i~|~i\in T\}$ to all paths between distinct vertices in $\overrightarrow{M}(\mathfrak{g}_2)$. To aid in defining our bijection, we make use of a transformation $\mathcal{S}$ of $\overrightarrow{M}(\mathfrak{g}_1)$ which is defined as follows. 
\begin{enumerate}
    \item If $u\in \{v_i~|~i\in T\}$ is adjacent to $v\in \{v_i~|~i\in H_1\cup H_2\},$ then identify $u$ and $v$ while removing any edges between them, and then
    \item if $u\in \{v_i~|~i\in H_3\}$ is adjacent to $v\in \{v_i~|~i\in H_1\}$, then identify $u$ and $v$ while removing any edges between them.
\end{enumerate}
Let $\mathfrak{g}_3=\mathfrak{p}^A\frac{k_2|k_1-k_2}{k_1}.$ Note that $\mathcal{S}(\overrightarrow{M}(\mathfrak{g}_1))=\overrightarrow{M}(\mathfrak{g}_3)$ with $\mathcal{S}$ mapping vertices $v_{i}$ and $v_{2k_1+k_2-i+1}$ in $M(\mathfrak{g}_1)$ to $v_{i}$ in $M(\mathfrak{g}_3)$, for $1\le i\le k_1$, and vertices $v_{k_1+k_2-j+1}$ in $M(\mathfrak{g}_1)$ to $v_{j}$ in $M(\mathfrak{g}_3)$, for $1\le j\le k_2$. We claim that $w(P_{i,j}(\mathfrak{g}_1))=w(P_{2k_1+k_2-i+1,2k_1+k_2-j+1}(\mathfrak{g}_3))$, for distinct $i,j\in T$.

To establish the claim, take distinct $i,j\in T$. If $v_{i}$ and $v_{j}$ are adjacent in $M(\mathfrak{g}_1)$, then $\mathcal{S}$ maps $P_{i,j}(\mathfrak{g}_1)$ along with its orientation in $\overrightarrow{M}(\mathfrak{g}_1)$ to $P_{2k_1+k_2-i+1,2k_1+k_2-j+1}(\mathfrak{g}_3)$ along with its orientation in $\overrightarrow{M}(\mathfrak{g}_3)$; that is, $w(P_{i,j}(\mathfrak{g}_1))=w(P_{i,j}(\mathfrak{g}_3)).$ On the other hand, if $v_{i}$ and $v_{j}$ are not adjacent in $M(\mathfrak{g}_1)$, then $P_{i,j}(\mathfrak{g}_1)$ contains subpaths $P_{t_1,t_2}(\mathfrak{g}_1)$ defined by sequences of vertices $v_{t_1},v_{h_1},\hdots,v_{h_{\ell}},v_{t_2}$ with $t_1,t_2\in T$ and $h_1,\hdots,h_{\ell}\in H_1\cup H_2\cup H_3$. To establish the claim, it suffices to show that the weights of such subpaths are invariant under $\mathcal{S}$. There are two cases.
\bigskip

\noindent
\textbf{Case 1:}  Subpaths $P_{t_1,t_2}$ defined by a sequence of vertices $v_{t_1},v_{h_1^1},v_{h_1^3},v_{h_2^3},v_{h_2^1},v_{t_2}$, where
\begin{itemize}
    \item $t_1,t_2\in T$, $h_1^1,h_2^1\in H_1$, and $h_1^3,h_2^3\in H_3$,
    \item $v_{t_1}$ is adjacent to $v_{h_1^1}$ via a bottom edge directed from $v_{h_1^1}$ to $v_{t_1}$ in $\overrightarrow{M}(\mathfrak{g}_1)$,
    \item $v_{h_1^1}$ is adjacent to $v_{h_1^3}$ via a top edge directed from $v_{h_1^3}$ to $v_{h^1_1}$ in $\overrightarrow{M}(\mathfrak{g}_1)$,
    \item $v_{h_1^3}$ is adjacent to $v_{h^3_2}$ via a bottom edge in $\overrightarrow{M}(\mathfrak{g}_1)$,
    \item $v_{h^3_2}$ is adjacent to $v_{h^1_2}$ via a top edge directed from $v_{h^3_2}$ to $v_{h^1_2}$ in $\overrightarrow{M}(\mathfrak{g}_1)$, and
    \item $v_{h^1_2}$ is adjacent to $v_{t_2}$ via a bottom edge directed from $v_{h^1_2}$ to $v_{t_2}$ in $\overrightarrow{M}(\mathfrak{g}_1)$.
\end{itemize}
Note that
\begin{align*}
    w(P_{t_1,t_2}(\mathfrak{g}_1))&=w(P_{t_1,h_1^1}(\mathfrak{g}_1))+w(P_{h^1_1,h^3_1}(\mathfrak{g}_1))+w(P_{h^3_1,h^3_2}(\mathfrak{g}_1))+w(P_{h^3_2,h^1_2}(\mathfrak{g}_1))+w(P_{h^1_2,t_2}(\mathfrak{g}_1))\\
    &=-1-1+w(P_{h^3_1,h^3_2}(\mathfrak{g}_1))+1+1\\
    &=w(P_{h^3_1,h^3_2}(\mathfrak{g}_1))\\
    &=w(P_{h^1_1,h^1_2}(\mathfrak{g}_3))\\
    &=w(\mathcal{S}(P_{t_1,t_2}(\mathfrak{g}_1))),
\end{align*}
where for the penultimate equality we have used the fact that $\mathcal{S}$ maps $v_{h^3_i}$ to $v_{h^1_i}$, for $i=1,2$, as well as the edge connecting $v_{h^3_1}$ and $v_{h^3_2}$ in $\overrightarrow{M}(\mathfrak{g}_1)$ to the edge connecting $v_{h^1_1}$ and $v_{h^1_2}$ in $\overrightarrow{M}(\mathfrak{g}_1)$, preserving their orientations. Thus, it follows that $w(P_{t_1,t_2}(\mathfrak{g}_1))$ is invariant under $\mathcal{S}$.
\bigskip

\noindent
\textbf{Case 2:} Subpaths $P_{t_1,t_2}(\mathfrak{g}_1)$ defined by a sequence of vertices $v_{t_1},v_{h_1^2},v_{h_2^2},v_{t_2}$, where
\begin{itemize}
    \item $t_1,t_2\in T$ and $h_1^2,h_2^2\in H_2$,
    \item $v_{t_1}$ is adjacent to $v_{h_1^2}$ via a bottom edge directed from $v_{h_1^2}$ to $v_{t_1}$ in $\overrightarrow{M}(\mathfrak{g}_1)$,
    \item $v_{h_1^2}$ is adjacent to $v_{h_2^2}$ via a top edge in $\overrightarrow{M}(\mathfrak{g}_1)$, and
    \item $v_{h_2^2}$ is adjacent to $v_{t_2}$ via a bottom edge directed from $v_{h_2^2}$ to $v_{t_2}$ in $\overrightarrow{M}(\mathfrak{g}_1)$.
\end{itemize}
Note that 
\begin{align*}
    w(P_{t_1,t_2}(\mathfrak{g}_1))&=w(P_{t_1,h_1^2}(\mathfrak{g}_1))+w(P_{h^2_1,h^2_2}(\mathfrak{g}_1))+w(P_{h^2_2,t_2}(\mathfrak{g}_1))\\
    &=-1+w(P_{h^2_1,h^2_2}(\mathfrak{g}_1))+1\\
    &=w(P_{h^2_1,h^2_2}(\mathfrak{g}_1))\\
    &=w(P_{h^2_1,h^2_2}(\mathfrak{g}_3))\\
    &=w(\mathcal{S}(P_{t_1,t_2}(\mathfrak{g}_1))).
\end{align*}
Thus, it follows that $w(P_{t_1,t_2}(\mathfrak{g}_1))$ is invariant under $\mathcal{S}$. 
\bigskip

\noindent
Consequently, $w(P_{i,j}(\mathfrak{g}_1))=w(P_{2k_1+k_2-i+1,2k_1+k_2-j+1}(\mathfrak{g}_3))$, for all distinct $i,j\in T$, as claimed. 

Now, applying Lemma~\ref{lem:hflip}, it follows that $$w(P_{i,j}(\mathfrak{g}_1))=w(P_{j-k_1-k_2,i-k_1-k_2}(\mathfrak{g}_2)),$$ for all distinct $i,j\in T$. Therefore, the desired weight-preserving bijection $\phi$ is given by mapping the path $P_{i,j}(\mathfrak{g}_1)$ to $P_{j-k_1-k_2,i-k_1-k_2}(\mathfrak{g}_2)$, for all distinct $i,j\in T.$ The result follows.
\end{proof}

\begin{remark}
In fact, for $\mathfrak{g}_1$ and $\mathfrak{g}_2$ as in Lemma~\ref{lem:blbbase1}, the block corresponding to rows $\{k_1+k_2+1,\hdots,2k_1+k_2\}$ and columns $\{k_1+k_2+1,\hdots,2k_1+k_2\}$ of $\Sigma(\mathfrak{g}_1)$ is the same as $\widehat{\Sigma}(\mathfrak{g}_2)^t$.
\end{remark}

Now, in Lemma~\ref{lem:blbgen} below, we extend Lemma~\ref{lem:blbbase1} so that it applies to the bottom right $b\times b$ block of $\Sigma(\mathfrak{g})$ when $\lfloor\frac{a}{b}\rfloor>1$. The following example provides an illustration of Lemma~\ref{lem:blbgen}.

\begin{Ex}
Let $\mathfrak{g}=\mathfrak{p}^A\frac{5|2}{7}$ and $\mathfrak{g}'=\mathfrak{p}^A\frac{1|1}{2}$. See Figure~\ref{fig:trb1} above for \textup{(a)} $\Sigma(\mathfrak{g})$ and \textup{(b)} $M(\mathfrak{g})$ and Figure~\ref{fig:brbtiny} for \textup{(a)} $\widehat{\Sigma}(\mathfrak{g}')$ and \textup{(b)} $M(\mathfrak{g}')$. Note that the bottom right $2\times 2$ block of $\Sigma(\mathfrak{g})$ consists of the same multiset of values as $\widehat{\Sigma}(\mathfrak{g}')$. Moreover, note that there is a copy of $M(\mathfrak{g}')$ inside of $M(\mathfrak{g})$. In particular, identifying vertices connected by all edges except $\{v_6,v_7\}$ and then rotating the resulting graph by 180 degrees yields $M(\mathfrak{g}')$.
\end{Ex}
\begin{figure}[H]
$$\begin{tikzpicture}[scale=0.7]
\draw (2,2) -- (2,4);
\draw (2,4) -- (4,4);
\draw (4,4) -- (4,2);
\draw (4,2) -- (2,2);
\draw [line width=3](2,4) -- (4,4);
\draw [line width=3](4,4) -- (4,2);
\draw [line width=3](4,2) -- (3,2);
\draw [line width=3](3,2) -- (3,3);
\draw [line width=3](3,3) -- (2,3);
\draw [line width=3](2,3) -- (2,4);
\node at (2.5,3.5) {\LARGE{0}};
\node at (3.5,3.5) {\LARGE{1}};
\node at (2.5,2.5) {\LARGE{-1}};
\node at (3.5,2.5) {\LARGE{0}};
\node at (3, 0) {(a)};

\end{tikzpicture}\hspace{1.5cm}\begin{tikzpicture}
	\def\Node{\node [circle,  fill, inner sep=2pt]}
	\tikzset{->-/.style={decoration={
  markings,
  mark=at position .55 with {\arrow{>}}},postaction={decorate}}}
	\node at (0,-0.3) {};
    \Node[label=left:\footnotesize {$v_1$}] (1) at (0,1.5) {};
	\Node[label=left:\footnotesize {$v_2$}] (2) at (1,1.5) {};
	\draw (1) to[bend right=50] (2);
    \node at (0.5, -0.5) {(b)};
\end{tikzpicture}$$
    \caption{(a) $\widehat{\Sigma}(\mathfrak{g}')$ and (b) $M(\mathfrak{g}')$}
    \label{fig:brbtiny}
\end{figure}
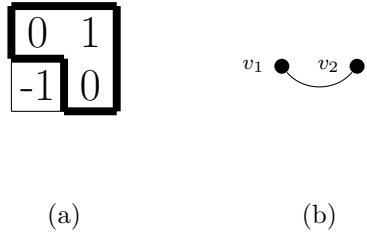

\begin{lemma}\label{lem:blbgen}
Let $k_1,k_2\in\mathbb{Z}_{>0}$ satisfy $k_1>k_2$ and $\gcd(k_1,k_2)=1$. If $\mathfrak{g}_1=\mathfrak{p}^A\frac{mk_1+k_2|k_1}{(m+1)k_1+k_2}$ and \\ $\mathfrak{g}_2=\mathfrak{p}^A\frac{k_1-k_2|k_2}{k_1}$, then the block corresponding to rows $\{mk_1+k_2+1,\hdots,(m+1)k_1+k_2\}$ and columns $\{mk_1+k_2+1,\hdots,(m+1)k_1+k_2\}$ of $\Sigma(\mathfrak{g}_1)$ consists of the same values as $\widehat{\Sigma}(\mathfrak{g}_2)$.
\end{lemma}
\begin{proof}
Note that if $m=1$, then the result follows by Lemma~\ref{lem:blbbase1}. So, assume that $m>1$. Define the sets of vertices $$H=\{1,\hdots,mk_1+k_2\},\quad H_1=\{1,\hdots,k_1\},\quad T_1=\{mk_1+k_2+1,\hdots,(m+1)k_1+k_2\},$$ and $$T_2=\{(m-1)k_1+k_2+1,\hdots,mk_1+k_2\}.$$
Let $\mathfrak{g}_3=\mathfrak{p}^A\frac{(m-1)k_1+k_2|k_1}{mk_1+k_2}$. First, we show that the block $\Delta_1$ of $\Sigma(\mathfrak{g}_1)$ corresponding to rows $r\in T_1$ and columns $c\in T_1$ contains the same values as the block $\Delta_2$ of $\Sigma(\mathfrak{g}_3)$ corresponding to rows $r\in T_2$ and columns $c\in T_2$. The block $\Delta_1$ of $\Sigma(\mathfrak{g}_1)$ consists of the values $w(P_{i,j}(\mathfrak{g}_1))$, for $i,j\in T_1$, and the block $\Delta_2$ of $\Sigma(\mathfrak{g}_3)$ consists of the values $w(P_{i,j}(\mathfrak{g}_3))$, for $i,j\in T_2$. Clearly, $\Delta_1$ and $\Delta_2$ have the same number, i.e., $k_1$, of 0's on their respective diagonals. Thus, we need to show that $\Delta_1$ and $\Delta_2$ have the same multisets of values corresponding to off-diagonal entries. To do so, we define a weight-preserving bijection 
$$\phi:~\{P_{i,j}(\mathfrak{g}_1)~|~i,j\in T_1,~i\neq j\}\to \{P_{i,j}(\mathfrak{g}_3)~|~i,j\in T_2,~i\neq j\}.$$ To aid in defining $\phi$, we make use of a transformation $\mathcal{S}$ of $\overrightarrow{M}(\mathfrak{g}_1)$ which is defined as follows. Note that in $M(\mathfrak{g}_1)$ (resp., $M(\mathfrak{g}_3)$), each vertex of $u\in\{v_i~|~i\in H_1\}$ is adjacent to a unique vertex of $v\in\{v_i~|~i\in T_1\}$ (resp., $v\in\{v_i~|~i\in T_2\}$) via an arc below. If $u\in \{v_i~|~i\in H_1\}$ is adjacent to $v\in \{v_i~|~i\in T_1\},$ then $\mathcal{S}$ identifies $u$ and $v$ while removing any edges between them. Let $\mathfrak{g}_4=\mathfrak{p}^A\frac{mk_1+k_2}{k_1|(m-1)k_1+k_2}$. Note that $\mathcal{S}(\overrightarrow{M}(\mathfrak{g}_1))=\overrightarrow{M}(\mathfrak{g}_4)$ with  $\mathcal{S}$ mapping vertices $v_{i}$ and $v_{(m+1)k_1+k_2-i+1}$ in  $\overrightarrow{M}(\mathfrak{g}_1)$ to $v_{i}$ in $\overrightarrow{M}(\mathfrak{g}_4)$, for $1\le i\le k_1$, and vertices $v_i$ in $\overrightarrow{M}(\mathfrak{g}_1)$ to $v_i$ in $\overrightarrow{M}(\mathfrak{g}_4)$, for $k_1+1\le i\le mk_1+k_2$. We claim that $w(P_{i,j}(\mathfrak{g}_1))=w(P_{(m+1)k_1+k_2-i+1,(m+1)k_1+k_2-j+1}(\mathfrak{g}_4))$, for distinct $i,j\in T_1$.

To establish the claim, take distinct $i,j\in T_1$. If $v_i$ and $v_j$ are adjacent in $M(\mathfrak{g}_1)$, then $\mathcal{S}$ maps $P_{i,j}(\mathfrak{g}_1)$ along with its orientation in $\overrightarrow{M}(\mathfrak{g}_1)$ to $P_{(m+1)k_1+k_2-i+1,(m+1)k_1+k_2-j+1}(\mathfrak{g}_4)$ along with its orientation in $\overrightarrow{M}(\mathfrak{g}_4)$; that is, $w(P_{i,j}(\mathfrak{g}_1))=w(P_{i,j}(\mathfrak{g}_4))$. On the other hand, if $v_i$ and $v_j$ are not adjacent in $M(\mathfrak{g}_1)$, then $P_{i,j}(\mathfrak{g}_1)$ contains subpaths of the form $v_{t_1},v_{h_1},\hdots,v_{h_{\ell}},v_{t_2}$ with $t_1,t_2\in T_1$ and $h_1,\hdots,h_{\ell}\in H$. To establish the claim, it suffices to show that the weights of such subpaths $P_{t_1,t_2}(\mathfrak{g}_1)$ are invariant under $\mathcal{S}.$ Since
\begin{align*}
w(P_{t_1,t_2}(\mathfrak{g}_1))&=w(P_{t_1,h_1}(\mathfrak{g}_1))+w(P_{h_1,h_{\ell}}(\mathfrak{g}_1))+w(P_{h_{\ell},t_2}(\mathfrak{g}_1)) \\
&=-1+w(P_{h_1,h_{\ell}}(\mathfrak{g}_1))+1 \\
&=w(P_{h_1,h_{\ell}}(\mathfrak{g}_1)) \\
&=w(P_{h_1,h_{\ell}}(\mathfrak{g}_4)) \\
&=w(\mathcal{S}(P_{t_1,t_2}(\mathfrak{g}_1))),
\end{align*}
the claim follows.

Now, applying Lemmas~\ref{lem:vflip} and~\ref{lem:hflip}, it follows that $w(P_{i,j}(\mathfrak{g}_1))=w(P_{i-k_1,j-k_1}(\mathfrak{g}_3))$, for all distinct $i,j\in T_1$. Consequently, our desired weight-preserving bijection sends the path $P_{i,j}(\mathfrak{g}_1)$ to $P_{i-k_1,j-k_1}(\mathfrak{g}_3)$, for all distinct $i,j\in T_1$. As this argument can be repeated until $\mathfrak{g}_3=\mathfrak{p}^A\frac{k_1+k_2|k_1}{2k_1+k_2}$, i.e., the case $m=1$, the result follows from Lemma~\ref{lem:blbbase1}.
\end{proof}

Having now addressed the top left $a\times a$ block and bottom right $b\times b$ block of $\Sigma(\mathfrak{g}),$ where $a>b$ and $\mathfrak{g}=\mathfrak{p}^A\frac{a|b}{n}$ is Frobenius, we proceed to the top right $a\times b$ block of $\Sigma(\mathfrak{g}).$ As was the case with the bottom right $b\times b$ block, we have two cases: one corresponding to $\lfloor\frac{a}{b}\rfloor=1$ and another corresponding to $\lfloor\frac{a}{b}\rfloor>1.$ The former of the two cases is covered by Lemma~\ref{lem:tlbbase1}, and the latter by Lemma~\ref{lem:trbgen}. Examples~\ref{ex:trb1} and \ref{ex:trb2} illustrate each result respectively.

\begin{Ex}\label{ex:trb1}
Let $\mathfrak{g}=\mathfrak{p}^A\frac{3|2}{5}$ and $\mathfrak{g}'=\mathfrak{p}^A\frac{2|1}{3}$. See Figure~\ref{fig:trb5} for \textup{(a)} $\Sigma(\mathfrak{g})$ and \textup{(b)} $M(\mathfrak{g})$, and see Figure~\ref{fig:trb6} for \textup{(a)} $\Sigma(\mathfrak{g}')$ and \textup{(b)} $M(\mathfrak{g}')$. Note that the top right $3\times 2$ block \textup(outlined by dashed lines\textup) of $\Sigma(\mathfrak{g})$ consists of the same multiset of values as the top $2\times 3$ block of $\Sigma(\mathfrak{g}')$ with each value incremented by 1. Moreover, note that there is a copy of $M(\mathfrak{g}')$ inside of $M(\mathfrak{g})$. In particular, identifying the vertices connected by dashed edges in Figure~\ref{fig:trb5} \textup{(b)} and then reflecting the resulting graph across the horizontal line through the vertices yields $M(\mathfrak{g}')$.
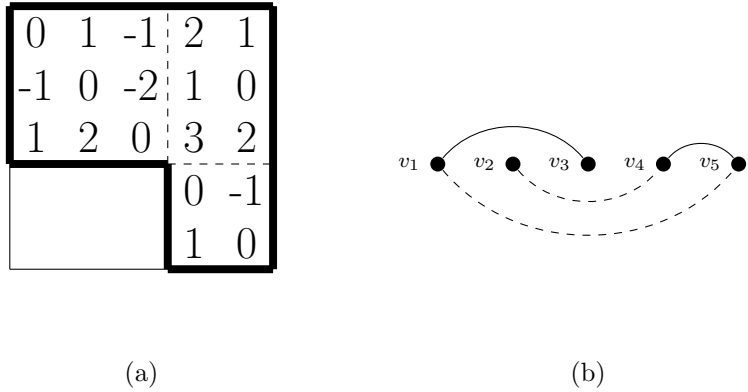
\begin{figure}[H]
$$\begin{tikzpicture}[scale=0.7]
\draw (0,2) -- (0,7);
\draw (0,7) -- (5,7);
\draw (5,7) -- (5,2);
\draw (5,2) -- (0,2);
\draw [line width=3](0,7) -- (5,7);
\draw [line width=3](5,7) -- (5,2);
\draw [line width=3](5,2) -- (3,2);
\draw [line width=3](3,2) -- (3,4);
\draw [line width=3](3,4) -- (0,4);
\draw [line width=3](0,4) -- (0,7);

\draw[dashed] (3,4)--(5,4);
\draw[dashed] (3,4)--(3,7);

\node at (.5,6.5) {\LARGE{0}};
\node at (1.5,6.5) {\LARGE{1}};
\node at (2.5,6.5) {\LARGE{-1}};
\node at (3.5,6.5) {\LARGE{2}};
\node at (4.5,6.5) {\LARGE{1}};
\node at (.5,5.5) {\LARGE{-1}};
\node at (1.5,5.5) {\LARGE{0}};
\node at (2.5,5.5) {\LARGE{-2}};
\node at (3.5,5.5) {\LARGE{1}};
\node at (4.5,5.5) {\LARGE{0}};
\node at (.5,4.5) {\LARGE{1}};
\node at (1.5,4.5) {\LARGE{2}};
\node at (2.5,4.5) {\LARGE{0}};
\node at (3.5,4.5) {\LARGE{3}};
\node at (4.5,4.5) {\LARGE{2}};
\node at (3.5,3.5) {\LARGE{0}};
\node at (4.5,3.5) {\LARGE{-1}};
\node at (3.5,2.5) {\LARGE{1}};
\node at (4.5,2.5) {\LARGE{0}};
\node at (2.5, 0) {(a)};

\end{tikzpicture}\hspace{1.5cm}\begin{tikzpicture}
	\def\Node{\node [circle,  fill, inner sep=2pt]}
	\tikzset{->-/.style={decoration={
  markings,
  mark=at position .55 with {\arrow{>}}},postaction={decorate}}}
	\node at (0,-0.3) {};
    \Node[label=left:\footnotesize {$v_1$}] (1) at (0,1.8) {};
	\Node[label=left:\footnotesize {$v_2$}] (2) at (1,1.8) {};
	\Node[label=left:\footnotesize {$v_3$}] (3) at (2,1.8) {};
	\Node[label=left:\footnotesize {$v_4$}] (4) at (3,1.8) {};
	\Node[label=left:\footnotesize {$v_5$}] (5) at (4,1.8) {};
	\draw (3) to[bend right=50] (1);
	\draw (5) to[bend right=50] (4);
	\draw[dashed] (1) to[bend right=50] (5);
	\draw[dashed] (2) to[bend right=50] (4);
    \node at (2, -1) {(b)};
\end{tikzpicture}$$
\caption{(a) $\Sigma(\mathfrak{g})$ and (b) $M(\mathfrak{g})$}\label{fig:trb5}
\end{figure}

\begin{figure}[H]
$$\begin{tikzpicture}[scale=0.7]
\draw (0,0) -- (3,0);
\draw (3,0) -- (3,3);
\draw (3,3) -- (0,3);
\draw (0,3) -- (0,0);
\draw [line width=3](2,0) -- (2,1);
\draw [line width=3](2,1) -- (0,1);
\draw [line width=3](0,1) -- (0,3);
\draw [line width=3](0,3) -- (3,3);
\draw [line width=3](3,3) -- (3,0);
\draw [line width=3](3,0) -- (2,0);
\node at (2.5,0.5) {\LARGE{0}};
\node at (0.5,1.5) {\LARGE{1}};
\node at (1.5,1.5) {\LARGE{0}};
\node at (2.5,1.5) {\LARGE{2}};
\node at (0.5,2.5) {\LARGE{0}};
\node at (1.5,2.5) {\LARGE{-1}};
\node at (2.5,2.5) {\LARGE{1}};
\node at (1.5, -1) {(a)};

\end{tikzpicture}\hspace{1.5cm}\begin{tikzpicture}
	\def\Node{\node [circle,  fill, inner sep=2pt]}
	\tikzset{->-/.style={decoration={
  markings,
  mark=at position .55 with {\arrow{>}}},postaction={decorate}}}
	\node at (0,-0.3) {};
    \Node[label=left:\footnotesize {$v_1$}] (1) at (0,1.8) {};
	\Node[label=left:\footnotesize {$v_2$}] (2) at (1,1.8) {};
    \Node[label=left:\footnotesize {$v_3$}] (3) at (2,1.8) {};
	\draw (1) to[bend right=50] (3);
    \draw (2) to[bend right=50] (1);
    \node at (1, 0) {(b)};
\end{tikzpicture}$$
\caption{(a) $\Sigma(\mathfrak{g}')$ and (b) $M(\mathfrak{g}')$}\label{fig:trb6}
\end{figure}
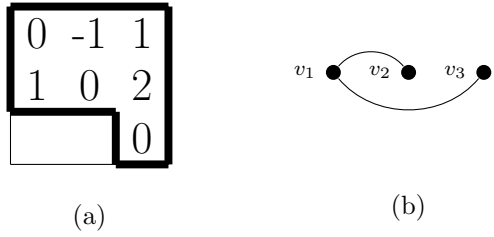
\end{Ex}

\begin{lemma}\label{lem:tlbbase1}
Let $k_1,k_2\in\mathbb{Z}_{>0}$ satisfy $k_1>k_2$ and $\gcd(k_1,k_2)=1$. If $\mathfrak{g}_1=\mathfrak{p}^A\frac{k_1+k_2|k_1}{2k_1+k_2}$ and $\mathfrak{g}_2=\mathfrak{p}^A\frac{k_1|k_2}{k_1+k_2}$, then the block corresponding to rows $\{1,\hdots,k_1+k_2\}$ and columns $\{k_1+k_2+1,\hdots,2k_1+k_2\}$ of $\Sigma(\mathfrak{g}_1)$ consists of the same multiset of values of entries as the block corresponding to rows $\{1,\hdots,k_1\}$ and columns $\{1,\hdots,k_1+k_2\}$ of $\Sigma(\mathfrak{g}_2)$ with each value incremented by 1.
\end{lemma}
\begin{proof}
Define the sets $$H_1=\{1,\hdots,k_1+k_2\},\quad H_2=\{1,\hdots,k_1\},\quad\text{and}\quad T=\{k_1+k_2+1,\hdots,2k_1+k_2\}.$$ Note that $\{v_i~|~i\in H_1\}$ and $\{v_i~|~i\in T\}$ form a partition of the vertices of $\overrightarrow{M}(\mathfrak{g}_1)$. Also, note that $H_1$, $H_2$, and $T$ are defined in such a way that top edges in $M(\mathfrak{g}_1)$ only connect pairs of vertices $u$ and $v$ with
\begin{itemize}
    \item $u,v\in \{v_i~|~i\in H_1\}$ or
    \item $u,v\in \{v_i~|~i\in T\}$
\end{itemize}
and bottom edges in $M(\mathfrak{g}_1)$ only connect pairs of vertices $u$ and $v$ with
\begin{itemize}
    \item $u\in\{v_i~|~i\in H_2\}$ and $v\in\{v_i~|~i\in T\}$ or
    \item $u,v\in \{v_i~|~i\in H_1\backslash H_2\}$.
\end{itemize}
Now, the block of $\Sigma(\mathfrak{g}_1)$ corresponding to rows $r\in H_1$ and columns $c\in T$ consists of the values $w(P_{i,j}(\mathfrak{g}_1))$, for $i\in H_1$ and $j\in T$, and the block of $\Sigma(\mathfrak{g}_2)$ corresponding to rows $r\in H_2$ and columns $c\in H_1$ consists of the values $w(P_{i,j}(\mathfrak{g}_2))$, for $i\in H_2$ and $j\in H_1$. To establish the result, we define a bijection $$\phi:\{P_{i,j}(\mathfrak{g}_1)~|~i\in H_1,~j\in T\}\to\{P_{i,j}(\mathfrak{g}_2)~|~i\in H_2,~j\in H_1\}$$ such that $w(P_{i,j}(\mathfrak{g}_1))=w(\phi(P_{i,j}(\mathfrak{g}_1)))+1$. To aid in defining $\phi$, we make use of a transformation $\mathcal{S}$ of $\overrightarrow{M}(\mathfrak{g}_1)$ which is defined as follows. If $u\in \{v_i~|~i\in H_2\}$ is adjacent to $v\in\{v_i~|~i\in T\}$, then identify $u$ and $v$ while removing edges between them. Let $\mathfrak{g}_3=\mathfrak{p}^A\frac{k_1+k_2}{k_1|k_2}$. Note that $\mathcal{S}(\overrightarrow{M}(\mathfrak{g}_1))=\overrightarrow{M}(\mathfrak{g}_3)$ with $\mathcal{S}$ mapping vertices $v_i$ and $v_{2k_1+k_2-i+1}$ in $\overrightarrow{M}(\mathfrak{g}_1)$ to $v_i$ in $\overrightarrow{M}(\mathfrak{g}_3)$, for $i\in H_2$, and vertices $v_i$ in $\overrightarrow{M}(\mathfrak{g}_1)$ to $v_i$ in $\overrightarrow{M}(\mathfrak{g}_3)$, for $i\in H_1\backslash H_2$. We claim that $w(P_{h,t}(\mathfrak{g}_1))=w(\mathcal{S}(P_{h,t}(\mathfrak{g}_1)))+1$, for all $h\in H_1$ and $t\in T$.

To establish the claim, take $h\in H_1$ and $t\in T$. There must exist a unique $h'\in H_1$ such that $(v_{h'},v_t)$ is an edge of $\overrightarrow{M}(\mathfrak{g}_1)$ and $\mathcal{S}(P_{h,t}(\mathfrak{g}_1))=P_{h,h'}(\mathfrak{g}_3)$. Recall that, in the proof of Lemma~\ref{lem:trbbase1}, it is shown that $w(P_{h,h'}(\mathfrak{g}_1))=w(P_{h,h'}(\mathfrak{g}_3))$. Now, if $(v_{h'},v_t)$ belongs to $P_{h,t}(\mathfrak{g}_1)$, then 
\begin{align*}
    w(P_{h,t}(\mathfrak{g}_1))&=w(P_{h,h'}(\mathfrak{g}_1))+w(P_{h',t}(\mathfrak{g}_1))\\
    &=w(P_{h,h'}(\mathfrak{g}_1))+1\\
    &=w(P_{h,h'}(\mathfrak{g}_3))+1.
\end{align*}
Otherwise, $P_{h,h'}(\mathfrak{g}_1)$ consists of $P_{h,t}(\mathfrak{g}_1)$ along with the edge $(v_{h'},v_t)$. Consequently,
\begin{align*}
    w(P_{h,t}(\mathfrak{g}_1))&=w(P_{h,h'}(\mathfrak{g}_1))-w(P_{t,h'}(\mathfrak{g}_1))\\
    &=w(P_{h,h'}(\mathfrak{g}_1))+1\\
    &=w(P_{h,h'}(\mathfrak{g}_3))+1.
\end{align*}
Thus, in either case, $$w(P_{h,t}(\mathfrak{g}_1))=w(\mathcal{S}(P_{h,t}(\mathfrak{g}_1)))+1=w(P_{h,h'}(\mathfrak{g}_3))+1,$$ establishing the claim. 

Now, applying Lemma~\ref{lem:vflip}, we have $w(P_{h,t}(\mathfrak{g}_1))=w(P_{h',h}(\mathfrak{g}_2))+1$. Therefore, the desired bijection $\phi$ is given by mapping the path $P_{i,j}(\mathfrak{g}_1)$ to $P_{2k_1+k_2-j+1,i}(\mathfrak{g}_2),$ for all $i\in H_1$ and $j\in T.$ The result follows.
\end{proof}

\begin{remark}
In fact, for $\mathfrak{g}_1$ and $\mathfrak{g}_2$ as in Lemma~\ref{lem:tlbbase1}, the block corresponding to rows $\{1,\hdots,k_1+k_2\}$ and columns $\{k_1+k_2+1,\hdots,2k_1+k_2\}$ of $\Sigma(\mathfrak{g}_1)$ is the same as the block corresponding to rows $\{1,\hdots,k_1\}$ and columns $\{1,\hdots,k_1+k_2\}$ of $\Sigma(\mathfrak{g}_2)$ with every value incremented by 1 and rotated 90 degrees clockwise.
\end{remark}

\begin{Ex}\label{ex:trb2}
Let $\mathfrak{g}=\mathfrak{p}^A\frac{5|2}{7}$ and $\mathfrak{g}'=\mathfrak{p}^A\frac{3|2}{5}$. See Figure~\ref{fig:trb7} for \textup{(a)} $\Sigma(\mathfrak{g})$ and \textup{(b)} $M(\mathfrak{g})$, and see Figure~\ref{fig:trb8} for \textup{(a)} $\Sigma(\mathfrak{g}')$ and \textup{(b)} $M(\mathfrak{g}')$. Note that the top right $5\times 2$ block \textup(outlined by dashed lines\textup) of $\Sigma(\mathfrak{g})$ consists of the same multiset of values as the rightmost $5\times 2$ block of $\Sigma(\mathfrak{g}')$ with each value incremented by 1. Moreover, note that there is a copy of $M(\mathfrak{g}')$ inside of $M(\mathfrak{g})$. In particular, identifying the vertices connected by dashed edges in Figure~\ref{fig:trb7} \textup{(b)} and then rotating the resulting graph by 180 degrees yields $M(\mathfrak{g}')$.
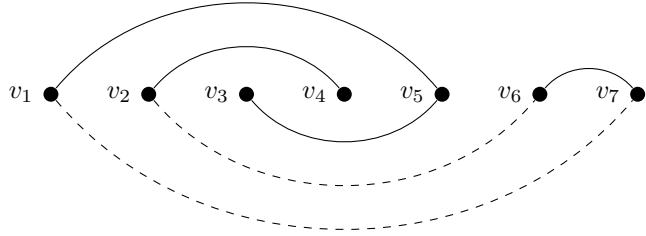
\begin{figure}[H]
$$\begin{tikzpicture}[scale=0.7]
\draw (0,0) -- (0,7);
\draw (0,7) -- (7,7);
\draw (7,7) -- (7,0);
\draw (7,0) -- (0,0);
\draw [line width=3](0,7) -- (0,2);
\draw [line width=3](0,2) -- (5,2);
\draw [line width=3](5,2) -- (5,0);
\draw [line width=3](5,0) -- (7,0);
\draw [line width=3](7,0) -- (7,7);
\draw [line width=3](0,7) -- (7,7);
\draw[dashed] (5,2)--(5,7);
\draw[dashed] (5,2)--(7,2);
%\draw [dotted] (0,6) -- (6,0);

\node at (.5,6.5) {{\LARGE 0}};
\node at (1.5,6.5) {{\LARGE 1}};
\node at (2.5,6.5) {{\LARGE -2}};
\node at (3.5,6.5) {{\LARGE 0}};
\node at (4.5,6.5) {{\LARGE -1}};
\node at (5.5,6.5) {{\LARGE 2}};
\node at (6.5,6.5) {{\LARGE 1}};

\node at (.5,5.5) {{\LARGE -1}};
\node at (1.5,5.5) {{\LARGE 0}};
\node at (2.5,5.5) {{\LARGE -3}};
\node at (3.5,5.5) {{\LARGE -1}};
\node at (4.5,5.5) {{\LARGE -2}};
\node at (5.5,5.5) {{\LARGE 1}};
\node at (6.5,5.5) {{\LARGE 0}};

\node at (.5,4.5) {{\LARGE 2}};
\node at (1.5,4.5) {{\LARGE 3}};
\node at (2.5,4.5) {{\LARGE 0}};
\node at (3.5,4.5) {{\LARGE 2}};
\node at (4.5,4.5) {{\LARGE 1}};
\node at (5.5,4.5) {{\LARGE 4}};
\node at (6.5,4.5) {{\LARGE 3}};

\node at (.5,3.5) {{\LARGE 0}};
\node at (1.5,3.5) {{\LARGE 1}};
\node at (2.5,3.5) {{\LARGE -2}};
\node at (3.5,3.5) {{\LARGE 0}};
\node at (4.5,3.5) {{\LARGE -1}};
\node at (5.5,3.5) {{\LARGE 2}};
\node at (6.5,3.5) {{\LARGE 1}};

\node at (.5,2.5) {{\LARGE 1}};
\node at (1.5,2.5) {{\LARGE 2}};
\node at (2.5,2.5) {{\LARGE -1}};
\node at (3.5,2.5) {{\LARGE 1}};
\node at (4.5,2.5) {{\LARGE 0}};
\node at (5.5,2.5) {{\LARGE 3}};
\node at (6.5,2.5) {{\LARGE 2}};

\node at (5.5,1.5) {{\LARGE 0}};
\node at (6.5,1.5) {{\LARGE -1}};

\node at (5.5,0.5) {{\LARGE 1}};
\node at (6.5,0.5) {{\LARGE 0}};

\node at (3.5, -1.5) {(a)};
\end{tikzpicture}\hspace{1.5cm}\begin{tikzpicture}[scale=1.3]
	\def\Node{\node [circle,  fill, inner sep=2pt]}
	\node at (0,0) {};
    \Node[label=left:$v_1$] (1) at (0,2) {};
	\Node[label=left:$v_2$] (2) at (1,2)  {};
	\Node[label=left:$v_3$] (3) at (2,2)  {};
	\Node[label=left:$v_4$] (4) at (3,2)  {};
	\Node[label=left:$v_5$] (5) at (4,2)  {};
	\Node[label=left:$v_6$] (6) at (5,2)  {};
	\Node[label=left:$v_7$] (7) at (6,2)  {};
	\draw (1) to[bend left=50] (5);
	\draw (2) to[bend left=50] (4);
	\draw (6) to[bend left=50] (7);
	\draw[dashed] (1) to[bend right=50] (7);
	\draw[dashed] (2) to[bend right=50] (6);
	\draw (3) to[bend right=50] (5);
    \node at (3, -1) {(b)};
\end{tikzpicture}$$
\caption{(a) $\Sigma(\mathfrak{g})$ and (b) $M(\mathfrak{g})$}\label{fig:trb7}
\end{figure}

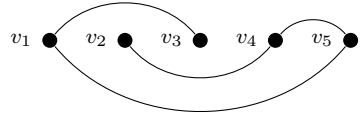
\begin{figure}[H]
$$\begin{tikzpicture}[scale=0.7]
\draw (0,2) -- (0,7);
\draw (0,7) -- (5,7);
\draw (5,7) -- (5,2);
\draw (5,2) -- (0,2);
\draw [line width=3](0,7) -- (5,7);
\draw [line width=3](5,7) -- (5,2);
\draw [line width=3](5,2) -- (3,2);
\draw [line width=3](3,2) -- (3,4);
\draw [line width=3](3,4) -- (0,4);
\draw [line width=3](0,4) -- (0,7);

\node at (.5,6.5) {\LARGE{0}};
\node at (1.5,6.5) {\LARGE{1}};
\node at (2.5,6.5) {\LARGE{-1}};
\node at (3.5,6.5) {\LARGE{2}};
\node at (4.5,6.5) {\LARGE{1}};
\node at (.5,5.5) {\LARGE{-1}};
\node at (1.5,5.5) {\LARGE{0}};
\node at (2.5,5.5) {\LARGE{-2}};
\node at (3.5,5.5) {\LARGE{1}};
\node at (4.5,5.5) {\LARGE{0}};
\node at (.5,4.5) {\LARGE{1}};
\node at (1.5,4.5) {\LARGE{2}};
\node at (2.5,4.5) {\LARGE{0}};
\node at (3.5,4.5) {\LARGE{3}};
\node at (4.5,4.5) {\LARGE{2}};
\node at (3.5,3.5) {\LARGE{0}};
\node at (4.5,3.5) {\LARGE{-1}};
\node at (3.5,2.5) {\LARGE{1}};
\node at (4.5,2.5) {\LARGE{0}};
\node at (2.5, 0) {(a)};

\end{tikzpicture}\hspace{1.5cm}\begin{tikzpicture}
	\def\Node{\node [circle,  fill, inner sep=2pt]}
	\tikzset{->-/.style={decoration={
  markings,
  mark=at position .55 with {\arrow{>}}},postaction={decorate}}}
	\node at (0,-0.3) {};
    \Node[label=left:\footnotesize {$v_1$}] (1) at (0,1.8) {};
	\Node[label=left:\footnotesize {$v_2$}] (2) at (1,1.8) {};
	\Node[label=left:\footnotesize {$v_3$}] (3) at (2,1.8) {};
	\Node[label=left:\footnotesize {$v_4$}] (4) at (3,1.8) {};
	\Node[label=left:\footnotesize {$v_5$}] (5) at (4,1.8) {};
	\draw (3) to[bend right=50] (1);
	\draw (5) to[bend right=50] (4);
	\draw (1) to[bend right=50] (5);
	\draw (2) to[bend right=50] (4);
    \node at (2, -1) {(b)};
\end{tikzpicture}$$
\caption{(a) $\Sigma(\mathfrak{g}')$ and (b) $M(\mathfrak{g}')$}\label{fig:trb8}
\end{figure}
\end{Ex}

\begin{lemma}\label{lem:trbgen}
Let $k_1,k_2\in\mathbb{Z}_{>0}$ satisfy $k_1>k_2$ and $\gcd(k_1,k_2)=1$. If $\mathfrak{g}_1=\mathfrak{p}^A\frac{mk_1+k_2|k_1}{(m+1)k_1+k_2}$ and \\ $\mathfrak{g}_2=\mathfrak{p}^A\frac{(m-1)k_1+k_2|k_1}{mk_1+k_2}$, then the block corresponding to rows $\{1,\hdots,mk_1+k_2\}$ and columns \\ $\{mk_1+k_2+1,\hdots,(m+1)k_1+k_2\}$ of $\Sigma(\mathfrak{g}_1)$ consists of the same values as the block corresponding to rows $\{1,\hdots,mk_1+k_2\}$ and columns $\{(m-1)k_1+k_2+1,\hdots,mk_1+k_2\}$ of $\Sigma(\mathfrak{g}_2)$ with each value incremented by 1.
\end{lemma}
\begin{proof}
Define the sets $$H_1=\{1,\hdots,k_1\},\quad H_2=\{k_1+1,\hdots,mk_1+k_2\},\quad T_1=\{mk_1+k_2+1,\hdots,(m+1)k_1+k_2\},$$ and $$T_2=\{(m-1)k_1+k_2+1,\hdots,mk_1+k_2\}.$$ Note that $\{v_i~|~i\in H_1\},\{v_i~|~i\in H_2\},$ and $\{v_i~|~i\in T_1\}$ form a partition of the vertices of $M(\mathfrak{g}_1)$ and are defined in such a way that top edges in $M(\mathfrak{g}_1)$ only connect pairs of vertices $u$ and $v$ with
\begin{itemize}
    \item $u,v\in \{v_i~|~i\in H_1\cup H_2\}$ or
    \item $u,v\in \{v_i~|~i\in T_1\}$
\end{itemize}
and bottom edges in $M(\mathfrak{g}_1)$ only connect pairs of vertices $u$ and $v$ with
\begin{itemize}
    \item $u\in\{v_i~|~i\in H_1\}$ and $v\in\{v_i~|~i\in T_1\}$ or
    \item $u,v\in \{v_i~|~i\in H_2\}$.
\end{itemize}
Now, the block of $\Sigma(\mathfrak{g}_1)$ corresponding to rows $r\in H_1\cup H_2$ and columns $c\in T_1$ consists of the values $w(P_{i,j}(\mathfrak{g}_1))$, for $i\in H_1\cup H_2$ and $j\in T_1$, and the block of $\Sigma(\mathfrak{g}_2)$ corresponding to rows $r\in H_1\cup H_2$ and columns $c\in T_2$ consists of the values $w(P_{i,j}(\mathfrak{g}_2))$, for $i\in H_1\cup H_2$ and $j\in T_2$. To establish the result, we define a bijection $$\phi:\{P_{i,j}(\mathfrak{g}_1)~|~i\in H_1\cup H_2,~j\in T_1\}\to\{P_{i,j}(\mathfrak{g}_2)~|~i\in H_1\cup H_2,~j\in T_2\}$$ such that $w(P_{i,j}(\mathfrak{g}_1))=w(\phi(P_{i,j}(\mathfrak{g}_1)))+1$. To aid in defining $\phi$, we make use of a transformation $\mathcal{S}$ of $\overrightarrow{M}(\mathfrak{g}_1)$ which is defined as follows. If $u\in \{v_i~|~i\in H_1\}$ is adjacent to $v\in \{v_i~|~i\in T_1\}$, then identify $u$ and $v$ while removing edges between them. Let $\mathfrak{g}_3=\mathfrak{p}^A\frac{mk_1+k_2}{k_1|(m-1)k_1+k_2}$. Note that $\mathcal{S}(\overrightarrow{M}(\mathfrak{g}_1))=\overrightarrow{M}(\mathfrak{g}_3)$ with $\mathcal{S}$ mapping vertices $v_i$ and $v_{(m+1)k_1+k_2-i+1}$ in $\overrightarrow{M}(\mathfrak{g}_1)$ to $v_i$ in $\overrightarrow{M}(\mathfrak{g}_3)$, for $i\in H_1$, and vertices $v_i$ in $\overrightarrow{M}(\mathfrak{g}_1)$ to $v_i$ in $\overrightarrow{M}(\mathfrak{g}_3)$, for $i\in H_2$. We claim that $w(P_{h,t}(\mathfrak{g}_1))=w(\mathcal{S}(P_{h,t}(\mathfrak{g}_1)))+1$, for $h\in H_1\cup H_2$ and $t\in T_1$. 

To establish the claim, take $h\in H_1\cup H_2$ and $t\in T_1$. There must exist a unique $h'\in H_1$ such that $(v_{h'},v_t)$ is an edge of $\overrightarrow{M}(\mathfrak{g}_1)$ and $\mathcal{S}(P_{h,t}(\mathfrak{g}_1))=P_{h,h'}(\mathfrak{g}_3)$. Note that in the proof of Lemma~\ref{lem:trbbase1} it is shown that $w(P_{h,h'}(\mathfrak{g}_1))=w(P_{h,h'}(\mathfrak{g}_3))$. Now, if $(v_{h'},v_t)$ belongs to $P_{h,t}(\mathfrak{g}_1)$, then 
\begin{align*}
    w(P_{h,t}(\mathfrak{g}_1))&=w(P_{h,h'}(\mathfrak{g}_1))+w(P_{h',t}(\mathfrak{g}_1))\\
    &=w(P_{h,h'}(\mathfrak{g}_1))+1\\
    &=w(P_{h,h'}(\mathfrak{g}_3))+1.
\end{align*}
Otherwise, $P_{h,h'}(\mathfrak{g}_1)$ consists of $P_{h,t}(\mathfrak{g}_1)$ along with the edge $(v_{h'},v_t)$. Consequently,
\begin{align*}
    w(P_{h,t}(\mathfrak{g}_1))&=w(P_{h,h'}(\mathfrak{g}_1))-w(P_{t,h'}(\mathfrak{g}_1))\\
    &=w(P_{h,h'}(\mathfrak{g}_1))+1\\
    &=w(P_{h,h'}(\mathfrak{g}_3))+1.
\end{align*}
Thus, in either case, $$w(P_{h,t}(\mathfrak{g}_1))=w(\mathcal{S}(P_{h,t}(\mathfrak{g}_1)))+1=w(P_{h,h'}(\mathfrak{g}_3))+1,$$ establishing the claim. 

Now, applying Lemmas~\ref{lem:vflip} and \ref{lem:hflip}, we have $$w(P_{h,t}(\mathfrak{g}_1))=w(P_{mk_1+k_2-h+1,mk_1+k_2-h'+1}(\mathfrak{g}_2))+1.$$ Therefore, the desired bijection $\phi$ is given by mapping the path $P_{i,j}(\mathfrak{g}_1)$ to $P_{mk_1+k_2-i+1,j-k_1}(\mathfrak{g}_2),$ for all $i\in H_1\cup H_2$ and $j\in T_1.$ The result follows.
\end{proof}

In the following subsections, we utilize the structural lemmas above to inductively determine closed formulas for the spectra of certain families of Frobenius, maximal parabolic, type-A seaweeds.

\subsection{$\mathfrak{p}^A\frac{k|1}{k+1}$}

In this subsection, we compute the (extended) spectra of Frobenius, type-A seaweed algebras of the form $\mathfrak{p}^A\frac{k|1}{k+1}$, for $k\ge 1$. See Figure~\ref{fig:g1p1} for illustrations of the oriented meanders corresponding to $\mathfrak{p}^A\frac{k|1}{k+1}$ for $k=1,2,$ and 3.

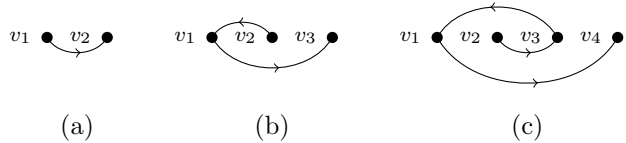
\begin{figure}[H]
    \centering
    \begin{tikzpicture}[scale=0.8]
\def\Node{\node [circle, fill, inner sep=1.5pt]}
\tikzset{->-/.style={decoration={
  markings,
  mark=at position .55 with {\arrow{>}}},postaction={decorate}}}
  \Node[label=left:\footnotesize {$v_1$}] at (0,0) {};
  \Node[label=left:\footnotesize {$v_2$}] at (1,0) {};
  \draw[->-] (0,0) to[bend right=60] (1,0);
  \node at (0.5, -1.5) {(a)};
\end{tikzpicture}\quad\quad \begin{tikzpicture}[scale=0.8]
\def\Node{\node [circle, fill, inner sep=1.5pt]}
\tikzset{->-/.style={decoration={
  markings,
  mark=at position .55 with {\arrow{>}}},postaction={decorate}}}
  \Node[label=left:\footnotesize {$v_1$}] at (0,0) {};
  \Node[label=left:\footnotesize {$v_2$}] at (1,0) {};
  \Node[label=left:\footnotesize {$v_3$}] at (2,0) {};
  \draw[->-] (0,0) to[bend right=60] (2,0);
  \draw[->-] (1,0) to[bend right=60] (0,0);
  \node at (1, -1.5) {(b)};
\end{tikzpicture}\quad\quad \begin{tikzpicture}[scale=0.8]
\def\Node{\node [circle, fill, inner sep=1.5pt]}
\tikzset{->-/.style={decoration={
  markings,
  mark=at position .55 with {\arrow{>}}},postaction={decorate}}}
  \Node[label=left:\footnotesize {$v_1$}] at (0,0) {};
  \Node[label=left:\footnotesize {$v_2$}] at (1,0) {};
  \Node[label=left:\footnotesize {$v_3$}] at (2,0) {};
  \Node[label=left:\footnotesize {$v_4$}] at (3,0) {};
  \draw[->-] (0,0) to[bend right=60] (3,0);
  \draw[->-] (1,0) to[bend right=60] (2,0);
  \draw[->-] (2,0) to[bend right=60] (0,0);
  \node at (1.5, -1.5) {(c)};
\end{tikzpicture}
    \caption{Oriented meanders of (a) $\mathfrak{p}^A\frac{1|1}{2}$, (b) $\mathfrak{p}^A\frac{2|1}{3}$, and (c) $\mathfrak{p}^A\frac{3|1}{4}$}
    \label{fig:g1p1}
\end{figure}

To establish our result, we require the following lemma which determines the multiset of values contained in the rightmost $(k+1)\times 1$ block of $\Sigma(\mathfrak{g})$, for $\mathfrak{g}=\mathfrak{p}^A\frac{k|1}{k+1}$.

\begin{lemma}\label{lem:tailk1}
Let $\mathfrak{g}_k=\mathfrak{p}^A\frac{k|1}{k+1}$, for $k\ge 1$. Then $$\{w(P_{i,k+1}(\mathfrak{g}_k))~|~1\le i\le k+1\}=\{0,1,\hdots,k\}$$ as multisets.
\end{lemma}
\begin{proof}
By induction on $k$. For $k=1$, using Figure~\ref{fig:g1p1} (a) we compute that $$\{w(P_{i,k+1}(\mathfrak{g}_k))~|~1\le i\le k+1\}=\{w(P_{1,2}(\mathfrak{g}_1)),w(P_{2,2}(\mathfrak{g}_1))\}=\{0,1\}.$$ Assume the result holds for $k-1\ge 1$. Evidently, $w(P_{k+1,k+1}(\mathfrak{g}_k))=0$. Note that the path $P_{i,k+1}(\mathfrak{g}_1)$ in $\overrightarrow{M}(\mathfrak{g}_k)$, for $1\le i<k+1$, must contain the edge $(v_1,v_{k+1})$. Thus, $$w(P_{i,k+1}(\mathfrak{g}_k))=w(P_{i,1}(\mathfrak{g}_k))+w(P_{1,k+1}(\mathfrak{g}_k))=w(P_{i,1}(\mathfrak{g}_k))+1.$$ Consequently, it suffices to compute the multiset $$\{w(P_{i,1}(\mathfrak{g}_k))~|~1\le i< k+1\}.$$ Let $\mathfrak{g}=\mathfrak{p}^A\frac{k}{1|k-1}$. Note that removing the edge $(v_1,v_{k+1})$ and vertex $v_{k+1}$ from $\overrightarrow{M}(\mathfrak{g}_k)$ results in $\overrightarrow{M}(\mathfrak{g})$. Therefore, since no path $P_{i,1}(\mathfrak{g}_k)$, for $1\le i<k+1$, contains the edge $(v_1,v_k+1)$, it follows that $$\{w(P_{i,1}(\mathfrak{g}_k))~|~1\le i<k+1\}=\{w(P_{i,1}(\mathfrak{g}))~|~1\le i<k+1\}.$$ Applying Lemmas~\ref{lem:vflip} and~\ref{lem:hflip} along with our induction hypothesis, we find that $$\{w(P_{i,1}(\mathfrak{g}))~|~1\le i<k+1\}=\{w(P_{k-i+1,k}(\mathfrak{g}_{k-1}))~|~1\le i<k+1\}=\{0,1,\hdots,k-1\};$$ that is, $$\{w(P_{i,k+1}(\mathfrak{g}_{k}))~|~1\le i<k+1\}=\{1,2,\hdots,k\}.$$ The result follows.
\end{proof}

\begin{theorem}\label{thm:base1}
Let $\mathfrak{g}_k=\mathfrak{p}^A\frac{k|1}{k+1}$, for $k\ge 1$. The spectrum of $\mathfrak{g}_k$ is equal to
$$\bigcup_{i=1}^k\left\{(-k+i)^i,(k-i+1)^i\right\}$$ and the extended spectrum is equal to $$\{0^k\}\cup\bigcup_{i=0}^{k-1}\left\{(-k+i)^{i+1},(k-i)^{i+1}\right\}.$$
\end{theorem}
\begin{proof}
By induction on $k.$ Let $\mathfrak{g}_k=\mathfrak{p}^A\frac{k|1}{k+1}$. For the base case, constructing the (extended) spectrum matrix of the seaweed algebra $\mathfrak{g}_1=\mathfrak{p}^A\frac{1|1}{2}$ directly from $\overrightarrow{M}(\mathfrak{g}_1)$ (see Figure~\ref{fig:g1p1} (a)), we find $$\Sigma(\mathfrak{g}_1)=\begin{bmatrix}
0 & 1\\
 & 0
\end{bmatrix}\quad\quad\text{and}\quad\quad \widehat{\Sigma}(\mathfrak{g}_1)=\begin{bmatrix}
0 & 1\\
-1 & 0
\end{bmatrix}.$$ Thus, the spectrum of $\mathfrak{g}_1$ is equal to $\{0,1\}$ and the extended spectrum is equal to $\{-1,0,1\}$.

Now, assume the result holds for $k-1\ge 1$. Note that $$\Sigma(\mathfrak{g}_k)=\begin{bmatrix} B_1 & B_2 \\  & 0 \end{bmatrix},$$ where
\begin{itemize}
    \item $B_1$ is the $k\times k$ block consisting of the values $w(P_{i,j}(\mathfrak{g}_k))$, for $1\le i,j\le k$, and
    \item $B_2$ is the $k\times 1$ block consisting of the values $w(P_{i,k+1}(\mathfrak{g}_k))$, for $1\le i\le k$.
\end{itemize}
Consequently, considering Remark~\ref{rem:skew}, it follows that $$\widehat{\Sigma}(\mathfrak{g}_k)=\begin{bmatrix} B_1 & B_2 \\ -B_2^t & 0 \end{bmatrix}.$$ Now, by Lemma~\ref{lem:trbbase1}, $B_1$ contains the same multiset of values as $\widehat{\Sigma}(\mathfrak{g}_{k-1})$. Thus, applying our inductive hypothesis -- and keeping Remark~\ref{rem:extra0} in mind -- it follows that $B_1$ contributes $$\{0^k\}\cup\bigcup_{i=1}^{k-1}\{(-k+i)^i,(k-i)^i\}$$ to the multiset of values contained in the (extended) spectrum matrix of $\mathfrak{g}_k$. It then follows from Lemma~\ref{lem:tailk1} that $B_2$ contributes $\{1,\hdots,k\}$ to the multiset of values contained in the (extended) spectrum of $\mathfrak{g}_k$ and $-B_2^t$ contributes $\{-k,\hdots,-1\}$ to the multiset of values contained in the extended spectrum of $\mathfrak{g}_k$. Putting these contributions together, the result follows.
\end{proof}

Considering the spectrum formula of Theorem~\ref{thm:base1}, we are immediately led to the following.

\begin{corollary}\label{cor:k1log}
For $k\ge 1$, if $\mathfrak{g}=\mathfrak{p}^A\frac{k|1}{k+1}$, then $\mathfrak{g}$ has the log-concave spectrum property.
\end{corollary}

\begin{remark}
    Combining Theorem~\ref{thm:base1} and Corollary~\ref{cor:k1log} with Corollaries~\ref{cor:vfes},~\ref{cor:hfes}, and~\ref{cor:vhfes}, we obtain similar results to Theorem~\ref{thm:base1} and Corollary~\ref{cor:k1log} for related families of Frobenius, type-A seaweeds. In particular, the \textup(extended\textup) spectrum of each of $\mathfrak{p}^A\frac{k+1}{k|1}$, $\mathfrak{p}^A\frac{1|k}{k+1}$, and $\mathfrak{p}^A\frac{k+1}{1|k},$ for $k\geq 1,$ is given in Theorem~\ref{thm:base1}, and each algebra possesses the log-concave spectrum property.
\end{remark}

\subsection{$\mathfrak{p}^A\frac{k|2}{k+2}$}

In this section, we consider the (extended) spectra of Frobenius, type-A seaweed algebras of the form $\mathfrak{p}^A\frac{k|2}{k+2}$, for $k>1$ odd. See Figure~\ref{fig:k2} for illustrations of the oriented meanders of $\mathfrak{p}^A\frac{k|2}{k+2}$, for $k=3$ and 5.

\begin{figure}[H]
    \centering
    \begin{tikzpicture}[scale=0.8]
\def\Node{\node [circle, fill, inner sep=1.5pt]}
\tikzset{->-/.style={decoration={
  markings,
  mark=at position .55 with {\arrow{>}}},postaction={decorate}}}
  \Node[label=left:\footnotesize {$v_1$}] at (0,0) {};
  \Node[label=left:\footnotesize {$v_2$}] at (1,0) {};
  \Node[label=left:\footnotesize {$v_3$}] at (2,0) {};
  \Node[label=left:\footnotesize {$v_4$}] at (3,0) {};
  \Node[label=left:\footnotesize {$v_5$}] at (4,0) {};
  \draw[->-] (0,0) to[bend right=60] (4,0);
  \draw[->-] (1,0) to[bend right=60] (3,0);
  \draw[->-] (2,0) to[bend right=60] (0,0);
  \draw[->-] (4,0) to[bend right=60] (3,0);
  \node at (2, -2.5) {(a)};
\end{tikzpicture}\quad\quad \begin{tikzpicture}[scale=0.8]
\def\Node{\node [circle, fill, inner sep=1.5pt]}
\tikzset{->-/.style={decoration={
  markings,
  mark=at position .55 with {\arrow{>}}},postaction={decorate}}}
  \Node[label=left:\footnotesize {$v_1$}] at (0,0) {};
  \Node[label=left:\footnotesize {$v_2$}] at (1,0) {};
  \Node[label=left:\footnotesize {$v_3$}] at (2,0) {};
  \Node[label=left:\footnotesize {$v_4$}] at (3,0) {};
  \Node[label=left:\footnotesize {$v_5$}] at (4,0) {};
  \Node[label=left:\footnotesize {$v_5$}] at (5,0) {};
  \Node[label=left:\footnotesize {$v_5$}] at (6,0) {};
  \draw[->-] (0,0) to[bend right=60] (6,0);
  \draw[->-] (1,0) to[bend right=60] (5,0);
  \draw[->-] (2,0) to[bend right=60] (4,0);
  \draw[->-] (4,0) to[bend right=60] (0,0);
  \draw[->-] (3,0) to[bend right=60] (1,0);
  \draw[->-] (6,0) to[bend right=60] (5,0);
  \node at (3, -2.5) {(b)};
\end{tikzpicture}
    \caption{Oriented meanders of (a) $\mathfrak{p}^A\frac{3|2}{5}$ and (b) $\mathfrak{p}^A\frac{5|2}{7}$}
    \label{fig:k2}
\end{figure}
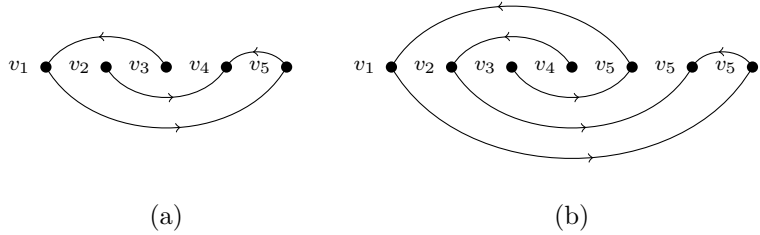

Analogous to the case of type-A seaweeds of the form $\mathfrak{p}^A\frac{k|1}{k+1}$, we require the following lemma, which determines the multiset of values contained in the top right $(k+2)\times 2$ block of $\Sigma(\mathfrak{g})$, for $\mathfrak{g}=\mathfrak{p}^A\frac{k|2}{k+2}$.

\begin{lemma}\label{lem:blbbase2}
Let $\mathfrak{g}_k=\mathfrak{p}^A\frac{k|2}{k+2}$ for $k=2m-1>5$. Then the multiset of values contained in the block of $\Sigma(\mathfrak{g}_k)$ corresponding to rows $\{1,\hdots,k\}$ and columns $\{k+1,k+2\}$ is equal to $$\{0,1^3,(m-1)^3,m^2,m+1\}\cup\bigcup_{i=2}^{m-2}\{i^4\}.$$ Moreover, the multiset of values contained in the block of $\Sigma(\mathfrak{g}_k)$ corresponding to rows $\{1,\hdots,k+2\}$ and columns $\{k+1,k+2\}$ is equal to $$\{-1,0^3,(m-1)^3,m^2,m+1\}\cup\bigcup_{i=1}^{m-2}\{i^4\}.$$
\end{lemma}
\begin{proof}
By induction on $m$. For $m=4$ (or $k=7$) one can compute directly that the multiset of values of $\Sigma(\mathfrak{g}_7)$ contained in rows $\{1,2,3,4,5,6,7\}$ and columns $\{8,9\}$ is $$\{0,1^3,2^4,3^3,4^2,5\}$$ and the multiset of values of $\Sigma(\mathfrak{g}_7)$ contained in rows $\{1,2,3,4,5,6,7,8,9\}$ and columns $\{8,9\}$ is $$\{-1,0^3,1^4,2^4,3^3,4^2,5\}.$$ Now, assume the result holds for $m-1\ge 4$ (or $k-2=2(m-1)-1=2m-3\geq 7$). By Lemma~\ref{lem:trbgen}, the multiset of values of $\Sigma(\mathfrak{g}_k)$ contained in rows $\{1,\hdots,k\}$ and columns $\{k+1,k+2\}$ is equal to the multiset of values of $\Sigma(\mathfrak{g}_{k-2})$ contained in rows $\{1,\hdots,k\}$ and columns $\{k-1,k\}$ with all values incremented by 1; that is, by the inductive hypothesis, the multiset of values of $\Sigma(\mathfrak{g}_k)$ contained in rows $\{1,\hdots,k\}$ and columns $\{k+1,k+2\}$ is equal to
\begin{align*}
    M&=\{-1+1,(0+1)^3,(m-2+1)^3,(m-1+1)^2,m+1\}\cup\bigcup_{i=1}^{m-3}\{(i+1)^4\}\\
    &=\{0,1^3,(m-1)^3,m^2,m+1\}\cup\bigcup_{i=2}^{m-2}\{i^4\},
\end{align*}
as desired.
As for the multiset of values contained in rows $\{1,\hdots,k+2\}$ and columns $\{k+1,k+2\}$ of $\Sigma(\mathfrak{g}_k)$, we simply need to add the values contained in rows $\{k+1,k+2\}$ and columns $\{k+1,k+2\}$ to $M$. By Lemma~\ref{lem:blbgen}, the multiset of values contained in rows $\{k+1,k+2\}$ and columns $\{k+1,k+2\}$ of $\Sigma(\mathfrak{g}_k)$ is equal to the extended spectrum of $\mathfrak{p}^A\frac{1|1}{2}$, i.e., $\{-1,0^2,1\}$. The result follows by induction.
\end{proof}

\begin{theorem}\label{thm:k2}
Let $\mathfrak{g}_k=\mathfrak{p}^A\frac{k|2}{k+2}$, for $k=2m-1>1$. Then the spectrum of $\mathfrak{g}_k$ is equal to $\{-2,-1^3,0^5,1^5,2^3,3\}$ if $m=2$ \textup(or $k=3$\textup) and is equal to $$\{-m,(-m+1)^3,0^{2k-1},1^{2k-1},m^3,m+1\}\cup\bigcup_{i=2}^{m-1}\{(-m+i)^{4i-2},(m-i+1)^{4i-2}\},$$ for $m>2$ \textup(or $k>3$\textup). Furthermore, the extended spectrum of $\mathfrak{g}_k$ is equal to $\{-3,-2^3,-1^5,0^6,1^5,2^3,3\}$ if $m=2$ \textup(or $k=3$\textup) and is equal to 
$$\{-m-1,-m^3,-1^{2k-1},0^{2k},1^{2k-1},m^3,m+1\}\cup\bigcup_{i=1}^{m-2}\{(-m+i)^{4i+2},(m-i)^{4i+2}\},$$ for $m>2$ \textup(or $k>3$\textup).
\end{theorem}
\begin{proof}
By induction on $m$. For $m=2$ (or $k=3$), computing directly we find that the spectrum of $\mathfrak{g}_3$ is equal to $$\{-2,-1^3,0^5,1^5,2^3,3\},$$ while the extended spectrum is equal to $$\{-3,-2^3,-1^5,0^6,1^5,2^3,3\}.$$ Similarly, for $m=3$ (or $k=5$), we find that the spectrum of $\mathfrak{g}_5$ is equal to $$\{-3,-2^3,-1^6,0^9,1^9,2^6,3^3,4\},$$ while the extended spectrum is equal to $$\{-4,-3^3,-2^6,-1^9,0^{10},1^9,2^6,3^3,4\}.$$ Now, assume the result holds for $m-1\ge 3$ (or $k-2=2(m-1)-1=2m-3\geq 5$). Note that $$\Sigma(\mathfrak{g}_k)=\begin{bmatrix} B_1 & B_2 \\ & B_3 \end{bmatrix},$$ where
\begin{itemize}
    \item $B_1$ is the $k\times k$ block consisting of the values $w(P_{i,j}(\mathfrak{g}_k))$, for $1\le i,j\le k$,
    \item $B_2$ is the $k\times 2$ block consisting of the values $w(P_{i,k+1}(\mathfrak{g}_k))$ and $w(P_{i,k+2}(\mathfrak{g}_k))$, for $1\le i\le k$,
    \item and $B_3$ is the $2\times 2$ block consisting of the values $w(P_{k+1,k+1}(\mathfrak{g}_k))$, $w(P_{k+1,k+2}(\mathfrak{g}_k))$, $w(P_{k+2,k+1}(\mathfrak{g}_k))$, and $w(P_{k+2,k+2}(\mathfrak{g}_k))$.
\end{itemize}
Consequently, considering Remark~\ref{rem:skew}, it follows that $$\widehat{\Sigma}(\mathfrak{g}_k)=\begin{bmatrix} B_1 & B_2 \\ -B_2^t & B_3 \end{bmatrix}.$$ Now, by Lemma~\ref{lem:trbbase1}, $B_1$ contains the same multiset of values as $\widehat{\Sigma}(\mathfrak{g}_{k-2})$. Thus, applying our induction hypothesis -- and keeping Remark~\ref{rem:extra0} in mind -- it follows that $B_1$ contributes 

\begin{equation}\label{eq:k2B1}
\{-m,(-m+1)^3,-1^{2k-5},0^{2k-3},1^{2k-5},(m-1)^3,m\}\cup\bigcup_{i=1}^{m-3}\{(-m+i+1)^{4i+2},(m-i-1)^{4i+2}\}
\end{equation}
to the multiset of values contained in the (extended) spectrum matrix of $\mathfrak{g}_k$. It then follows from Lemma~\ref{lem:blbbase2} that $B_2$ and $B_3$ together contribute 
\begin{equation}\label{eq:k2B2}
    \{-1,0^3,(m-1)^3,m^2,m+1\}\cup\bigcup_{i=1}^{m-2}\{i^4\}
\end{equation}
to the multiset of values contained in the (extended) spectrum matrix of $\mathfrak{g}_k$. Thus, combining the contributions (\ref{eq:k2B1}) and (\ref{eq:k2B2}), the multiset of values contained in $\Sigma(\mathfrak{g}_k)$ is equal to $$\{-m,(-m+1)^3,0^{2k},1^{2k-1},m^3,m+1\}\cup\bigcup_{i=2}^{m-1}\{(-m+i)^{4i-2},(m-i+1)^{4i-2}\}.$$ Considering Remark~\ref{rem:extra0}, we have established the claimed form of the spectrum of $\mathfrak{g}_k$.

As for the extended spectrum, it remains to add the multiset of values contained in $-B_2^t$ to the spectrum of $\mathfrak{g}_k$. Applying Lemma~\ref{lem:blbbase2}, it follows that $-B_2^t$ contains the multiset of values 
$$\{-m-1,-m^2,(-m+1)^3,-1^3,0\}\cup\bigcup_{i=2}^{m-2}\{-i^4\};$$ that is, the extended spectrum of $\mathfrak{g}_k$ is equal to $$\{-m-1,-m^3,-1^{2k-1},0^{2k},1^{2k-1},m^3,m+1\}\cup\bigcup_{i=1}^{m-2}\{(-m+i)^{4i+2},(m-i)^{4i+2}\},$$ as desired. The result follows by induction.
\end{proof}

Considering the spectrum formula of Theorem~\ref{thm:k2}, we are immediately led to the following.

\begin{corollary}\label{cor:k2log}
For $k\ge 1$ odd, if $\mathfrak{g}=\mathfrak{p}^A\frac{k|2}{k+2}$, then $\mathfrak{g}$ has the log-concave spectrum property.
\end{corollary}

\begin{remark}
    Combining Theorem~\ref{thm:k2} and Corollary~\ref{cor:k2log} with Corollaries~\ref{cor:vfes},~\ref{cor:hfes}, and~\ref{cor:vhfes}, we obtain similar results to Theorem~\ref{thm:k2} and Corollary~\ref{cor:k2log} for related families of Frobenius, type-A seaweeds. In particular, the \textup(extended\textup) spectrum of each of $\mathfrak{p}^A\frac{k+2}{k|2}$, $\mathfrak{p}^A\frac{2|k}{k+2}$, and $\mathfrak{p}^A\frac{k+2}{2|k}$, for $k\ge 1$ odd, is given in Theorem~\ref{thm:k2}, and each algebra possesses the log-concave spectrum property.
\end{remark}

As displayed by Theorems~\ref{thm:base1} and~\ref{thm:k2}, the structural lemmas established at the beginning of this section allow for the inductive determination of formulas for the spectra of Frobenius, maximal parabolic, type-A seaweeds $\mathfrak{p}^A\frac{a|b}{n}$ with $a>b$ for fixed $b$. Unfortunately, each value of $b$ seems to require its own supporting lemma analogous to Lemma~\ref{lem:blbbase2}; that is, it does not seem possible to extend the inductive procedure above to obtain explicit formulae -- or establish log-concavity or unimodality -- for the spectra of all Frobenius, maximal parabolic, type-A seaweeds. Consequently, we do not proceed by finding the spectra of Frobenius, type-A seaweeds of the form $\mathfrak{p}^A\frac{k|3}{k+3}$.

Instead, to finish this section, we use Theorems~\ref{thm:base1} and \ref{thm:k2} to determine the spectra of Frobenius, type-A seaweeds of the form $\mathfrak{p}^A\frac{a|b}{n}=\mathfrak{p}^A\frac{k+1|k}{2k+1}$ and $\mathfrak{p}^A\frac{k+2|k}{2k+2}$, where, unlike in our previous examples, neither $a$ nor $b$ is fixed.

\subsection{$\mathfrak{p}^A\frac{k+1|k}{2k+1}$}

In this subsection, we compute the spectra of Frobenius, type-A seaweed algebras of the form $\mathfrak{p}^A\frac{k+1|k}{2k+1}$, for $k\ge 1$. See Figure~\ref{fig:k+1k} for illustrations of the oriented meanders corresponding to $\mathfrak{p}^A\frac{k+1|k}{2k+1}$ for $k=1,2,$ and 3.

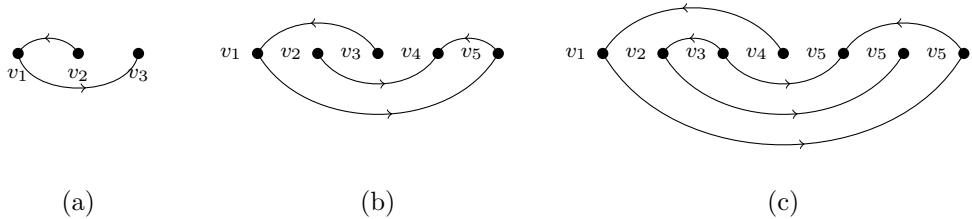
\begin{figure}[H]
    \centering
    \begin{tikzpicture}[scale=0.8]
\def\Node{\node [circle, fill, inner sep=1.5pt]}
\tikzset{->-/.style={decoration={
  markings,
  mark=at position .55 with {\arrow{>}}},postaction={decorate}}}
 \Node[label=below:\footnotesize {$v_1$}] at (0,0) {};
  \Node[label=below:\footnotesize {$v_2$}] at (1,0) {};
  \Node[label=below:\footnotesize {$v_3$}] at (2,0) {};
  \draw[->-] (1,0) to[bend right=60] (0,0);
  \draw[->-] (0,0) to[bend right=80] (2,0);
  \node at (1, -2.5) {(a)};
\end{tikzpicture}\quad\quad \begin{tikzpicture}[scale=0.8]
\def\Node{\node [circle, fill, inner sep=1.5pt]}
\tikzset{->-/.style={decoration={
  markings,
  mark=at position .55 with {\arrow{>}}},postaction={decorate}}}
  \Node[label=left:\footnotesize {$v_1$}] at (0,0) {};
  \Node[label=left:\footnotesize {$v_2$}] at (1,0) {};
  \Node[label=left:\footnotesize {$v_3$}] at (2,0) {};
  \Node[label=left:\footnotesize {$v_4$}] at (3,0) {};
  \Node[label=left:\footnotesize {$v_5$}] at (4,0) {};
  \draw[->-] (0,0) to[bend right=60] (4,0);
  \draw[->-] (1,0) to[bend right=60] (3,0);
  \draw[->-] (2,0) to[bend right=60] (0,0);
  \draw[->-] (4,0) to[bend right=60] (3,0);
  \node at (2, -2.5) {(b)};
\end{tikzpicture}\quad\quad \begin{tikzpicture}[scale=0.8]
\def\Node{\node [circle, fill, inner sep=1.5pt]}
\tikzset{->-/.style={decoration={
  markings,
  mark=at position .55 with {\arrow{>}}},postaction={decorate}}}
  \Node[label=left:\footnotesize {$v_1$}] at (0,0) {};
  \Node[label=left:\footnotesize {$v_2$}] at (1,0) {};
  \Node[label=left:\footnotesize {$v_3$}] at (2,0) {};
  \Node[label=left:\footnotesize {$v_4$}] at (3,0) {};
  \Node[label=left:\footnotesize {$v_5$}] at (4,0) {};
  \Node[label=left:\footnotesize {$v_5$}] at (5,0) {};
  \Node[label=left:\footnotesize {$v_5$}] at (6,0) {};
  \draw[->-] (0,0) to[bend right=60] (6,0);
  \draw[->-] (1,0) to[bend right=60] (5,0);
  \draw[->-] (2,0) to[bend right=60] (4,0);
  \draw[->-] (3,0) to[bend right=60] (0,0);
  \draw[->-] (2,0) to[bend right=60] (1,0);
  \draw[->-] (6,0) to[bend right=60] (4,0);
  \node at (3, -2.5) {(c)};
\end{tikzpicture}
    \caption{Oriented meanders of (a) $\mathfrak{p}^A\frac{2|1}{3}$, (b) $\mathfrak{p}^A\frac{3|2}{5}$, and (c) $\mathfrak{p}^A\frac{4|3}{7}$}
    \label{fig:k+1k}
\end{figure}

\begin{remark}
    In the preceding two sections, to compute the spectra of the algebras of interest it was necessary to also compute the extended spectra. This is not true of the seaweeds considered for the remainder of this paper, so we omit discussion about extended spectra ongoing.
\end{remark}

\begin{theorem}\label{thm:k1k}
Let $\mathfrak{g}_k=\mathfrak{p}^A\frac{k+1|k}{2k+1}$, for $k\ge 1$. The spectrum of $\mathfrak{g}_k$ is equal to $\{-1,0^2,1^2,2\}$ if $k=1$ and is equal to $$\{-k,0^{3k-1},1^{3k-1},k+1\}\cup\bigcup_{i=1}^{k-1}\{(-k+i)^{3i},(k-i+1)^{3i}\},$$ for $k>1$.
\end{theorem}
\begin{proof}
By induction on $k$. The cases $k=1$ and 2 were considered in Theorems~\ref{thm:base1} and~\ref{thm:k2}, respectively. So, assume the result holds for $k-1\geq 2$. Note that $$\Sigma(\mathfrak{g}_k)=\begin{bmatrix} B_1 & B_2 \\  & B_3 \end{bmatrix},$$ where 
\begin{itemize}
    \item $B_1$ is the $(k+1)\times (k+1)$ block consisting of the values $w(P_{i,j}(\mathfrak{g}_k))$, for $1\le i,j\le k+1$,
    \item $B_2$ is the $(k+1)\times k$ block consisting of the values $w(P_{i,j}(\mathfrak{g}_k))$, for $1\le i\le k+1$ and $k+2\le j\le 2k+1$, and
    \item $B_3$ is the $k\times k$ block consisting of the values $w(P_{i,j}(\mathfrak{g}_k))$, for $k+2\le i,j\le 2k+1$.
\end{itemize}
Now, if  $\mathfrak{p}_k=\mathfrak{p}^A\frac{k|1}{k+1}$, then by Lemma~\ref{lem:trbbase1} combined with Corollary~\ref{cor:hfes}, $B_1$ contains the same multiset of values as $\widehat{\Sigma}(\mathfrak{p}_k)$. Thus, applying Theorem~\ref{thm:base1} -- and keeping Remark~\ref{rem:extra0} in mind -- it follows that $B_1$ contributes 
\begin{equation}\label{eq:k1kB1}
    \{0^{k+1}\}\cup\bigcup_{i=0}^{k-1}\left\{(-k+i)^{i+1},(k-i)^{i+1}\right\}
\end{equation}
to the multiset of values contained in $\Sigma(\mathfrak{g}_k)$. It then follows from Lemma~\ref{lem:tlbbase1} that the multiset of values contained in $B_2$ is equal to that of $\Sigma(\mathfrak{p}_k)$ with each value increased by 1; that is, applying Theorem~\ref{thm:base1}, $B_2$ contributes 
\begin{equation}\label{eq:k1kB2}
    \bigcup_{i=2}^{k+1}\left\{(-k+i)^{i-1},(k-i+3)^{i-1}\right\}
\end{equation}
to the multiset of values contained in $\Sigma(\mathfrak{g}_k)$.  Finally, by Lemma~\ref{lem:blbbase1}, it follows that the multiset of values contained in $B_3$ is equal to that of $\widehat{\Sigma}(\mathfrak{p}_{k-1})$; that is, applying Theorem~\ref{thm:base1} -- and keeping Remark~\ref{rem:extra0} in mind -- $B_3$ contributes 
\begin{equation}\label{eq:k1kB3}
    \{0^{k}\}\cup\bigcup_{i=0}^{k-2}\left\{(-k+i+1)^{i+1},(k-i-1)^{i+1}\right\}
\end{equation}
to the multiset of values contained in $\Sigma(\mathfrak{g}_k)$. Therefore, putting contributions (\ref{eq:k1kB1}), (\ref{eq:k1kB2}), and (\ref{eq:k1kB3}) together -- keeping Remark~\ref{rem:extra0} in mind -- we find that the spectrum of $\mathfrak{g}_k$ is equal to$$\{-k,0^{3k-1},1^{3k-1},k+1\}\cup\bigcup_{i=1}^{k-1}\{(-k+i)^{3i},(k-i+1)^{3i}\},$$ as desired.
\end{proof}

Considering the spectrum formula of Theorem~\ref{thm:k1k}, we are immediately led to the following.

\begin{corollary}\label{cor:logk1k}
For $k\ge 1$, if $\mathfrak{g}=\mathfrak{p}^A\frac{k+1|k}{2k+1}$, then $\mathfrak{g}$ has the log-concave spectrum property.
\end{corollary}

\begin{remark}
    Combining Theorem~\ref{thm:k1k} and Corollary~\ref{cor:logk1k} with Corollaries~\ref{cor:vfes},~\ref{cor:hfes}, and~\ref{cor:vhfes}, we obtain similar results to Theorem~\ref{thm:k1k} and Corollary~\ref{cor:logk1k} for related families of Frobenius, type-A seaweeds. In particular, the spectrum of each of $\mathfrak{p}^A\frac{2k+1}{k+1|k}$, $\mathfrak{p}^A\frac{k|k+1}{2k+1}$, and $\mathfrak{p}^A\frac{2k+1}{k|k+1}$, for $k\ge 1$, is given in Theorem~\ref{thm:k1k}, and each algebra possesses the log-concave spectrum property.
\end{remark}

\subsection{$\mathfrak{p}^A\frac{k+2|k}{2k+2}$}
In this subsection, we compute the spectra of Frobenius, type-A seaweed algebras of the form $\mathfrak{p}^A\frac{k+2|k}{2k+2}$, for $k\geq 1$ odd. See Figure~\ref{fig:k+2k} for illustrations of the oriented menaders corresponding to $\mathfrak{p}^A\frac{k+2|k}{2k+2}$ for $k=1$ and $3.$

\begin{figure}[H]
    \centering
    \begin{tikzpicture}[scale=0.8]
\def\Node{\node [circle, fill, inner sep=1.5pt]}
\tikzset{->-/.style={decoration={
  markings,
  mark=at position .55 with {\arrow{>}}},postaction={decorate}}}
 \Node[label=left:\footnotesize {$v_1$}] at (0,0) {};
  \Node[label=left:\footnotesize {$v_2$}] at (1,0) {};
  \Node[label=left:\footnotesize {$v_3$}] at (2,0) {};
  \Node[label=left:\footnotesize {$v_4$}] at (3,0) {};
  \draw[->-] (2,0) to[bend right=60] (0,0);
  \draw[->-] (0,0) to[bend right=60] (3,0);
  \draw[->-] (1,0) to[bend right=80] (2,0);
  \node at (1.5, -2.5) {(a)};
\end{tikzpicture}\quad\quad \begin{tikzpicture}[scale=0.8]
\def\Node{\node [circle, fill, inner sep=1.5pt]}
\tikzset{->-/.style={decoration={
  markings,
  mark=at position .55 with {\arrow{>}}},postaction={decorate}}}
  \Node[label=left:\footnotesize {$v_1$}] at (0,0) {};
  \Node[label=left:\footnotesize {$v_2$}] at (1,0) {};
  \Node[label=left:\footnotesize {$v_3$}] at (2,0) {};
  \Node[label=left:\footnotesize {$v_4$}] at (3,0) {};
  \Node[label=left:\footnotesize {$v_5$}] at (4,0) {};
  \Node[label=left:\footnotesize {$v_6$}] at (5,0) {};
  \Node[label=left:\footnotesize {$v_7$}] at (6,0) {};
  \Node[label=left:\footnotesize {$v_8$}] at (7,0) {};
  \draw[->-] (4,0) to[bend right=60] (0,0);
  \draw[->-] (3,0) to[bend right=60] (1,0);
  \draw[->-] (7,0) to[bend right=60] (5,0);
  \draw[->-] (0,0) to[bend right=60] (7,0);
  \draw[->-] (1,0) to[bend right=60] (6,0);
  \draw[->-] (2,0) to[bend right=60] (5,0);
  \draw[->-] (3,0) to[bend right=60] (4,0);
  \node at (3.5, -2.5) {(b)};
\end{tikzpicture}
    \caption{Oriented meanders of (a) $\mathfrak{p}^A\frac{3|1}{4}$ and (b) $\mathfrak{p}^A\frac{5|3}{8}$}
    \label{fig:k+2k}
\end{figure}
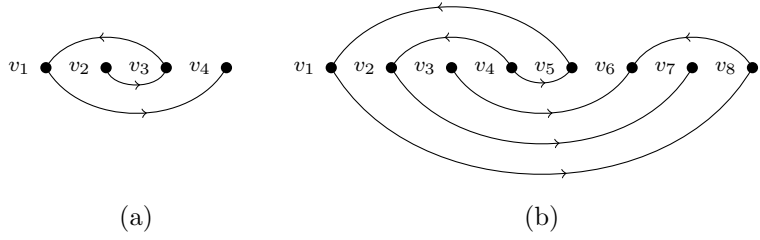

\begin{theorem}\label{thm:k+2k}
Let $\mathfrak{g}_k=\mathfrak{p}^A\frac{k+2|k}{2k+2},$ for $k=2m-1\ge 1.$ The spectrum of $\mathfrak{g}_k$ is equal to $$\{-2,-1^2,0^3,1^3,2^2,3\}$$ if $k=1$,  $$\{-3,-2^4,-1^8,0^{11},1^{11},2^8,3^4,4\}$$ if $k=3$, $$\{-4,-3^4,-2^{10},-1^{17},0^{22},1^{22},2^{17},3^{10},4^4,5\}$$ if $k=5$, $$\{-5,-4^4,-3^{10},-2^{19},-1^{28},0^{34},1^{34},2^{28},3^{19},4^{10},5^4,6\}$$ if $k=7$, and
$$\{-m-1,-m^4,(-m+1)^{10},(-m+2)^{19},-1^{6k-14},0^{6k-8},1^{6k-8},2^{6k-14},(m-1)^{19},m^{10},(m+1)^4,m+2\}$$ $$\cup\bigcup_{i=1}^{m-4}\{(-m+i+2)^{12i+18},(m-i-1)^{12i+18}\}$$ if $k>7$.
\end{theorem}

\begin{proof}
By induction on $m.$ The spectra of $\mathfrak{g}_1,$ $\mathfrak{g}_3,$ $\mathfrak{g}_5,$ and $\mathfrak{g}_7$ can be easily verified by direct calculation. For $m=5$ (or $k=9$), computing directly we find that the spectrum of $\mathfrak{g}_9$ is equal to $$\{-6,-5^4,-4^{10},-3^{19},-2^{30},-1^{40},0^{46},1^{46},2^{40},3^{30},4^{19},5^{10},6^4,7\}.$$ So, assume the result holds for $m-1\geq 5.$ Note that $$\Sigma(\mathfrak{g}_k)=\begin{bmatrix}
B_1 & B_2\\
 & B_3
\end{bmatrix},$$ where 
\begin{itemize}
    \item $B_1$ is the $(k+2)\times(k+2)$ block consisting of the values $w(P_{i,j}(\mathfrak{g}_k)),$ for $1\leq i,j\leq k+2,$
    \item $B_2$ is the $(k+2)\times k$ block consisting of the values $w(P_{i,j}(\mathfrak{g}_k)),$ for $1\leq i\leq k+2$ and $k+3\leq j\leq 2k+2,$ and
    \item $B_3$ is the $k\times k$ block consisting of the values $w(P_{i,j}(\mathfrak{g}_k)),$ for $k+3\leq i,j\leq 2k+2.$
\end{itemize}
Now, if $\mathfrak{p}_k=\mathfrak{p}^A\frac{k|2}{k+2},$ then by Lemma~\ref{lem:trbbase1} combined with Corollary~\ref{cor:hfes}, $B_1$ contains the same multiset of values of entries as $\widehat{\Sigma}(\mathfrak{p}_k).$ Thus, applying Theorem~\ref{thm:k2} -- and keeping Remark~\ref{rem:extra0} in mind -- it follows that $B_1$ contributes
\begin{equation}\label{eqn:k+2kB1}
\{-m-1,-m^3,-1^{2k-1},0^{2k+1},1^{2k-1},m^3,m+1\}\cup\bigcup_{i=1}^{m-2}\{(-m+i)^{4i+2},(m-i)^{4i+2}\}
\end{equation} to the multiset of values contained in $\Sigma(\mathfrak{g}_k).$ It then follows from Lemma~\ref{lem:tlbbase1} that the multiset of values contained in $B_2$ is equal to that of rows $\{1,\dots,k\}$ and columns $\{1,\dots,k+2\}$ in $\Sigma(\mathfrak{p}_k)$ with each value increased by 1. By the proof of Theorem~\ref{thm:k2}, the multiset of values contained in rows $\{1,\dots,k\}$ and columns $\{1,\dots,k\}$ of $\mathfrak{p}_k$ is $$\{-m,(-m+1)^3,-1^{2k-5},0^{2k-3},1^{2k-5},(m-1)^3,m\}\cup\bigcup_{i=1}^{m-3}\{(-m+i+1)^{4i+2},(m-i-1)^{4i+2}\},$$ and by Lemma~\ref{lem:blbbase2}, the multiset of values contained in rows $\{1,\dots,k\}$ and columns $\{k+1,k+2\}$ of $\mathfrak{p}_k$ is $$\{0,1^3,(m-1)^3,m^2,m+1\}\cup\bigcup_{i=2}^{m-2}\{i^4\}.$$ Thus, $B_2$ contributes 
\begin{equation}\label{eqn:k+2kB2}
\begin{aligned}
\{-m+1&,(-m+2)^3,0^{2k-5},1^{2k-2},2^{2k-2},m^6,(m+1)^3,m+2\}\\
&\cup\bigcup_{i=1}^{m-3}\{(-m+i+2)^{4i+2},(m-i)^{4i+6}\}
\end{aligned}
\end{equation}
to the multiset of values contained in $\Sigma(\mathfrak{g}_k).$ Finally, by Lemma~\ref{lem:blbbase1}, it follows that the multiset of values contained in $B_3$ is equal to that of $\widehat{\Sigma}(\mathfrak{p}_{k-2});$ that is, applying Theorem~\ref{thm:k2}, $B_3$ contributes
\begin{equation}\label{eqn:k+2kB3}
\{-m,(-m+1)^3,-1^{2k-5},0^{2k-3},1^{2k-5},(m-1)^3,m\}
\cup\bigcup_{i=1}^{m-3}\{(-m+i+1)^{4i+2},(m-i-1)^{4i+2}\}
\end{equation}
to the multiset of values contained in $\Sigma(\mathfrak{g}_k).$ Therefore, putting contributions (\ref{eqn:k+2kB1}), (\ref{eqn:k+2kB2}), and (\ref{eqn:k+2kB3}) together -- keeping Remark~\ref{rem:extra0} in mind -- we find that the spectrum of $\mathfrak{g}_k$ is equal to
$$\{-m-1,-m^4,(-m+1)^{10},(-m+2)^{19},-1^{6k-14},0^{6k-8},1^{6k-8},2^{6k-14},(m-1)^{19},m^{10},(m+1)^4,m+2\}$$
$$\cup\bigcup_{i=1}^{m-4}\{(-m+i+2)^{12i+18},(m-i-1)^{12i+18}\}$$ as desired.
\end{proof}

Considering the spectrum formula of Theorem~\ref{thm:k+2k}, we are immediately led to the following.

\begin{corollary}\label{cor:logk2k}
For $k\ge 1$, if $\mathfrak{g}=\mathfrak{p}^A\frac{k+2|k}{2k+2}$, then $\mathfrak{g}$ has the log-concave spectrum property.
\end{corollary}

\begin{remark}
    Combining Theorem~\ref{thm:k+2k} and Corollary~\ref{cor:logk2k} with Corollaries~\ref{cor:vfes},~\ref{cor:hfes}, and~\ref{cor:vhfes}, we obtain similar results to Theorem~\ref{thm:k+2k} and Corollary~\ref{cor:logk2k} for related families of Frobenius, type-A seaweeds. In particular, the specctrum of each of $\mathfrak{p}^A\frac{2k+2}{k+2|k}$, $\mathfrak{p}^A\frac{k|k+2}{2k+2}$, and $\mathfrak{p}^A\frac{2k+2}{k|k+2}$, for $k\ge 1$ odd, is given in Theorem~\ref{thm:k+2k}, and each algebra possesses the log-concave spectrum property.
\end{remark}

One quality that makes the family of maximal parabolic, type-A seaweeds a desirable candidate family in our investigations is Elashvili's (\textbf{\cite{Elashvili}}, 1990) simple closed-form index formula $$\ind\mathfrak{p}^A\frac{a|b}{n}=\gcd(a,b)-1,$$ allowing for quick identification of Frobenius such algebras. Outside of this family, there are two others known with similar index formulas: the families of type-A seaweeds of the form $\mathfrak{p}^A\frac{a|b|c}{a+b+c}$ and of the form $\mathfrak{p}^A\frac{a|b}{c|d}$ which both have index given by $\gcd(a+b,b+c)-1$ (see \textbf{\cite{meanders2}}). For the sake of comparison, we provide the formulas for the spectra of two such families of seaweeds below. The proofs are left to the interested reader.

\begin{theorem}\label{thm:2k1/12k}
Let $\mathfrak{g}_k=\mathfrak{p}^A\frac{2k|1}{1|2k}$, for $k\ge 1$. Then the spectrum of $\mathfrak{g}_k$ is equal to 
$$\bigcup_{i=1}^k\{(-k+i)^{4i-2},(k-i+1)^{4i-2}\}.$$ Moreover, $\mathfrak{g}_k$ has the log-concave spectrum property.
\end{theorem}

\begin{theorem}\label{thm:2k11/2k2}
Let $\mathfrak{g}_k=\mathfrak{p}^A\frac{2k|1|1}{2k+2}$, for $k\ge 1$. Then the spectrum of $\mathfrak{g}_k$ is equal to $$\{-k,k+1\}\cup\bigcup_{i=1}^k\{(-k+i)^{4i},(k-i+1)^{4i}\}.$$ Moreover, $\mathfrak{g}_k$ has the log-concave spectrum property.
\end{theorem}

Note that all families of Frobenius, type-A seaweeds considered above have been defined by compositions with a fixed number of parts, while the sizes of the parts vary. In the next section, we consider families of Frobenius, type-A seaweeds parametrized by the number of parts of a fixed size in their defining compositions. In contrast to the spectra of this section -- where both the set of distinct eigenvalues and the sequence of multiplicities varied with the parameter -- the seaweeds discussed in Section~\ref{sec:stable} have spectra whose sets of distinct eigenvalues exhibit a suprising stability property.

\section{Stability}\label{sec:stable}

In this section, we consider families of Frobenius, type-A seaweeds that are parametrized by the number of 2's (Section~\ref{sec:add2's}) and by the number of 4's (Section~\ref{sec:add4's}) in the defining compositions. For such families, we find that the sets of distinct eigenvalues in the spectra stabilize. At the end of Section~\ref{sec:add4's}, we conjecture that such stabilization occurs among more general families as well. Note that this behavior stands in sharp contrast to that of the Frobenius seaweeds $\mathfrak{p}^A\frac{k|1}{k+1}$, $\mathfrak{p}^A\frac{k|2}{k+2}$, $\mathfrak{p}^A\frac{k+1|k}{2k+1}$, $\mathfrak{p}^A\frac{k+2|k}{2k+2}$, $\mathfrak{p}^A\frac{2k|1}{1|2k}$, and $\mathfrak{p}^A\frac{2k|1|1}{2k+2}$, where the number of eigenvalues contained in the spectrum strictly increased with $k$.

\begin{remark}
We do not consider families of Frobenius, type-A seaweeds parametrized by the number of occurrences of an odd integer in the defining compositions because such families do not exist. For a given type-A seaweed $\mathfrak{g}$, each odd part in the defining composition of $\mathfrak{g}$ contributes a vertex of degree 1 to $M(\mathfrak{g})$. Consequently, if $\mathfrak{g}$ is to be Frobenius, i.e., if $M(\mathfrak{g})$ consists of a single path, then $\mathfrak{g}$ can have at most two odd parts in its defining composition.
\end{remark}

\subsection{Parameterized by number of 2's}\label{sec:add2's}

In this subsection, we consider families of Frobenius, type-A seaweeds which are parameterized by the number of 2's in the defining compositions. More specifically, we determine the spectra of type-A seaweeds of the form $$\mathfrak{p}^A\frac{k|2|\cdots|2}{k+1|2|\cdots|2|1},\quad\quad\mathfrak{p}^A\frac{k|2|\cdots|2|1}{k+1|2|\cdots|2},\quad\quad\text{and}\quad\quad \mathfrak{p}^A\frac{2|\cdots|2|1}{2r+1}.$$

To start, we establish a general result concerning the relationship between the spectra of the Frobenius, type-A seaweeds $$\mathfrak{g}=\mathfrak{p}^A\frac{a_1|\hdots|a_m|1}{b_1|\hdots|b_t},\quad \mathfrak{g}'=\mathfrak{p}^A\frac{a_1|\hdots|a_m|\overbrace{2|\cdots|2}^r}{b_1|\hdots|b_t|\underbrace{2|\cdots|2}_{r-1}|1},\quad\text{and}\quad \mathfrak{g}''=\mathfrak{p}^A\frac{a_1|\hdots|a_m|\overbrace{2|\cdots|2}^r|1}{b_1|\hdots|b_t|\underbrace{2|\cdots|2}_{r}},$$ for $r\ge 1$. For example, taking $\mathfrak{g}=\frac{3|1}{4}$ and $r=1$, we have $\mathfrak{g}'=\frac{3|2}{4|1}$ and  $\mathfrak{g}''=\frac{3|2|1}{4|2}$. The oriented meanders of these type-A seaweeds are illustrated in Figure~\ref{fig:2s}.

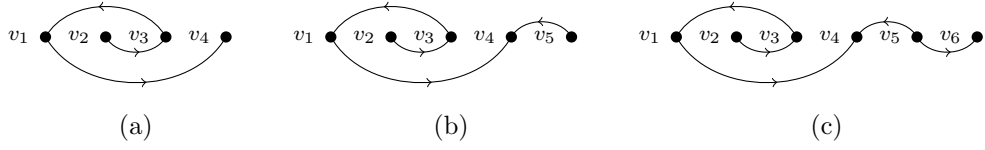
\begin{figure}[H]
    \centering
    \begin{tikzpicture}[scale=0.8]
\def\Node{\node [circle, fill, inner sep=1.5pt]}
\tikzset{->-/.style={decoration={
  markings,
  mark=at position .55 with {\arrow{>}}},postaction={decorate}}}
  \Node[label=left:\footnotesize {$v_1$}] at (0,0) {};
  \Node[label=left:\footnotesize {$v_2$}] at (1,0) {};
  \Node[label=left:\footnotesize {$v_3$}] at (2,0) {};
  \Node[label=left:\footnotesize {$v_4$}] at (3,0) {};
  \draw[->-] (0,0) to[bend right=60] (3,0);
  \draw[->-] (1,0) to[bend right=60] (2,0);
  \draw[->-] (2,0) to[bend right=60] (0,0);
  \node at (1.5, -1.5) {(a)};
\end{tikzpicture}\quad\quad \begin{tikzpicture}[scale=0.8]
\def\Node{\node [circle, fill, inner sep=1.5pt]}
\tikzset{->-/.style={decoration={
  markings,
  mark=at position .55 with {\arrow{>}}},postaction={decorate}}}
  \Node[label=left:\footnotesize {$v_1$}] at (0,0) {};
  \Node[label=left:\footnotesize {$v_2$}] at (1,0) {};
  \Node[label=left:\footnotesize {$v_3$}] at (2,0) {};
  \Node[label=left:\footnotesize {$v_4$}] at (3,0) {};
  \Node[label=left:\footnotesize {$v_5$}] at (4,0) {};
  \draw[->-] (0,0) to[bend right=60] (3,0);
  \draw[->-] (1,0) to[bend right=60] (2,0);
  \draw[->-] (2,0) to[bend right=60] (0,0);
  \draw[->-] (4,0) to[bend right=60] (3,0);
  \node at (2, -1.5) {(b)};
\end{tikzpicture}\quad\quad \begin{tikzpicture}[scale=0.8]
\def\Node{\node [circle, fill, inner sep=1.5pt]}
\tikzset{->-/.style={decoration={
  markings,
  mark=at position .55 with {\arrow{>}}},postaction={decorate}}}
  \Node[label=left:\footnotesize {$v_1$}] at (0,0) {};
  \Node[label=left:\footnotesize {$v_2$}] at (1,0) {};
  \Node[label=left:\footnotesize {$v_3$}] at (2,0) {};
  \Node[label=left:\footnotesize {$v_4$}] at (3,0) {};
  \Node[label=left:\footnotesize {$v_5$}] at (4,0) {};
  \Node[label=left:\footnotesize {$v_6$}] at (5,0) {};
  \draw[->-] (0,0) to[bend right=60] (3,0);
  \draw[->-] (1,0) to[bend right=60] (2,0);
  \draw[->-] (2,0) to[bend right=60] (0,0);
  \draw[->-] (4,0) to[bend right=60] (3,0);
  \draw[->-] (4,0) to[bend right=60] (5,0);
  \node at (2.5, -1.5) {(c)};
\end{tikzpicture}
    \caption{Oriented meanders of (a) $\mathfrak{p}^A\frac{3|1}{4}$, (b) $\mathfrak{p}^A\frac{3|2}{4|1}$, and (c) $\mathfrak{p}^A\frac{3|2|1}{4|2}$}
    \label{fig:2s}
\end{figure}

\begin{lemma}\label{lem:add2s}
If $\mathfrak{g}=\mathfrak{p}^A\frac{a_1|\hdots|a_m|1}{b_1|\hdots|b_t}$ is a Frobenius, type-A seaweed with spectrum $S$, then $\mathfrak{g}'=\mathfrak{p}^A\frac{a_1|\hdots|a_m|2}{b_1|\hdots|b_t|1}$ is a Frobenius, type-A seaweed with spectrum $S\cup\{0,1\}$.
\end{lemma}
\begin{proof}
Let $n=1+\sum_{i=1}^ma_i=\sum_{j=1}^tb_j$. Evidently, $\overrightarrow{M}(\mathfrak{g}')$ is equal to $\overrightarrow{M}(\mathfrak{g})$ with the addition of a vertex $v_{n+1}$ and a directed edge $(v_{n+1},v_n)$. Considering Corollary~\ref{cor:indA}, it follows that $\mathfrak{g}'$ is Frobenius. Now, setting $a_0=b_0=0$, note that the spectrum of $\mathfrak{g}$ is the multiset \begin{align*}
S=\bigcup_{k=1}^m&\left\{w(P_{i,j}(\mathfrak{g}))~|~1+\sum_{l=0}^{k-1}a_{l}\le j< i\le \sum_{l=1}^ka_l\right\} \\
    &\cup \bigcup_{k=1}^t\left\{w(P_{i,j}(\mathfrak{g}))~|~1+\sum_{l=0}^{k-1}b_{l}\le i< j\le \sum_{l=1}^kb_l\right\} \cup\{0^{n-1}\},
    \end{align*}
    and the spectrum of $\mathfrak{g}'$ is the multiset
\begin{align*}
    S'=\bigcup_{k=1}^m&\left\{w(P_{i,j}(\mathfrak{g}'))~|~1+\sum_{l=0}^{k-1}a_{l}\le j< i\le \sum_{l=1}^ka_l\right\} \\
    &\cup \bigcup_{k=1}^t\left\{w(P_{i,j}(\mathfrak{g}'))~|~1+\sum_{l=0}^{k-1}b_{l}\le i< j\le \sum_{l=1}^kb_l\right\} \cup\{0^{n},w(P_{n+1,n}(\mathfrak{g}'))\}.
\end{align*}
Given the relationship between $\overrightarrow{M}(\mathfrak{g}')$ and $\overrightarrow{M}(\mathfrak{g})$  outlined above, it follows that 
$$S'=S\cup\{0,w(P_{n+1,n}(\mathfrak{g}'))\}=S\cup\{0,1\}.$$
\end{proof}

\begin{theorem}\label{thm:addm2s}
If $\mathfrak{g}=\mathfrak{p}^A\frac{a_1|\hdots|a_m|1}{b_1|\hdots|b_t}$ is a Frobenius, type-A seaweed with spectrum $S$, then
\begin{enumerate}
    \item[\textup{(1)}] $$\mathfrak{g}'=\mathfrak{p}^A\frac{a_1|\hdots|a_m|\overbrace{2|\cdots|2}^r}{b_1|\hdots|b_t|\underbrace{2|\cdots|2}_{r-1}|1},$$ for $r\ge 1$, is a Frobenius, type-A seaweed with spectrum $S\cup\{0^{2r-1},1^{2r-1}\}$; and
    \item[\textup{(2)}] $$\mathfrak{g}''=\mathfrak{p}^A\frac{a_1|\hdots|a_m|\overbrace{2|\cdots|2}^r|1}{b_1|\hdots|b_t|\underbrace{2|\cdots|2}_{r}},$$ for $r\ge 1$, is a Frobenius, type-A seaweed with spectrum $S\cup\{0^{2r},1^{2r}\}$.
\end{enumerate}
\end{theorem}
\begin{proof}
By induction on $r$. An application of Lemma~\ref{lem:add2s} establishes the result for $r=1$ in (1). Then applying Corollary~\ref{cor:vfes}, followed by Lemma~\ref{lem:add2s}, and then Corollary~\ref{cor:vfes} again, to $$\mathfrak{p}^A\frac{a_1|\hdots|a_m|2}{b_1|\hdots|b_t|1}$$ establishes the case $r=1$ for (2). Assume the result holds for $r-1\ge 1$. In particular, assume that the algebra $$\mathfrak{p}^A\frac{a_1|\hdots|a_m|\overbrace{2|\cdots|2}^{r-1}|1}{b_1|\hdots|b_t|\underbrace{2|\cdots|2}_{r-1}}$$ is Frobenius with spectrum equal to $S\cup \{0^{2r-2},1^{2r-2}\}$. Applying Lemma~\ref{lem:add2s}, we find that the algebra $$\mathfrak{g}'=\mathfrak{p}^A\frac{a_1|\hdots|a_m|\overbrace{2|\cdots|2}^r}{b_1|\hdots|b_t|\underbrace{2|\cdots|2}_{r-1}|1}$$ is Frobenius with spectrum equal to $S\cup \{0^{2r-1},1^{2r-1}\}$, as desired. Now, applying Corollary~\ref{cor:vfes} followed by Lemma~\ref{lem:add2s} and then a second application of Corollary~\ref{cor:vfes} to $\mathfrak{g}'$, we find that the algebra $$\mathfrak{g}''=\mathfrak{p}^A\frac{a_1|\hdots|a_m|\overbrace{2|\cdots|2}^r|1}{b_1|\hdots|b_t|\underbrace{2|\cdots|2}_{r}}$$ is Frobenius with spectrum equal to $S\cup\{0^{2r},1^{2r}\}$, as desired. The result follows by induction.
\end{proof}

Considering the spectrum formulas of Theorem~\ref{thm:addm2s}, we are immediately led to the following.

\begin{corollary}\label{cor:222}
Let $\mathfrak{g}=\mathfrak{p}^A\frac{a_1|\hdots|a_m|1}{b_1|\hdots|b_t}$ be a Frobenius, type-A seaweed with the unimodal spectrum property. If
$$\mathfrak{g}'=\mathfrak{p}^A\frac{a_1|\hdots|a_m|\overbrace{2|\cdots|2}^r}{b_1|\hdots|b_t|\underbrace{2|\cdots|2}_{r-1}|1}\quad\text{and}\quad \mathfrak{g}''=\mathfrak{p}^A\frac{a_1|\hdots|a_m|\overbrace{2|\cdots|2}^r|1}{b_1|\hdots|b_t|\underbrace{2|\cdots|2}_{r}},$$ then $\mathfrak{g}'$ and $\mathfrak{g}''$ have the unimodal spectrum property.
\end{corollary}

It is important to note that ``unimodal" cannot be strengthened to ``log-concave" in the conclusion of Corollary~\ref{cor:222}. A counterexample is provided by the following theorem (see Remark~\ref{rem:nlcc}) which is a corollary of Theorems~\ref{thm:base1} and \ref{thm:addm2s}.

\begin{theorem}\label{thm:base1ext}
\begin{enumerate}
    \item[\textup{(1)}] For $k\ge 1$, if $$\mathfrak{g}_k=\mathfrak{p}^A\frac{k|\overbrace{2|\cdots|2}^r}{k+1|\underbrace{2|\cdots|2}_{r-1}|1},$$ then $\mathfrak{g}_k$ has the unimodal spectrum property with spectrum equal to 
$$\{0^{k+2r-1},1^{k+2r-1}\}\cup\bigcup_{i=1}^{k-1}\{(-k+i)^i,(k-i)^i\}.$$
\item[\textup{(2)}] For $k\ge 1$, if $$\mathfrak{g}_k=\mathfrak{p}^A\frac{k|\overbrace{2|\cdots|2}^r|1}{k+1|\underbrace{2|\cdots|2}_{r}},$$ then $\mathfrak{g}_k$ has the unimodal spectrum property with spectrum equal to $$\{0^{k+2r},1^{k+2r}\}\cup\bigcup_{i=1}^{k-1}\{(-k+i)^i,(k-i)^i\}.$$
\end{enumerate}
\end{theorem}

\begin{remark}\label{rem:nlcc}
    Utilizing Theorem~\ref{thm:base1ext}, we can construct examples of Frobenius, type-A seaweeds which do not have the log-concave spectrum property. For example, let $\mathfrak{g}=\mathfrak{p}^A\frac{3|1}{4}$ and $\mathfrak{g}'=\mathfrak{p}^A\frac{3|2|2}{4|2|1}.$ Recall from Theorem~\ref{thm:base1} that the spectrum of $\mathfrak{g}$ is $\{-2,-1^2,0^3,1^3,2^2,3\}$. On the other hand, by Theorem~\ref{thm:addm2s}, we find that the spectrum of $\mathfrak{g}'$ is $\{-2,-1^2,0^6,1^6,2^2,3\}.$ Clearly, $\mathfrak{g}'$ does not have the log-concave spectrum property.
\end{remark}

\begin{remark}
Note that for fixed $k$ and varying values of $r$, the collection of distinct eigenvalues in the spectra of the Frobenius, type-A seaweeds considered in Theorem~\ref{thm:base1ext} is fixed.
\end{remark}

\begin{remark}\label{rem:posets}
Note that one can also apply Corollary~\ref{cor:222} to the Frobenius, type-A seaweed $\mathfrak{p}^A\frac{1|1}{2}$ to determine the spectrum of type-A seaweeds of the form $\mathfrak{g}_r=\mathfrak{p}^A\frac{1|2|\cdots|2}{2|\cdots|2|1}$, where $r\geq 1$ is the number of $2$'s in the numerator and denominator. See Figure~\ref{fig:21over12} for some example meanders of such type-A seaweeds. In particular, it is straightforward to prove that the spectrum of $\mathfrak{g}_r$ is $\{0^{2r},1^{2r}\}$. Now, one can show that type-A seaweeds of the form $\mathfrak{p}^A\frac{1|2|\cdots|2}{2|\cdots|2|1}$ are related to algebras of another combinatorially defined family of Lie subalgebras of $\mathfrak{sl}(n)$. In particular, $\mathfrak{p}^A\frac{1|2|\cdots|2}{2|\cdots|2|1}$ is isomorphic to a type-A Lie poset algebra \textup(see \textup{\textbf{\cite{seriesA}}}\textup). Type-A Lie poset algebras are subalgebras of $\mathfrak{sl}(n)$ parametrized by posets on $\{1,\hdots,n\}$. It is known that the spectra of Frobenius, type-A Lie poset algebras corresponding to posets with chains of cardinality at most two must consist of an equal number of 0's and 1's \textup(see \textup{\textbf{\cite{Binary}}}\textup). It is conjectured that this is true in general for Frobenius Lie poset algebras \textup(see \textup{\textbf{\cite{Binary,BCD}}}\textup).
\end{remark}

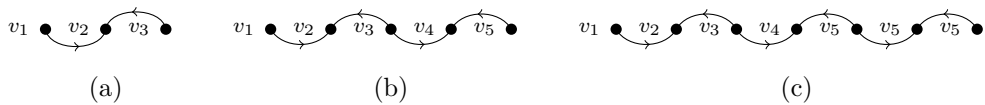
\begin{figure}[H]
    \centering
    \begin{tikzpicture}[scale=0.8]
\def\Node{\node [circle, fill, inner sep=1.5pt]}
\tikzset{->-/.style={decoration={
  markings,
  mark=at position .55 with {\arrow{>}}},postaction={decorate}}}
  \Node[label=left:\footnotesize {$v_1$}] at (0,0) {};
  \Node[label=left:\footnotesize {$v_2$}] at (1,0) {};
  \Node[label=left:\footnotesize {$v_3$}] at (2,0) {};
  \draw[->-] (0,0) to[bend right=80] (1,0);
  \draw[->-] (2,0) to[bend right=80] (1,0);
  \node at (1, -1) {(a)};
\end{tikzpicture}\quad\quad \begin{tikzpicture}[scale=0.8]
\def\Node{\node [circle, fill, inner sep=1.5pt]}
\tikzset{->-/.style={decoration={
  markings,
  mark=at position .55 with {\arrow{>}}},postaction={decorate}}}
  \Node[label=left:\footnotesize {$v_1$}] at (0,0) {};
  \Node[label=left:\footnotesize {$v_2$}] at (1,0) {};
  \Node[label=left:\footnotesize {$v_3$}] at (2,0) {};
  \Node[label=left:\footnotesize {$v_4$}] at (3,0) {};
  \Node[label=left:\footnotesize {$v_5$}] at (4,0) {};
  \draw[->-] (2,0) to[bend right=60] (1,0);
  \draw[->-] (4,0) to[bend right=60] (3,0);
  \draw[->-] (0,0) to[bend right=60] (1,0);
  \draw[->-] (2,0) to[bend right=60] (3,0);
  \node at (2, -1) {(b)};
\end{tikzpicture}\quad\quad \begin{tikzpicture}[scale=0.8]
\def\Node{\node [circle, fill, inner sep=1.5pt]}
\tikzset{->-/.style={decoration={
  markings,
  mark=at position .55 with {\arrow{>}}},postaction={decorate}}}
  \Node[label=left:\footnotesize {$v_1$}] at (0,0) {};
  \Node[label=left:\footnotesize {$v_2$}] at (1,0) {};
  \Node[label=left:\footnotesize {$v_3$}] at (2,0) {};
  \Node[label=left:\footnotesize {$v_4$}] at (3,0) {};
  \Node[label=left:\footnotesize {$v_5$}] at (4,0) {};
  \Node[label=left:\footnotesize {$v_5$}] at (5,0) {};
  \Node[label=left:\footnotesize {$v_5$}] at (6,0) {};
  \draw[->-] (2,0) to[bend right=60] (1,0);
  \draw[->-] (4,0) to[bend right=60] (3,0);
  \draw[->-] (6,0) to[bend right=60] (5,0);
  \draw[->-] (0,0) to[bend right=60] (1,0);
  \draw[->-] (2,0) to[bend right=60] (3,0);
  \draw[->-] (4,0) to[bend right=60] (5,0);
  \node at (3, -1) {(c)};
\end{tikzpicture}
    \caption{Oriented meanders of (a) $\mathfrak{p}^A\frac{1|2}{2|1}$, (b) $\mathfrak{p}^A\frac{1|2|2}{2|2|1}$, and (c) $\mathfrak{p}^A\frac{1|2|2|2}{2|2|2|1}$}
    \label{fig:21over12}
\end{figure}

\begin{remark}
    Additionally, one can apply Theorem~\ref{thm:addm2s} to the Frobenius, type-A seaweeds of Theorems~\ref{thm:2k1/12k} and~\ref{thm:2k11/2k2}. We leave such applications as exercises for interested readers.
\end{remark}

We now consider Frobenius, type-A seaweeds of the form $\mathfrak{p}^A\frac{2|\dots|2|1}{2r+1}.$ Although such algebras are parametrized by the number of 2's in their defining compositions, they are parabolic and so descriptions of their spectra do not follow from Corollary~\ref{cor:222}. See Figure~\ref{fig:2...21} for some example meanders of such type-A seaweeds.

\begin{figure}[H]
    \centering
    \begin{tikzpicture}[scale=0.8]
\def\Node{\node [circle, fill, inner sep=1.5pt]}
\tikzset{->-/.style={decoration={
  markings,
  mark=at position .55 with {\arrow{>}}},postaction={decorate}}}
  \Node[label=left:\footnotesize {$v_1$}] at (0,0) {};
  \Node[label=left:\footnotesize {$v_2$}] at (1,0) {};
  \Node[label=left:\footnotesize {$v_3$}] at (2,0) {};
  \draw[->-] (1,0) to[bend right=60] (0,0);
  \draw[->-] (0,0) to[bend right=80] (2,0);
  \node at (1, -2.5) {(a)};
\end{tikzpicture}\quad\quad \begin{tikzpicture}[scale=0.8]
\def\Node{\node [circle, fill, inner sep=1.5pt]}
\tikzset{->-/.style={decoration={
  markings,
  mark=at position .55 with {\arrow{>}}},postaction={decorate}}}
  \Node[label=left:\footnotesize {$v_1$}] at (0,0) {};
  \Node[label=left:\footnotesize {$v_2$}] at (1,0) {};
  \Node[label=left:\footnotesize {$v_3$}] at (2,0) {};
  \Node[label=left:\footnotesize {$v_4$}] at (3,0) {};
  \Node[label=left:\footnotesize {$v_5$}] at (4,0) {};
  \draw[->-] (0,0) to[bend right=60] (4,0);
  \draw[->-] (1,0) to[bend right=60] (3,0);
  \draw[->-] (1,0) to[bend right=60] (0,0);
  \draw[->-] (3,0) to[bend right=60] (2,0);
  \node at (2, -2.5) {(b)};
\end{tikzpicture}\quad\quad \begin{tikzpicture}[scale=0.8]
\def\Node{\node [circle, fill, inner sep=1.5pt]}
\tikzset{->-/.style={decoration={
  markings,
  mark=at position .55 with {\arrow{>}}},postaction={decorate}}}
  \Node[label=left:\footnotesize {$v_1$}] at (0,0) {};
  \Node[label=left:\footnotesize {$v_2$}] at (1,0) {};
  \Node[label=left:\footnotesize {$v_3$}] at (2,0) {};
  \Node[label=left:\footnotesize {$v_4$}] at (3,0) {};
  \Node[label=left:\footnotesize {$v_5$}] at (4,0) {};
  \Node[label=left:\footnotesize {$v_5$}] at (5,0) {};
  \Node[label=left:\footnotesize {$v_5$}] at (6,0) {};
  \draw[->-] (0,0) to[bend right=60] (6,0);
  \draw[->-] (1,0) to[bend right=60] (5,0);
  \draw[->-] (2,0) to[bend right=60] (4,0);
  \draw[->-] (1,0) to[bend right=60] (0,0);
  \draw[->-] (3,0) to[bend right=60] (2,0);
  \draw[->-] (5,0) to[bend right=60] (4,0);
  \node at (3, -2.5) {(c)};
\end{tikzpicture}
    \caption{Oriented meanders of (a) $\mathfrak{p}^A\frac{2|1}{3}$, (b) $\mathfrak{p}^A\frac{2|2|1}{5}$, and (c) $\mathfrak{p}^A\frac{2|2|2|1}{7}$}
    \label{fig:2...21}
\end{figure}
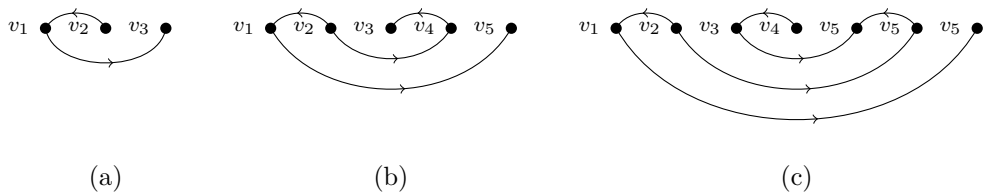

To determine the form of the spectrum of the Frobenius, type-A seaweed $\mathfrak{p}^A\frac{2|\cdots|2|1}{2r+1}$, we require the following lemma.

\begin{lemma}\label{lem:2...21}
Let $\mathfrak{g}_r=\mathfrak{p}^A\frac{2|\cdots|2|1}{2r+1}$, for $r\ge 1$. Then $$\{w(P_{i,2r+1}(\mathfrak{g}_r))~|~1\le i< 2r+1\}=\{2^{\lceil\frac{r}{2}\rceil},1^r,0^{\lceil\frac{r-1}{2}\rceil}\}.$$
\end{lemma}
\begin{proof}
By induction on $r$. For $r=1$, we can compute directly from $\overrightarrow{M}(\mathfrak{g}_1)$ (see Figure~\ref{fig:2...21}(a)) that $$w(P_{1,3}(\mathfrak{g}_1))=1\quad\text{and}\quad w(P_{2,3}(\mathfrak{g}_1))=2.$$ Thus, $$\{w(P_{i,3}(\mathfrak{g}_1))~|~1\le i< 3\}=\{2,1\}=\{2^{\lceil\frac{1}{2}\rceil},1^1,0^{\lceil\frac{1-1}{2}\rceil}\}.$$ Now, assume that the result holds for $r-1\ge 1$. Let $\mathfrak{p}_{r-1}=\mathfrak{p}^A\frac{1|2|\cdots|2}{2r-1}$. Note that removing the edges $(v_1,v_{2r+1})$ and $(v_2,v_1)$ and the vertices $v_1$ and $v_{2r+1}$ from $\overrightarrow{M}(\mathfrak{g}_r)$ yields $\overrightarrow{M}(\mathfrak{p}_{r-1}),$ with $v_{i+1}$ in place of $v_i$, for $1\le i\le 2r-1$. Thus,
\begin{align*}
    w(P_{i,2r+1}(\mathfrak{g}_r))&=w(P_{i,2}(\mathfrak{g}_r))+w(P_{2,1}(\mathfrak{g}_r))+w(P_{1,2r+1}(\mathfrak{g}_r)) \\
    &=w(P_{i-1,1}(\mathfrak{p}_{r-1}))+2,
\end{align*}
for $2\le i\le 2r$. Applying Lemma~\ref{lem:hflip}, it follows that $$w(P_{i,2r+1}(\mathfrak{g}_k))=w(P_{i-1,1}(\mathfrak{p}_{r-1}))+2=-w(P_{2r-i+1,2r-1}(\mathfrak{g}_{r-1}))+2,$$ for $2\le i\le 2r$. Therefore, applying the induction hypothesis, $$\{w(P_{i,2r+1}(\mathfrak{g}_r))~|~3\le i\le 2r\}=\{-w(P_{i,2r-1}(\mathfrak{g}_{r-1}))+2~|~1\le i< 2r-1\}=\{0^{\lceil\frac{r-1}{2}\rceil},1^{r-1},2^{\lceil\frac{r-2}{2}\rceil}\}.$$ As $w(P_{1,2r+1}(\mathfrak{g}_r))=1$ and $w(P_{2,2r+1}(\mathfrak{g}_r))=2$, it follows that $$\{w(P_{i,2r+1}(\mathfrak{g}_r))~|~1\le i< 2r+1\}=\{0^{\lceil\frac{r-1}{2}\rceil},1^{r},2^{\lceil\frac{r}{2}\rceil}\},$$ as desired.
\end{proof}

\begin{theorem}\label{thm:221}
Let $\mathfrak{g}_r=\mathfrak{p}^A\frac{2|\cdots|2|1}{2r+1}$, for $r\ge 1$. There exists positive integers $a_r>b_r$ such that the spectrum of $\mathfrak{g}_r$ is given by $$\{-1^{b_r},0^{a_r},1^{a_r},2^{b_r}\}.$$
\end{theorem}
\begin{proof}
By induction on $r$. For $r=1$, calculating directly using $\overrightarrow{M}(\mathfrak{g}_1)$ (see Figure~\ref{fig:2...21}(a)) we find that the spectrum of $\mathfrak{g}_1$ is $$\{-1^1,0^2,1^2,2^1\}.$$ Assume the result holds for $r-1\ge 1$. We break the spectrum of $\mathfrak{g}_r$ into three groups: $$G_1=\{w(P_{i,j}(\mathfrak{g}_r))~|~2\le i\le j\le 2r\}\cup\{w(P_{2j,2j-1}(\mathfrak{g}_r))~|~1\le j\le r\},$$ $$G_2=\{w(P_{i,2r+1}(\mathfrak{g}_r))~|~1\le i< 2r+1\},$$ and $$G_3=\{w(P_{1,i}(\mathfrak{g}_r))~|~1\le i\le 2r\}.$$  Let $\mathfrak{p}_{r-1}=\mathfrak{p}^A\frac{1|2|\cdots|2}{2r-1}$. Note that removing the edges $(v_1,v_{2r+1})$ and $(v_2,v_1)$ and the vertices $v_1$ and $v_{2r+1}$ from $\overrightarrow{M}(\mathfrak{g}_r)$ yields $\overrightarrow{M}(\mathfrak{p}_{r-1})$ with $v_{i+1}$ in place of $v_i$, for $1\le i\le 2r-1$. Thus, $$w(P_{i,j}(\mathfrak{g}_r))=w(P_{i-1,j-1}(\mathfrak{p}_{r-1})),$$ for $2\le i\le j\le 2r$, and $$w(P_{2j,2j-1}(\mathfrak{g}_r))=w(P_{2j-1,2j-2}(\mathfrak{p}_{r-1})),$$ for $2\le j\le r$. Consequently, letting $S$ denote the spectrum of of $\mathfrak{p}_{r-1}$, we have that
$$G_1=S\cup \{w(P_{2r,2r}(\mathfrak{g}_r)),w(P_{2,1}(\mathfrak{g}_r))\}=S\cup \{0,1\}.$$ Applying Corollary~\ref{cor:hfes}, the spectrum of $\mathfrak{g}_{r-1}$ is also $S$; that is, applying our induction hypothesis, $$G_1=\{-1^{b_{r-1}},0^{a_{r-1}+1},1^{a_{r-1}+1},2^{b_{r-1}}\}.$$ Now, considering Lemma~\ref{lem:2...21}, we have that $$G_2=\{2^{\lceil\frac{r}{2}\rceil},1^r,0^{\lceil\frac{r-1}{2}\rceil}\}.$$ As for $G_3$, note that
\begin{align*}
    w(P_{1,i}(\mathfrak{g}_r))&=-w(P_{i,1}(\mathfrak{g}_r)) \\
    &=-(w(P_{i,2r+1}(\mathfrak{g}_r))-w(P_{1,2r+1}(\mathfrak{g}_r))) \\
    &=-w(P_{i,2r+1}(\mathfrak{g}_r))+1,
\end{align*}
for $1\le i\le 2r$. Thus, applying Lemma~\ref{lem:2...21}, it follows that $$G_3=\{-1^{\lceil\frac{r}{2}\rceil},0^r,1^{\lceil\frac{r-1}{2}\rceil}\}.$$ Putting everything together, we find that the spectrum of $\mathfrak{g}_r$ is given by $$\{-1^{b_{r-1}+\lceil\frac{r}{2}\rceil},0^{a_{r-1}+r+\lceil\frac{r-1}{2}\rceil+1},1^{a_{r-1}+r+\lceil\frac{r-1}{2}\rceil+1},2^{b_{r-1}+\lceil\frac{r}{2}\rceil}\}.$$ The result follows.
\end{proof}

Considering the spectrum formula of Theorem~\ref{thm:221}, we are immediately led to the following.

\begin{corollary}\label{cor:221}
For $r\ge 1$, if $\mathfrak{g}_r=\mathfrak{p}^A\frac{2|\cdots|2|1}{2r+1}$, then $\mathfrak{g}_r$ has the log-concave spectrum property.
\end{corollary}

\begin{remark}
    Combining Theorems~\ref{thm:addm2s}, ~\ref{thm:base1ext}, and~\ref{thm:221} and Corollaries~\ref{cor:221} and~\ref{cor:222} with Corollaries~\ref{cor:vfes},~\ref{cor:hfes}, and~\ref{cor:vhfes}, we obtain similar results for related families of Frobenius, type-A seaweeds.
\end{remark}

\subsection{Parameterized by number of 4's}\label{sec:add4's}

In this subsection, we consider families of Frobneius, type-A seaweeds which are parameterized by the number of 4's in the defining compositions. More specifically, we determine the spectra of type-A seaweeds of the form $$\mathfrak{p}^A\frac{k|4|\cdots|4}{k+2|4|\cdots|4|2}\quad\quad\text{and}\quad\quad \mathfrak{p}^A\frac{k|4|\cdots|4|2}{k+2|4|\cdots|4}.$$

Similar to Section~\ref{sec:add2's}, we start by establishing a general result concerning the relationship between the spectra of the Frobenius, type-A seaweeds $$\mathfrak{g}=\mathfrak{p}^A\frac{a_1|\hdots|a_m|2}{b_1|\hdots|b_t},\quad \mathfrak{g}'=\mathfrak{p}^A\frac{a_1|\hdots|a_m|\overbrace{4|\cdots|4}^r}{b_1|\hdots|b_t|\underbrace{4|\cdots|4}_{r-1}|2},\quad\text{and}\quad \mathfrak{g}''=\mathfrak{p}^A\frac{a_1|\hdots|a_m|\overbrace{4|\cdots|4}^r|2}{b_1|\hdots|b_t|\underbrace{4|\cdots|4}_{r}},$$ for $r\ge 1$. For example, taking $\mathfrak{g}=\frac{1|2}{3}$ and $r=1$, we have $\mathfrak{g}'=\frac{1|4}{3|2}$ and  $\mathfrak{g}''=\frac{1|4|2}{3|4}$. The oriented meanders of these type-A seaweeds are illustrated in Figure~\ref{fig:4s}.

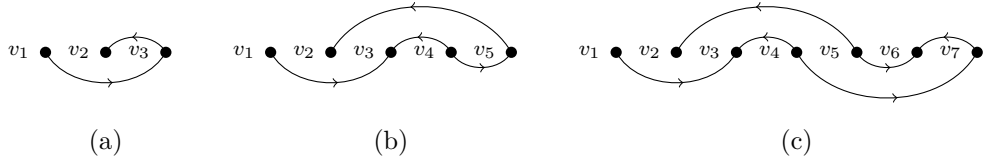
\begin{figure}[H]
    \centering
    \begin{tikzpicture}[scale=0.8]
\def\Node{\node [circle, fill, inner sep=1.5pt]}
\tikzset{->-/.style={decoration={
  markings,
  mark=at position .55 with {\arrow{>}}},postaction={decorate}}}
  \Node[label=left:\footnotesize {$v_1$}] at (0,0) {};
  \Node[label=left:\footnotesize {$v_2$}] at (1,0) {};
  \Node[label=left:\footnotesize {$v_3$}] at (2,0) {};
  \draw[->-] (0,0) to[bend right=60] (2,0);
  \draw[->-] (2,0) to[bend right=60] (1,0);
  \node at (1, -1.5) {(a)};
\end{tikzpicture}\quad\quad  \begin{tikzpicture}[scale=0.8]
\def\Node{\node [circle, fill, inner sep=1.5pt]}
\tikzset{->-/.style={decoration={
  markings,
  mark=at position .55 with {\arrow{>}}},postaction={decorate}}}
  \Node[label=left:\footnotesize {$v_1$}] at (0,0) {};
  \Node[label=left:\footnotesize {$v_2$}] at (1,0) {};
  \Node[label=left:\footnotesize {$v_3$}] at (2,0) {};
   \Node[label=left:\footnotesize {$v_4$}] at (3,0) {};
  \Node[label=left:\footnotesize {$v_5$}] at (4,0) {};
  \draw[->-] (0,0) to[bend right=60] (2,0);
  \draw[->-] (4,0) to[bend right=60] (1,0);
  \draw[->-] (3,0) to[bend right=60] (2,0);
   \draw[->-] (3,0) to[bend right=60] (4,0);
  \node at (2, -1.5) {(b)};
\end{tikzpicture}\quad\quad \begin{tikzpicture}[scale=0.8]
\def\Node{\node [circle, fill, inner sep=1.5pt]}
\tikzset{->-/.style={decoration={
  markings,
  mark=at position .55 with {\arrow{>}}},postaction={decorate}}}
  \Node[label=left:\footnotesize {$v_1$}] at (0,0) {};
  \Node[label=left:\footnotesize {$v_2$}] at (1,0) {};
  \Node[label=left:\footnotesize {$v_3$}] at (2,0) {};
   \Node[label=left:\footnotesize {$v_4$}] at (3,0) {};
  \Node[label=left:\footnotesize {$v_5$}] at (4,0) {};
   \Node[label=left:\footnotesize {$v_6$}] at (5,0) {};
  \Node[label=left:\footnotesize {$v_7$}] at (6,0) {};
  \draw[->-] (0,0) to[bend right=60] (2,0);
  \draw[->-] (4,0) to[bend right=60] (1,0);
  \draw[->-] (3,0) to[bend right=60] (2,0);
   \draw[->-] (3,0) to[bend right=60] (6,0);
    \draw[->-] (4,0) to[bend right=60] (5,0);
  \draw[->-] (6,0) to[bend right=60] (5,0);
  \node at (3, -1.5) {(c)};
\end{tikzpicture}
    \caption{Oriented meanders of (a) $\mathfrak{p}^A\frac{1|2}{3}$, (b) $\mathfrak{p}^A\frac{1|4}{3|2}$, and (c) $\mathfrak{p}^A\frac{1|4|2}{3|4}$}
    \label{fig:4s}
\end{figure}

\begin{lemma}\label{lem:add4s}
If $\mathfrak{g}=\mathfrak{p}^A\frac{a_1|\hdots|a_m|2}{b_1|\hdots|b_t}$ is a Frobenius, type-A seaweed with spectrum $S$, then $\mathfrak{g}'=\mathfrak{p}^A\frac{a_1|\hdots|a_m|4}{b_1|\hdots|b_t|2}$ is a Frobenius, type-A seaweed with spectrum $S\cup\{-1,0^3,1^3,2\}$.
\end{lemma}
\begin{proof}
Let $n=2+\sum_{i=1}^ma_i=\sum_{j=1}^tb_j$. Evidently, $\overrightarrow{M}(\mathfrak{g}')$ is equal to $\overrightarrow{M}(\mathfrak{g})$ with the removal of the directed edge $(v_{n},v_{n-1})$, the addition of vertices $v_{n+1}$ and $v_{n+2}$, and the addition of directed edges $(v_{n+1},v_n)$, $(v_{n+2},v_{n-1})$, and $(v_{n+1},v_{n+2})$. Considering Corollary~\ref{cor:indA}, it follows that $\mathfrak{g}'$ is Frobenius. Now, setting $a_0=b_0=0$ and $a_{m+1}=2$, note that the spectrum of $\mathfrak{g}'$ is the multiset
\begin{align*}
    S'=\bigcup_{k=1}^{m+1}&\left\{w(P_{i,j}(\mathfrak{g}'))~|~1+\sum_{l=0}^{k-1}a_{l}\le j< i\le \sum_{l=1}^ka_l\right\} \\
    &\cup \bigcup_{k=1}^t\left\{w(P_{i,j}(\mathfrak{g}'))~|~1+\sum_{l=0}^{k-1}b_{l}\le i< j\le \sum_{l=1}^kb_l\right\} \\
    &\cup\{0^{n+1},w(P_{n+1,n-1}(\mathfrak{g}')),w(P_{n+1,n}(\mathfrak{g}')),w(P_{n+1,n+2}(\mathfrak{g}'))\}\\
    &\cup\{w(P_{n+2,n-1}(\mathfrak{g}')),w(P_{n+2,n}(\mathfrak{g}')),w(P_{n+2,n+1}(\mathfrak{g}'))\};
\end{align*}
that is, 
\begin{align*}
    S'=\bigcup_{k=1}^{m+1}&\left\{w(P_{i,j}(\mathfrak{g}'))~|~1+\sum_{l=0}^{k-1}a_{l}\le j< i\le \sum_{l=1}^ka_l\right\} \\
    &\cup \bigcup_{k=1}^t\left\{w(P_{i,j}(\mathfrak{g}'))~|~1+\sum_{l=0}^{k-1}b_{l}\le i< j\le \sum_{l=1}^kb_l\right\} \\
    &\cup\{-1,0^{n+2},1^3,2\}.
\end{align*}
We claim that $w(P_{i,j}(\mathfrak{g}'))=w(P_{i,j}(\mathfrak{g}))$, for $1\le i\neq j\le n$. To establish the claim, we break it into two cases.
\bigskip

\noindent
\textbf{Case 1:} $\{v_{n-1},v_n\}$ is not an edge in $P_{i,j}(\mathfrak{g})$. Considering the relationship between $\overrightarrow{M}(\mathfrak{g})$ and $\overrightarrow{M}(\mathfrak{g}')$ outlined above, we have that $P_{i,j}(\mathfrak{g})=P_{i,j}(\mathfrak{g}')$, and the claim follows. 
\bigskip

\noindent
\textbf{Case 2:} $\{v_{n-1},v_n\}$ is an edge in $P_{i,j}(\mathfrak{g})$. In this case $P_{i,j}(\mathfrak{g})$ can be decomposed into three (possibly trivial) subpaths as follows: either
\begin{itemize}
    \item $P_{i,j}(\mathfrak{g})$ is equal to the concatenation of the paths $P_{i,n-1}(\mathfrak{g})$, $P_{n-1,n}(\mathfrak{g})$, and $P_{n,j}(\mathfrak{g})$, or 
    \item $P_{i,j}(\mathfrak{g})$ is equal to the concatenation of the paths $P_{i,n}(\mathfrak{g})$, $P_{n,n-1}(\mathfrak{g})$, and $P_{n-1,j}(\mathfrak{g})$.
\end{itemize}
Assume that $P_{i,j}(\mathfrak{g})$ is equal to the concatenation of the paths $P_{i,n-1}(\mathfrak{g})$, $P_{n-1,n}(\mathfrak{g})$, and $P_{n,j}(\mathfrak{g})$; the other case follows via a similar argument. Considering the relationship between $\overrightarrow{M}(\mathfrak{g})$ and $\overrightarrow{M}(\mathfrak{g}')$ outlined above, it follows that $P_{i,j}(\mathfrak{g}')$ is the concatenation of the paths $P_{i,n-1}(\mathfrak{g}')$, $P_{n-1,n+2}(\mathfrak{g}')$, $P_{n+2,n+1}(\mathfrak{g}')$, $P_{n+1,n}(\mathfrak{g}')$, and $P_{n,j}(\mathfrak{g}')$, where $P_{i,n-1}(\mathfrak{g}'))=P_{i,n-1}(\mathfrak{g}))$ and $P_{n,j}(\mathfrak{g}'))=P_{n,j}(\mathfrak{g}))$. Thus,
\begin{align*}
    w(P_{i,j}(\mathfrak{g}'))&=w(P_{i,n-1}(\mathfrak{g}'))+w(P_{n-1,n+2}(\mathfrak{g}'))+w(P_{n+2,n+1}(\mathfrak{g}'))+w(P_{n+1,n}(\mathfrak{g}'))+w(P_{n,j}(\mathfrak{g}')) \\
    &=w(P_{i,n-1}(\mathfrak{g}'))-1-1+1+w(P_{n,j}(\mathfrak{g}')) \\
    &=w(P_{i,n-1}(\mathfrak{g}))-1+w(P_{n,j}(\mathfrak{g})) \\
    &=w(P_{i,n-1}(\mathfrak{g}))+w(P_{n-1,n}(\mathfrak{g}))+w(P_{n,j}(\mathfrak{g})) \\
    &=w(P_{i,j}(\mathfrak{g})),
\end{align*}
establishing the claim.
\bigskip

\noindent
Therefore,
\begin{align*}
    S'=\bigcup_{k=1}^{m+1}&\left\{w(P_{i,j}(\mathfrak{g}))~|~1+\sum_{l=0}^{k-1}a_{l}\le j< i\le \sum_{l=1}^ka_l\right\} \\
    &\cup \bigcup_{k=1}^t\left\{w(P_{i,j}(\mathfrak{g}))~|~1+\sum_{l=0}^{k-1}b_{l}\le i< j\le \sum_{l=1}^kb_l\right\} \\
    &\cup\{-1,0^{n+2},1^3,2\};
\end{align*}
that is, $$S'=S\cup\{-1,0^3,1^3,2\}.$$
\end{proof}

\begin{theorem}\label{thm:addm4s}
If $\mathfrak{g}=\mathfrak{p}^A\frac{a_1|\hdots|a_m|2}{b_1|\hdots|b_t}$ is a Frobenius, type-A seaweed with spectrum $S$, then
\begin{enumerate}
    \item[\textup{(1)}] $$\mathfrak{g}'=\mathfrak{p}^A\frac{a_1|\hdots|a_m|\overbrace{4|\cdots|4}^r}{b_1|\hdots|b_t|\underbrace{4|\cdots|4}_{r-1}|2},$$ for $r\ge 1$, is a Frobenius, type-A seaweed with spectrum $S\cup\{-1^{2r-1},0^{6r-3},1^{6r-3},2^{2r-1}\}$; and
    \item[\textup{(2)}] $$\mathfrak{g}'=\mathfrak{p}^A\frac{a_1|\hdots|a_m|\overbrace{4|\cdots|4}^r|2}{b_1|\hdots|b_t|\underbrace{4|\cdots|4}_{r}},$$ for $r\ge 1$, is a Frobenius, type-A seaweed with spectrum $S\cup\{-1^{2r},0^{6r},1^{6r},2^{2r}\}$.
\end{enumerate}
\end{theorem}
\begin{proof}
By induction on $r$. An application of Lemma~\ref{lem:add4s} establishes the result for $r=1$ in (1). Then applying Corollary~\ref{cor:vfes}, followed by Lemma~\ref{lem:add4s}, and then Corollary~\ref{cor:vfes} again, to $$\mathfrak{p}^A\frac{a_1|\hdots|a_m|4}{b_1|\hdots|b_t|2}$$ establishes the case $r=1$ for (2). Assume the result holds for $r-1\ge 1$. In particular, assume that the algebra $$\mathfrak{p}^A\frac{a_1|\hdots|a_m|\overbrace{4|\cdots|4}^{r-1}|2}{b_1|\hdots|b_t|\underbrace{4|\cdots|4}_{r-1}}$$ is Frobenius with spectrum equal to $S\cup \{-1^{2r-2},0^{6r-6},1^{6r-6},2^{2r-2}\}$. Applying Lemma~\ref{lem:add4s}, we find that the algebra $$\mathfrak{g}'=\mathfrak{p}^A\frac{a_1|\hdots|a_m|\overbrace{4|\cdots|4}^r}{b_1|\hdots|b_t|\underbrace{4|\cdots|4}_{r-1}|2}$$ is Frobenius with spectrum equal to $S\cup \{-1^{2r-1},0^{6r-3},1^{6r-3},2^{2r-1}\}$, as desired. Now, applying Corollary~\ref{cor:vfes}, followed by Lemma~\ref{lem:add4s}, and then Corollary~\ref{cor:vfes} again, to $\mathfrak{g}'$, we find that the algebra $$\mathfrak{g}''=\mathfrak{p}^A\frac{a_1|\hdots|a_m|\overbrace{4|\cdots|4}^r|2}{b_1|\hdots|b_t|\underbrace{4|\cdots|4}_{r}}$$ is Frobenius with spectrum equal to $S\cup\{-1^{2r},0^{6r},1^{6r},2^{2r}\}$, as desired. The result follows by induction.
\end{proof}

Considering the spectrum formula of Theorem~\ref{thm:addm4s}, we are immediately led to the following.

\begin{corollary}\label{cor:444}
Let $\mathfrak{g}=\mathfrak{p}^A\frac{a_1|\hdots|a_m|2}{b_1|\hdots|b_t}$ be a Frobenius, type-A seaweed with the unimodal spectrum property. If
$$\mathfrak{g}'=\mathfrak{p}^A\frac{a_1|\hdots|a_m|\overbrace{4|\cdots|4}^r}{b_1|\hdots|b_t|\underbrace{4|\cdots|4}_{r-1}|2}\quad\text{and}\quad \mathfrak{g}''=\mathfrak{p}^A\frac{a_1|\hdots|a_m|\overbrace{4|\cdots|4}^r|2}{b_1|\hdots|b_t|\underbrace{4|\cdots|4}_{r}},$$ then $\mathfrak{g}'$ and $\mathfrak{g}''$ have the unimodal spectrum property.
\end{corollary}

As in Corollary~\ref{cor:222}, ``unimodal" cannot be strengthened to ``log-concave" in; the conclusion of Corollary~\ref{cor:444}. A counterexample is provided by the following theorem (see Remark~\ref{rem:nlcc2}) which is a corollary of Theorems~\ref{thm:k2} and \ref{thm:addm4s}.

\begin{theorem}\label{thm:base2ext} Let $k=2m-1\ge 1.$
\begin{enumerate}
    \item[\textup{(1)}] For $k=2m-1\ge 1$, if $$\mathfrak{g}_{k,r}=\mathfrak{p}^A\frac{k|\overbrace{4|\cdots|4}^r}{k+2|\underbrace{4|\cdots|4}_{r-1}|2},$$ then $\mathfrak{g}_{k,r}$ has the unimodal spectrum property with spectrum equal to $\{-1^{2r},0^{6r-1},1^{6r-1},2^{2r}\}$ if $k=1$, $\{-2,-1^{2r+2},0^{6r+2},1^{6r+2},2^{2r+2},3\}$ if $k=3$, and
    \begin{align*}
        \{-m,(-m+1)^3&,-1^{2(k+r)-5},0^{2(k+3r)-4},1^{2(k+3r)-4},2^{2(k+r)-5},m^3,m+1\}\\
        &\cup\bigcup_{i=2}^{m-2}\{(-m+i)^{4i-2},(m-i+1)^{4i-2}\},
    \end{align*}
    if $k>3$.
\end{enumerate}
\begin{enumerate}
\item[\textup{(2)}] For $k=2m-1\ge 1$, if $$\mathfrak{g}_{k,r}=\mathfrak{p}^A\frac{k|\overbrace{4|\cdots|4}^r|2}{k+2|\underbrace{4|\cdots|4}_{r}},$$ then $\mathfrak{g}_{k,r}$ has the unimodal spectrum property with spectrum equal to $\{-1^{2r+1},0^{6r+2},1^{6r+2},2^{2r+1}\}$ if $k=1$, $\{-2,-1^{2r+3},0^{6r+5},1^{6r+5},2^{2r+3},3\}$ if $k=3$, and
\begin{align*}
        \{-m,(-m+1)^3&,-1^{2(k+r)-4},0^{2(k+3r)-1},1^{2(k+3r)-1},2^{2(k+r)-4},m^3,m+1\}\\
        &\cup\bigcup_{i=2}^{m-2}\{(-m+i)^{4i-2},(m-i+1)^{4i-2}\},
    \end{align*}
    if $k>3$.
\end{enumerate}
\end{theorem}

\begin{remark}\label{rem:nlcc2}
    Utilizing Theorem~\ref{thm:base2ext}, we can construct examples of Frobenius, type-A seaweeds which do not have the log-concave spectrum property. For example, let $\mathfrak{g}=\mathfrak{p}^A\frac{5|2}{7}$ and $\mathfrak{g}'=\mathfrak{p}^A\frac{5|4|4|2}{7|4|4}.$ Recall from Theorem~\ref{thm:k2} that the spectrum of $\mathfrak{g}$ is $\{-3,-2^3,-1^6,0^9,1^9,2^6,3^3,4\}$. On the other hand, by Theorem~\ref{thm:addm4s}, we find that the spectrum of $\mathfrak{g}'$ is $\{-3,-2^3,-1^{10},0^{21},1^{21},2^{10},3^3,4\}.$ Clearly, $\mathfrak{g}'$ does not have the log-concave spectrum property.
\end{remark}

\begin{remark}
Note that for fixed $k$ and varying values of $r$, the collection of distinct eigenvalues in the spectra of the Frobenius, type-A seaweeds considered in Theorem~\ref{thm:base2ext} is fixed.
\end{remark}

\begin{remark}
    Combining Theorems~\ref{thm:addm4s} and~\ref{thm:base2ext} and Corollary~\ref{cor:444} with Corollaries~\ref{cor:vfes},~\ref{cor:hfes}, and~\ref{cor:vhfes}, we obtain similar results for related families of Frobenius, type-A seaweeds.
\end{remark}

Considering the results found above, as well as some experimental evidence, we are naturally led to the following conjectures.

\begin{conj}\label{conj:stab1}
Let $\mathfrak{g}=\mathfrak{p}^A\frac{a_1|\hdots|a_m|k}{b_1|\hdots|b_t}$, for $k\ge 1$, be a Frobenius, type-A seaweed with spectrum $S$. If
$$\mathfrak{g}'=\mathfrak{p}^A\frac{a_1|\hdots|a_m|\overbrace{2k|\cdots|2k}^r}{b_1|\hdots|b_t|\underbrace{2k|\cdots|2k}_{r-1}|k}\quad\text{or}\quad\mathfrak{g}'=\mathfrak{p}^A\frac{a_1|\hdots|a_m|\overbrace{2k|\cdots|2k}^r|k}{b_1|\hdots|b_t|\underbrace{2k|\cdots|2k}_{r}},$$ for $r\ge 1$, then $\mathfrak{g}'$ is a Frobenius, type-A seaweed with spectrum $S\cup S'$, where $S'$ is a multiset consisting of values contained in $S$. Moreover, if $\mathfrak{g}$ has the unimodal spectrum property, then so does $\mathfrak{g}'$.
\end{conj}

\begin{conj}\label{conj:stab2}
If $$\mathfrak{g}=\mathfrak{p}^A\frac{\overbrace{2k|\cdots|2k}^r|1}{2kr+1},$$ for $k,r\ge 1$, then $\mathfrak{g}$ is Frobenius, and the set of distinct eigenvalues contained in the spectrum of $\mathfrak{g}$ is equal to the collection of integers contained in the interval
$$\begin{cases}
    [-2k+1,2k], & \text{if}~r~\text{is odd}; \\
    [-k,k+1], & \text{if}~r~\text{is even}.
\end{cases}$$
Moreover, $\mathfrak{g}$ has the unimodal spectrum property.
\end{conj}

\begin{remark}
The family of type-A seaweeds considered in Conjecture~\ref{conj:stab2} provide us with yet another example of a Frobenius, type-A seaweed which does not have the log-concave spectrum property. One can compute that the spectrum of $\mathfrak{g}=\mathfrak{p}^A\frac{8|8|8|1}{25}$ is equal to $$\{-7^2,-6^5,-5^8,-4^{13},-3^{23},-2^{37},-1^{52},0^{64},1^{64},2^{52},3^{37},4^{23},5^{13},6^8,7^5,8^2\}.$$ Notice that $8^2=64<65=5*13,$ so $\mathfrak{g}$ does not have the log-concave spectrum property.
\end{remark}

\begin{conj}\label{conj:stab3}
If $$\mathfrak{g}_{k,r}=\mathfrak{p}^A\frac{\overbrace{2k|\cdots|2k}^r|1}{1|\underbrace{2k|\cdots|2k}_r},$$ for $k,r\ge 1$, then the set of distinct eigenvalues contained in the spectrum of $\mathfrak{g}_{k,r}$ is equal to the collection of integers contained in the interval $[-k+1,k]$, and $\mathfrak{g}_{k,r}$ has the log-concave spectrum property. Moreover, if $i\in[-k+1,0]$, then the multiplicity of $i$ in the spectrum of $\mathfrak{g}_{k,r}$ is equal to the multiplicity of $i-1$ in the spectrum of $\mathfrak{g}_{k+1,r}$; similarly, if $i\in(0,k]$, then the multiplicity of $i$ in the spectrum of $\mathfrak{g}_{k,r}$ is equal to the multiplicity of $i+1$ in the spectrum of $\mathfrak{g}_{k+1,r}$.
\end{conj}

\section{Epilogue}\label{sec:epilogue}
The inductive proof methods outlined in this article correspond naturally to a set of ``winding moves" on the meander of a Frobenius, type-A seaweed. Such winding moves were first introduced by Panyushev (\textbf{\cite{Pan2001}}, 2001) -- and later recast graph-theoretically by Coll et al. (\textbf{\cite{meanders2}}, 2012) -- for the purpose of simplifying computations involving seaweeds. The initial, overall goal of our study was to prove Conjecture~\ref{conj:uni} by tracking the effects of the winding moves on the spectra of Frobenius, type-A seaweeds via (extended) spectrum matrices. It quickly became clear, however, that the spectrum of a generic Frobenius, type-A seaweed grew increasingly wild upon iterative applications of winding moves; e.g., compare Theorems~\ref{thm:k2} and \ref{thm:k+2k}. Thus, we still have the following general question.
\begin{que}\label{que:1}
\textit{How do the winding moves on the meander of a Frobenius, type-A seaweed $\mathfrak{g}$ affect the spectrum of $\mathfrak{g}?$}
\end{que}

\noindent Note that upon successfully answering this question, one should obtain a combinatorial proof of Conjecture~\ref{conj:uni}.

While the results of Section~\ref{sec:mp} provide significant insight into answering Question~\ref{que:1}, such results have led to further questions. In particular, to the authors' knowledge, this article marks the first acknowledgement of the log-concave spectrum property for seaweeds. While not a property of Frobenius, type-A seaweeds in general, per Section~\ref{sec:stable}, the results of Section~\ref{sec:mp} suggest that all Frobenius, maximal parabolic, type-A seaweeds do possess the log-concave spectrum property. Thus, we have the following conjecture.
\begin{conj}
If $\mathfrak{g}$ is a Frobenius, maximal parabolic, type-A seaweed, then $\mathfrak{g}$ has the log-concave spectrum property.
\end{conj}

\noindent The conjecture above is certainly not comprehensive. For example, the seaweed $\mathfrak{p}^A\frac{1|2|2|4|4|2}{2|2|2|1|4|4}$ is not parabolic, yet it possesses the log-concave spectrum property. We pose the following question.

\begin{que}\label{que:2}
Which Frobenius, type-A seaweeds admit the log-concave spectrum property?
\end{que}

 In contrast to Section~\ref{sec:mp}, in which the focus was primarily on applications of winding moves that increase the values of the parts in the defining compositions of the seaweeds while leaving the number of parts fixed, the motivation for Section~\ref{sec:stable} was to increase the number of parts without altering the size of each. In addition to providing a construction procedure for Frobenius, type-A seaweeds without the log-concave spectrum property, applications of such winding moves preserve the set of distinct eigenvalues in the spectra of the associated algebras. To the author's knowledge, this is the first instance of such behavior being noted, and it leads to the following question.

\begin{que}
Are the methods discussed in Section~\ref{sec:stable} the only means by which to construct a family of Frobenius, type-A seaweeds with a fixed set of distinct eigenvalues.
\end{que}

Recall that, in Section~\ref{sec:mp}, our proof methods required the determination of formulas for the extended spectrum of each Frobenius, type-A seaweed. While it follows from its definition that the extended spectrum of 
a Frobenius, type-A seaweed consists entirely of integers centered at 0, it is unclear whether these integers are generally unbroken and whether their multiplicities generally form a log-concave (resp., unimodal) sequence. Thus, we ask
\begin{que}\label{que:3}
\textit{Is the set of values in the extended spectrum of a Frobenius type-A seaweed unbroken?}
\end{que}
\noindent and
\begin{que}\label{que:4}
\textit{If the values in the extended spectrum of a Frobenius, type-A seaweed $\mathfrak{g}$ are written in increasing order, is the corresponding sequence of multiplicities log-concave \textup(resp., unimodal\textup)?}
\end{que}

\noindent While an answer to Question~\ref{que:3} would be an interesting result in its own right, consideration of Question~\ref{que:4} may also provide insight into an answer to Question~\ref{que:1}. Part of the challenge with using meanders when attempting to answer Question~\ref{que:1} is that it can be difficult to discern which paths in a given meander contribute their weights to the spectrum of the associated algebra. However, this particular difficulty is resolved when computing the extended spectrum of a Frobenius, type-A seaweed $\mathfrak{g}$, since the weights of all paths in $\overrightarrow{M}(\mathfrak{g})$ are counted. Thus, it is plausible that a purely graph-theoretic approach may yield answers to Question~\ref{que:4}, which, in turn, may provide an avenue toward a proof (or counterexample) of Conjecture~\ref{conj:uni}.

\end{document}